\documentclass[reqno]{amsart}

\usepackage[english]{babel}
\usepackage{amsmath}
\usepackage{amsthm}
\usepackage{amsfonts}
\usepackage{amssymb}
\usepackage{anysize}
\usepackage{hyperref}
\usepackage{appendix}
\usepackage{mathtools}
\usepackage{graphicx}
\usepackage{pgfplots}
\usepackage{subcaption} 
\usepackage{forest} 
\usepackage{enumitem}
\usepackage{dsfont}
\usepackage{cancel}

\forestset{
  colour my roots/.style={
    before typesetting nodes={
      edge=#1,
      for ancestors={
        edge=#1,
        #1,
      },
      #1,
    }
  }
}%Add "colour my roots=<color>" in the corresponding vertex/leaf to colour the corresponding path towards the root

\makeatletter
\@namedef{subjclassname@2020}{\textup{2020} Mathematics Subject Classification}
\makeatother

\newtheorem{definition}{Definition}[section]
\newtheorem{lemma}[definition]{Lemma}
\newtheorem{theorem}[definition]{Theorem}
\newtheorem{corollary}[definition]{Corollary}
\newtheorem{proposition}[definition]{Proposition}
\newtheorem{example}[definition]{Example}
\newtheorem{remark}[definition]{Remark}

\DeclareMathOperator{\Tree}{\mathsf{T}}

\DeclareMathOperator{\Words}{\mathsf{W}}
\DeclareMathOperator{\Leaves}{\mathsf{L}}
\DeclareMathOperator{\Level}{\mathsf{L}}
\DeclareMathOperator{\child}{\mathsf{c}}
\DeclareMathOperator{\Parents}{\mathsf{P}}
\DeclareMathOperator{\mate}{\mathsf{m}}
\DeclareMathOperator{\blambda}{\boldsymbol{\lambda}}
\DeclareMathOperator{\bsigma}{\boldsymbol{\sigma}}
\DeclareMathOperator{\bF}{\boldsymbol{F}}
\DeclareMathOperator{\bk}{\boldsymbol{k}}
\DeclareMathOperator{\br}{\boldsymbol{r}}
\DeclareMathOperator{\bx}{\mathbf{x}}
\DeclareMathOperator{\by}{\mathbf{y}}
\DeclareMathOperator{\bz}{\mathbf{z}}

\DeclareMathOperator{\bY}{\mathbf{Y}}

\numberwithin{equation}{section}

\title[Ergodicity of the Fisher infinitesimal model with quadratic selection]{Ergodicity of the Fisher infinitesimal model with quadratic selection}

\author[Vincent Calvez]{Vincent Calvez}
\address{Institut Camille Jordan (ICJ), UMR 5208 CNRS \& Universit\'e Claude Bernard Lyon 1, 69100 Villeurbanne, France}
\email{vincent.calvez@math.cnrs.fr}

\author[Thomas Lepoutre]{Thomas Lepoutre}
\address{Univ Lyon, Inria, Universit\'e Claude Bernard Lyon 1, CNRS UMR5208, Institut Camille Jordan, F-69603 Villeurbanne, France}
\email{thomas.lepoutre@inria.fr}

\author[David Poyato]{David Poyato}
\address{Departamento de Matem\'atica Aplicada and Research Unit ``Modeling Nature'' (MNat), Facultad de Ciencias, Universidad de Granada, 18071 Granada, Spain}
\email{davidpoyato@ugr.es}

\begin{document}

\date{\today}


\thanks{\textbf{Acknowledgment.}  VC and DP have received funding from the European Research Council (ERC) under the European Union’s Horizon 2020 research and innovation program (grant agreement No 865711). DP has also received founding from the European Union’s Horizon Europe research and innovation program under the Marie Sk\l odowska-Curie grant agreement No 101064402, and partially from the State Research Agency (SRA) of the Spanish Ministry of Science and Innovation and European Regional Development Fund (ERDF), project PID2022-137228OB-I00, and by Modeling Nature Research Unit, project QUAL21-011.}

\begin{abstract}
We study the convergence towards a unique equilibrium distribution of the solutions to a time-discrete model with non-overlapping generations arising in quantitative genetics. The model describes the dynamics of a phenotypic distribution with respect to a multi-dimensional trait, which is shaped by selection and Fisher's infinitesimal model of sexual reproduction. We extend some previous works devoted to the time-continuous analogues, that followed a  perturbative approach in the regime of weak selection, by exploiting the contractivity of the infinitesimal model operator in the Wasserstein metric. Here, we tackle the case of quadratic selection by a global approach. We establish uniqueness of the equilibrium distribution and exponential convergence of the renormalized profile. Our technique relies on an accurate control of the propagation of information across the large binary trees of ancestors (the pedigree chart), and reveals an ergodicity property, meaning that the shape of the initial datum is quickly forgotten across generations. We combine this information with appropriate estimates for the emergence of Gaussian tails and propagation of quadratic and exponential moments to derive quantitative convergence rates. Our result can be interpreted as a generalization of the Krein-Rutman theorem in a genuinely non-linear, and non-monotone setting. 
\end{abstract}

\maketitle

\section{Introduction}

{\textit{Fisher's infinitesimal model} (also known as the \textit{polygenic model}) is a widely used statistical model in quantitative genetics initially proposed by {\sc R. Fisher} \cite{F-18}. It assumes that the genetic component of a quantitative phenotypical trait is affected by an infinite number of loci with infinitesimal and additive allelic effects and claims that the genetic component of descendants' traits is normally distributed around the mean value of parents' traits, with a constant (genetic) variance across generations. This model allowed reconcilling Mendelian inheritance and the continuous trait variations documented by {\sc F. Galton} via a Central Limit Theorem. More specifically, by taking limits when the number of underlying loci tends to infinity on a model with Mendelian inheritance, {\sc N. Barton}, {\sc A. Etheridge} and {\sc A. V\'eber} \cite{BEV-17} recently proved rigorously the validity of Fisher's infinitesimal model under various evolutionary processes ({\em e.g.}, natural selection). 

In this paper we study a time-discrete evolution problem for the distribution of a phenotypical trait $x\in \mathbb{R}^d$ in a population undergoing sexual reproduction and the effect of natural selection. Specifically, starting at any initial configuration $F_0\in \mathcal{M}_+(\mathbb{R}^d)$ of trait distribution, we analyze the long term dynamics of trait distributions $\{F_n\}_{n\in \mathbb{N}}$ across successive generations $n\in \mathbb{N}$, which solves the following recursion
\begin{equation}\label{E-sexual-reproduction-time-discrete}
F_n=\mathcal{T}[F_{n-1}],
\end{equation}
for any $n\in \mathbb{N}$. The operator $\mathcal{T}$ encodes the balanced effect of sexual reproduction in the population (under the infinitesimal model) and natural selection. Specifically, $\mathcal{T}$ is defined by
\begin{equation}\label{E-sexual-reproduction-T-operator}
\mathcal{T}[F]:=e^{-m}\mathcal{B}[F],
\end{equation}
for any trait distribution $F\in \mathcal{M}_+(\mathbb{R}^d)$. On the one hand, $m=m(x)\geq 0$ is called the {\it selection function} and represents the mortality effect of a trait-dependent natural selection on the population, so that $e^{-m}$ stands for the survival probability in the next generation. On the other hand, the operator $\mathcal{B}$ is chosen to be {\it Fisher's infinitesimal operator} and it takes the form
\begin{equation}\label{E-sexual-reproduction-operator}
\mathcal{B}[F](x):=\int_{\mathbb{R}^{2d}}G\left(x-\frac{x_1+x_2}{2}\right)\frac{F(x_1)F(x_2)}{\int_{\mathbb{R}^d}F(x')\,dx'}\,dx_1\,dx_2,\quad x\in \mathbb{R}^d,
\end{equation}
for any trait distribution $F\in \mathcal{M}_+(\mathbb{R}^d)$. Here $G=G(x)$ is a probability density (the {\it mixing kernel}) and the factor $G(x-\frac{x_1+x_2}{2})$ represents the transition probability that two given individuals with trait values $x_1,x_2\in \mathbb{R}^d$ will mate and yield a descendant with trait value $x\in \mathbb{R}^d$. In other words the resulting trait distributes around the mean value $\frac{x_1+x_2}{2}$ of parents' trait with law $G$. By definition, $\mathcal{B}[F]$ quantifies the number of births after all possible matches of any couple of individuals according to the trait distribution $F$. Altogether, $\mathcal{T}[F]$ quantifies the amount of offspring of a population distributed according to $F$ having resisted the effect of selection.

The above sexual reproduction operator $\mathcal{B}$ has recently pulled the attention of both the applied and more theoretical communities, {\em cf.} \cite{BEV-17,CGP-19, MR-13, P-20-arxiv, R-17-arxiv, R-21-arxiv}. In this paper we shall restrict to Gaussian mixing kernel and quadratic selection function, {\em i.e.},
\begin{align}
G(x)&:=\frac{1}{(2\pi)^{d/2}}e^{-\frac{\vert x\vert^2}{2}},\quad x\in \mathbb{R}^d,\label{E-mixing-Gaussian}\\
m(x)&:=\frac{\alpha}{2}\vert x\vert^2, \hspace{45pt} x\in \mathbb{R}^d,\label{E-selection-quadratic}
\end{align}
where $\alpha\in \mathbb{R}_+$ is a fixed parameter. Thereby, the trait of offsprings is normally distributed around the mean value of the trait of parents by assumption \eqref{E-mixing-Gaussian}, thus reducing to the standard infinitesimal model when the assumptions of the Central Limit Theorem are met \cite{BEV-17}. For simplicity, we set a Gaussian $G$ with unit variance, but any value of the genetic variance could also be considered (see nondimensionalization in Appendix \ref{Appendix-nondimensionalization}). 

Before introducing our results, we shall relate the previous time-discrete problem \eqref{E-sexual-reproduction-time-discrete} to analogous time-continuous quantitative genetics models of evolutionary dynamics that have been studied in the literature. Meanwhile, we will anticipate the major difficulties that can be faced when analyzing the long-time dynamics of \eqref{E-sexual-reproduction-time-discrete}. To this end, we consider, the following type of integro-differential equations for the evolutionary dynamics of a trait distribution $f(t,x)$:
\begin{equation}\label{E-general-reproduction}
\left\{\begin{array}{ll}
\displaystyle\partial_t f=-m(x)f+\mathcal{R}[f], & t\geq 0,\,x\in \mathbb{R}^d,\\
\displaystyle f(0,x)=f_0(x), & x\in \mathbb{R}^d,
\end{array}\right.
\end{equation}
where, $m=m(x)$ is again the mortality rate and $\mathcal{R}=\mathcal{R}[f]$ is the reproduction operator. Hence, $\mathcal{R}[f](x)$ determines the amount of births with trait value $x\in \mathbb{R}^d$ per unit of time. As for \eqref{E-sexual-reproduction-time-discrete}, the resulting dynamics of the population becomes a consequence of a balance between selection encoded by the trait-dependent mortality and the diversity generated by the growth term $\mathcal{R}$ across generations.

Many studies consider a {\it linear} reproduction operator $\mathcal{R}$, associated with a probability density $K=K(x)$ characterizing the mutational effects at birth, of the form
\begin{equation}\label{E-asexual-reproduction-operator}
\mathcal{R}[f](x):=\int_{\mathbb{R}^d}K(x-y)f(t,y)\,dy,\quad x\in \mathbb{R}^d.
\end{equation}
The factor $K(x-y)$ determines the probability that an individual with trait value $y\in \mathbb{R}^d$ produces a descendant with trait value $x\in \mathbb{R}^d$ (possibly deviating from $y$). This class of linear reproduction operators is well-suited for an asexual mode of reproduction. This includes parabolic equations in the limit of small variance of $K$. In particular, we refer to a series of works about the long-time asymptotics in the regime of small variance, initiated in \cite{BM-07, BMP-09, DJMP-05}, including an additional density-dependent competition term that makes the analysis non standard (see {\em e.g.} \cite{CL-20} and references therein for the well-posedness of the constrained Hamilton-Jacobi equation derived in the limit).

Recently, inspired by the infinitesimal model, the case of {\it sexual reproduction} has been addressed by invoking the preceding nonlinear version $\mathcal{B}$ in \eqref{E-sexual-reproduction-operator} as reproduction operator $\mathcal{R}=\mathcal{B}$. Several asymptotic regimes have been addressed: large reproduction rate \cite{MR-13, R-17-arxiv, R-21-arxiv}, small variance asymptotics \cite{CGP-19, P-20-arxiv}. In the latter case, the limiting problem keeps the non-local nature of the problem, being of a finite-difference type, rather than a Hamilton-Jacobi PDE. Before we continue the discussion about the state-of-the-art, let us emphasize that there is no restriction on the parameter $\alpha$ in the present work.

Since both $f\mapsto m f$ and $f\mapsto \mathcal{B}[f]$ are $1$-homogeneous operators, we can seek for special steady solutions of \eqref{E-sexual-reproduction-operator} through the following ansatz:
\begin{equation}\label{E-ansatz}
f(t,x)=e^{\lambda t}F(x),\quad (t,x)\in\mathbb{R}_+\times \mathbb{R}^d,
\end{equation}
The parameter $\lambda\in \mathbb{R}$ represents the rate at which the number of individuals grows (if $\lambda\geq 0$) or decreases (if $\lambda\leq 0$), and $F=F(x)\geq 0$ is an unknown probability density. By imposing such an ansatz on \eqref{E-general-reproduction}-\eqref{E-sexual-reproduction-operator}, the following generalized spectral problem arises for the couple $(\lambda,F)$:
\begin{equation}\label{E-eigenproblem-time-continuous}
\left\{\begin{array}{l}
\lambda F(x)+m(x)F(x)=\mathcal{B}[F](x),\quad x\in \mathbb{R}^d,\\
\int_{\mathbb{R}^d}F(x')\,dx'=1.
\end{array}\right.
\end{equation}
Note that the operator $\mathcal{B}$ is genuinely non-linear so that methods based on the Krein-Rutman theory or maximum principles cannot be applied straightforwardly, see \cite{BCV-16, CLW-17} and the references therein for the linear case. Further, usual extensions of the Krein-Rutman theory to $1$-homogeneous operators \cite{M-07} cannot be applied neither because $\mathcal{B}$ is not monotone. To date, the main strategy behind the existence of solutions of \eqref{E-eigenproblem-time-continuous} relies on a suitable application of Schauder fixed-point theorem to the operator $F\mapsto (\lambda+m)^{-1}\mathcal{B}[F]$ over an appropriate cone of $L^1(\mathbb{R}^d)$ that is conserved by the nonlinear operator, see \cite{BCGL-17-arxiv}. However, uniqueness cannot be achieved by this method. Moreover, it has been proven in \cite[Corollary 1.5]{CGP-19} that several equilibrium states $(\lambda,F)$ can co-exist in the presence of multiple local minima of $m$ (provided the variance of $G$ is small enough). That is, the generalized eigenproblem \eqref{E-eigenproblem-time-continuous} does not admit a unique positive eigenfunction, in contrast with general conclusions of the Krein-Rutman theory. 

Recently, {\sc G. Raoul} addressed the long-time dynamics of \eqref{E-general-reproduction} in 1D, with $\mathcal{R}=\mathcal{B}$ and a trait dependent fecundity rate \cite{R-21-arxiv}. He obtained local uniqueness and exponential relaxation under the assumption of weak and localized (compactly supported) selection effects. To this end, he controlled locally in space the Wasserstein distance between the solution and the stationary state, using the uniform contraction property of $\mathcal{B}$ in the space of probability measures sharing the same center of mass. Unfortunately, the Wasserstein metric is not fully compatible with multiplicative operators, such as trait-dependent fecundity (see the discussion in \cite[Section 3.4]{R-21-arxiv} and Section \ref{S-observations} below). This can be circumvented under the additional assumption that the trait density is locally uniformly bounded below, following \cite{BGG-12}. Obviously, this cannot hold globally for integrable densities, hence {\sc G. Raoul} developed estimates of the distribution's tail to complete the contraction estimates. Also, a lot of attention has to be paid to the dynamics of the center of the distribution which is essentially driven by selection. Indeed, in the case of flat selection ($m\equiv0$), the problem is invariant by translation, so that local uniqueness cannot hold. To conclude this discussion, let us emphasize that global uniqueness and the  asymptotic behavior of generic solutions to the evolution problem \eqref{E-general-reproduction}-\eqref{E-sexual-reproduction-operator} is still open. Below, we provide a first result in this direction, for the time-discrete problem \eqref{E-sexual-reproduction-time-discrete}, though.

We remark that the time-discrete version \eqref{E-sexual-reproduction-time-discrete} that we propose in this paper can be partially regarded as a discretization in time of the above time-continuous problem \eqref{E-general-reproduction}-\eqref{E-sexual-reproduction-operator}, with non-overlapping generations (see Appendix \ref{Appendix-nondimensionalization} for further details). As for the time-continuous problem, we could seek special solutions to the time-discrete problem \eqref{E-sexual-reproduction-time-discrete} in the following form
\begin{equation}\label{E-ansatz-discrete}
F_n(x)=\lambda^n F(x),\quad (n,x)\in \mathbb{N}\times\mathbb{R}^d.
\end{equation}
Again, $\lambda\in \mathbb{R}_+^*$ is the rate of growth (if $\lambda\geq 1$) or decrease (if $\lambda\leq 1$) of individuals, and $F=F(x)$ is an unknown probability density. This yields the following generalized eigenproblem for the couple $(\lambda,F)$:
\begin{equation}\label{E-eigenproblem-time-discrete}
\left\{\begin{array}{l}
\lambda F(x)=\mathcal{T}[F](x),\quad x\in \mathbb{R}^d,\\
\int_{\mathbb{R}^d}F(x')\,dx'=1.
\end{array}\right.
\end{equation}
In this paper we aim to address the following questions:
\begin{enumerate}[label={\bf(Q\arabic*)}]
\item Does the eigenproblem \eqref{E-eigenproblem-time-discrete} have a unique solution $(\blambda_\alpha,\bF_\alpha)$, with $\blambda_\alpha\in\mathbb{R}$ and $\bF_\alpha$ being a probability measure, for each $\alpha\in \mathbb{R}_+^*$?
\item Consider any generic initial datum $F_0\in \mathcal{M}_+(\mathbb{R}^d)$ and its associated solution $\{F_n\}_{n\in \mathbb{N}}$ of the time-discrete problem \eqref{E-sexual-reproduction-time-discrete}. Do the renormalized profiles $F_n/\Vert F_n\Vert_{L^1(\mathbb{R}^d)}$ converge to the unique steady profile $\bF_\alpha$ solving \eqref{E-eigenproblem-time-discrete} when $n\rightarrow \infty$?
\end{enumerate}

We shall prove that the answer to both questions is affirmative. It stands to reason that similar existence and local uniqueness results like in \cite{BCGL-17-arxiv,CGP-19} could be extended to the new eigenproblem \eqref{E-eigenproblem-time-discrete} by applying Schauder and Banach fixed point theorems. Nevertheless, in this paper we introduce a novel method that unravels an ergodicity property of the operator $\mathcal{T}$, leading to quantitative estimates for the relaxation of profiles $\{F_n\}_{n\in \mathbb{N}}$ towards $\bF_\alpha$. First, we prove that an explicit (Gaussian) solution to the eigenproblem \eqref{E-eigenproblem-time-discrete} exists. Second, by computing $n$ iterations of the operator $\mathcal{T}$, that is $F_n=\mathcal{T}^n[F_0]$, we notice that information of $F_n$ at the trait value $x$ is propagated from the initial datum $F_0$ across $2^n$ ancestors over a binary tree with height $n$ and rooted at $x$ (the pedigree chart). Interestingly, an appropriate reformulation of $\mathcal{T}^n$ in the case of Gaussian mixing \eqref{E-mixing-Gaussian} and quadratic selection \eqref{E-selection-quadratic} shows that the dependence of the solution $\{F_n\}_{n\in \mathbb{N}}$ on the initial datum $F_0$ is rapidly lost across the different levels of the tree. More specifically, a strong convergence of generic solutions $\{F_n\}$ of the time-discrete problem \eqref{E-sexual-reproduction-time-discrete} towards the steady Gaussian profile $\bF_\alpha$ solving \eqref{E-eigenproblem-time-discrete} is achieved locally with respect to $x$. Third, we prove an appropriate propagation of quadratic and exponential moments, leading to uniform tightness of the family $\{F_n\}_{n\in \mathbb{N}}$. Finally, we glue all the information together and conclude the final global convergence result in question {\bf (Q2)} in relative entropy. We refer to Section \ref{S-observations} for a more detailed sketch of our strategy of proof. Specifically, we obtain our main result:

\begin{theorem}\label{T-main}
Assume that $\alpha\in \mathbb{R}_+^*$ and set any initial non-negative measure $F_0\in \mathcal{M}_+(\mathbb{R}^d)$. The solution $\{F_n\}_{n\in \mathbb{N}}$ to the time-discrete problem \eqref{E-sexual-reproduction-time-discrete} verifies that the growth rates $\Vert F_n\Vert_{L^1(\mathbb{R}^d)}/\Vert F_{n-1}\Vert_{L^1(\mathbb{R}^d)}$ relax towards $\blambda_\alpha$ and the normalized profiles $F_n/\Vert F_n\Vert_{L^1(\mathbb{R}^d)}$ relax towards $\bF_\alpha$ with
\begin{equation}\label{E-sexual-reproduction-Gaussian-solution}
\blambda_\alpha:=\left(1+\alpha\left(1+\frac{\bsigma_\alpha^2}{2}\right)\right)^{-\frac{d}{2}},\quad \bF_\alpha:=G_{0,\bsigma_\alpha^2},
\end{equation}
and the variance $\bsigma_\alpha^2\in \mathbb{R}_+^*$ is the unique positive root of the equation
\begin{equation}\label{E-sexual-reproduction-Gaussian-solution-variance-equation}
\frac{1}{\bsigma_\alpha^2}=\alpha+\frac{1}{1+\frac{\bsigma_\alpha^2}{2}},\quad \mbox{i.e.},\quad \bsigma_\alpha^2=\frac{\sqrt{(1+2\alpha)^2+8\alpha}-(1+2\alpha)}{2\alpha}.
\end{equation}
Specifically, for any $\varepsilon\in \mathbb{R}_+^*$ there exists a sufficiently large $C_\varepsilon\in \mathbb{R}_+^*$ such that
\begin{align*}
\mathcal{D}_{\rm KL}\left(\left.\frac{F_n}{\Vert F_n\Vert_{L^1(\mathbb{R}^d)}}\right\Vert  \bF_\alpha\right)&\leq C_\varepsilon((2\bk_\alpha)^2+\varepsilon)^n,\\
\left\vert\frac{\Vert F_n\Vert_{L^1(\mathbb{R}^d)}}{\Vert F_{n-1}\Vert_{L^1(\mathbb{R}^d)}}-\blambda_\alpha\right\vert&\leq C_\varepsilon(2\bk_\alpha+\varepsilon)^n,
\end{align*}
for any $n\in \mathbb{N}$, where $\mathcal{D}_{\rm KL}$ is the Kullback-Leibler divergence (or relative entropy), {\it i.e.},
\begin{equation}\label{E-KL-divergence}
\mathcal{D}_{\rm KL}(P\Vert Q):=\int_{\mathbb{R}^d}P(x)\log\left(\frac{P(x)}{Q(x)}\right)\,dx,
\end{equation}
for any $P,Q\in L^1_+(\mathbb{R}^d)\cap \mathcal{P}(\mathbb{R}^d)$,
and the coefficient $\bk_\alpha\in (0,\frac{1}{2})$ reads
\begin{equation}\label{E-sexual-reproduction-coefficients-limit}
\bk_\alpha=\frac{\bsigma_\alpha^2}{2+\bsigma_\alpha^2},\quad \mbox{i.e.},\quad \frac{(3+2\alpha)-\sqrt{(1+2\alpha)^2+8\alpha}}{4}.
\end{equation}
\end{theorem}

Note that any solution $(\lambda,F)$ to the eigenproblem \eqref{E-eigenproblem-time-discrete} yields a solution to the time-discrete problem \eqref{E-sexual-reproduction-time-discrete} through the ansatz \eqref{E-ansatz-discrete}. Then, question {\bf (Q1)} will readily follow from question {\bf (Q2)} and, in particular, the unique solution $(\blambda_\alpha,\bF_\alpha)$ to \eqref{E-eigenproblem-time-discrete} is Gaussian.

\begin{corollary}\label{C-main}
Assume that $\alpha\in \mathbb{R}_+^*$, then there exists a unique solution $(\blambda_\alpha,\bF_\alpha)$ to the eigenproblem \eqref{E-eigenproblem-time-discrete}, with $\bF_\alpha$ being a  probability measure, namely given by \eqref{E-sexual-reproduction-Gaussian-solution}.
\end{corollary}

\begin{remark}[About the assumption of quadratic selection]
By the assumption of a quadratic selection function, we can henceforth push extensively explicit computations of the iterated operator. The latter consists in recursive multiplication and convolution by Gaussian functions. This enables capturing the essence of the relaxation phenomenon, and having a precise description of the behaviour at infinity, which crucially helps to localize the convergence argument. As a by-product, we are able to consider very general initial condition $F_0$. Two of the authors, together with F. Santambrogio, have obtained similar results when the selection function $m$ is more generally assumed to be strongly convex \cite{CPS-23-arxiv}. However, they imposed stringent conditions on the initial datum $F_0$, that is, it should behave at infinity as the equilibrium profile $\mathbf{F}_\alpha$ in a very strong sense. It would be of interest to merge the two works, having a general (strongly convex) selection function, and a general initial datum. This is left for future work.
\end{remark}

\begin{remark}[About the choice of metric]\label{R-choice-metric} The convergence of the profiles has been quantified in Kullback-Leibler divergence \eqref{E-KL-divergence} in Theorem \ref{T-main}, in contrast with the results in \cite{R-17-arxiv,R-21-arxiv}, where the quadratic Wasserstein distance was used for perturbative regimes of the case without selection. We anticipate that there are several compelling reasons for such a change of metric in our non-perturbative setting:

\begin{enumerate}[label=(\roman*)]
\item {\bf (Quadratic Wasserstein distance)} When $\alpha=0$, the operator $\mathcal{T}$ reduces to $\mathcal{B}$, which is non-expansive with respect to the quadratic Wasserstein distance, and indeed contractive over distributions with common center of mass, {\it cf.} Section \ref{SS-properties-B}. However, when $\alpha>0$, the multiplicative operator leads to an operator $\mathcal{T}$ which is not even Lipschitz continuous with respect to the quadratic Wasserstein distance, {\it cf} Section \ref{SS-incompatibility-wasserstein-multiplicative}. Therefore, the quadratic Wasserstein distance seems to be unadapted to scenarios where reproduction and selection operate together.

\item {\bf (Log-Lipschitz norm)} As we show in Section \ref{SS-log-lipschitz-contraction}, the operator $\mathcal{T}$ is non-expansive in the log-Lipschits norm $\Vert \nabla\log\frac{F}{\bF_\alpha}\Vert_{L^\infty(\mathbb{R}^d)}$ for all $\alpha\geq 0$, and indeed it is contractive if $\alpha>0$. However, this contraction only gives actual information when the initial datum $F_0$ has identical tails to the Gaussian density $\bF_\alpha$ so that initially the log-Lipschitz norm is finite.

\item {\bf (Relative entropy distance)} Note that the above log-Lipschit norm amounts to the natural $L^\infty$ version of the relative Fisher information $\Vert \nabla\log\frac{F}{\bF_\alpha}\Vert_{L^2(\mathbb{R}^d,F)}$. By the log-Soboled inequality, contraction of the log-Lipschits norm readily implies decay of the Kullback-Leibler divergence, which is a more standard metric in relative entropy arguments. However, we emphasize that our use of the Kullback-Leibler divergence is not only aesthetic, but we actually need it in order to go beyond the above structural constraint on the tails of the initial data. Specifically, for generic initial data we need to prove an appropriate shaping of tails over time, which cannot be expressed using uniform norms, but only in an averaged sense (compatible with the relative entropy).
\end{enumerate}

Altogether justifies that the whilst the quadratic Wasserstein distance is useful in perturbative regimes, the use of alternative norms is necessary to quantify contraction in purely non-perturbative settings.

\end{remark}

\begin{remark}[About the positivity of $\alpha$]\label{R-positivity-alpha}
Our form of ergodicity, as measured in Theorem \ref{T-main}, it breaks down when $\alpha=0$, simply because the operator is invariant by translation in that case, and it admits a one-parameter family $\bF_{\alpha=0}(\cdot-\mu)$ with $\mu\in \mathbb{R}^d$ of fixed Gaussian probability densities. Nevertheless, as
mentioned in item $(i)$ above, when $\alpha=0$ and we restrict to centered initial data, there is convergence to the right centered Gaussian probability density $\bF_{\alpha=0}$ with respect to the Wasserstein distance. It is an open problem to make both approaches meet for $\alpha=0$, and prove contraction in norms stronger than the Wasserstein distance, but comparable to the log-Lipschitz norm, in this subclass of initial data.
\end{remark}

The rest of the paper is organized as follows. In Section \ref{S-observations} we discuss about the generic incompatibility of the Wasserstein distance with a multiplicative operator and we provide a brief outlook of the strategy of our proof. In Section \ref{S-preliminaries} we provide some necessary notation and we introduce the special class of Gaussian solutions of both problems \eqref{E-sexual-reproduction-time-discrete} and \eqref{E-eigenproblem-time-discrete}, which will inspire some parts of the paper. Section \ref{S-properties-operator-T} is devoted to introduce some main properties of $\mathcal{T}$ regarding the emergence of Gaussian behavior (in the large) from generic initial data, and a suitable propagation of quadratic and exponential moments across generations. In Section \ref{S-restating-iterations} we reformulate \eqref{E-sexual-reproduction-time-discrete} via a high-dimensional integral operator propagating ancestors' information across the different levels of the pedigree chart, which will be the cornerstone to study the long-term dynamics. In Section \ref{S-ergodicity} we prove our main results, namely, Theorem \ref{T-main} and Corollary \ref{C-main}. Section \ref{S-numerical-experiments} contains some numerical experiments that illustrate the results in this paper. In Section \ref{S-conclusions} we provide some conclusions and perspectives. Finally, Appendix \ref{Appendix-nondimensionalization} contains the adimensionalization of the problem, and the relationship between \eqref{E-sexual-reproduction-time-discrete} and the previous time-continuous analogues in the literature is discussed. A full list of the main notations in the paper is presented in Table \ref{tab:list-notations}.

\begin{table}[t]
\centering
\begin{tabular}{||l|l|l||}
\hline
{\bf Notation} & {\bf Meaning} & {\bf Page}\\
\hline
$\mathcal{M}(\mathbb{R}^d)$, $\mathcal{M}_+(\mathbb{R}^d)$ & Signed and non-negative finite Radon measures & \pageref{T-main}\\
\hline
$\mathcal{P}(\mathbb{R}^d)$, $\mathcal{P}_2(\mathbb{R}^d)$ & Probability measures (with finite $2$nd order moment) & \pageref{lem:contraction W2}\\
\hline
$P\otimes Q\in \mathcal{M}(\mathbb{R}^{2d})$ & Tensor product of the finite measures $P,Q\in\mathcal{M}(\mathbb{R}^d)$ & \pageref{E-rate-growth-restatement-lambda}\\
\hline
$\mathcal{D}_{\rm KL}(P\Vert Q)$ & Kullback-Leibler divergence between $P,Q\in \mathcal{P}(\mathbb{R}^d)$ & \pageref{E-KL-divergence}\\
\hline
$\mathcal{W}_2(P,Q)$ & Quadratic Wasserstein distance between $P,Q\in \mathcal{P}_2(\mathbb{R}^d)$ & \pageref{E-W2-characterizations}\\
\hline
$G_{\mu,\sigma^2}$ & Gaussian with mean $\mu\in \mathbb{R}^d$ and covariance $\sigma^2I_d\in \mathbb{R}^{d\times d}$ & \pageref{SS-Gausian-solutions}\\
\hline
$G_{\mu,\Sigma}$ & Gaussian with mean $\mu\in \mathbb{R}^d$ and covariance $\Sigma\in \mathbb{R}^{d\times d}$ & \pageref{D-sexual-reproduction-multivariate-normal-Gn}\\
\hline
$\mathcal{N}(\mu,\Sigma)$ & Normal with mean $\mu\in \mathbb{R}^d$ and covariance $\Sigma\in \mathbb{R}^{d\times d}$ & \pageref{D-sexual-reproduction-multivariate-normal-Gn}\\
\hline
$\mathbb{E}[X]$ & Expectation of a random variable $X$ & \pageref{E-sexual-Gaussian-quadratic-emF}\\
\hline
$G(x)$ & Gaussian mixing kernel $G_{0,I_d}$ & \pageref{E-mixing-Gaussian}\\
\hline
$m(x)$ & Quadratic selection function $\frac{\alpha}{2}|x|^2$ & \pageref{E-selection-quadratic}\\
\hline
$\mathcal{B}[F]$ & Fisher's infinitesimal operator & \pageref{E-sexual-reproduction-operator}\\
\hline
$\mathcal{T}[F]$ & Selection-reproduction operator & \pageref{E-sexual-reproduction-T-operator}\\
\hline
$\mathcal{M}[F]$ & Normalized multiplicative operator by $e^{-m}$ & \pageref{D-operator-M}\\
\hline
$\mathcal{S}[F]$ & Scaled selection-reproduction operator & \pageref{D-normalized-operator-T}\\
\hline
$\mathcal{E}_n[F]$ & High-dimensional integral operators & \pageref{D-expectation-operator}\\
\hline
$\{F_n\}_{n\in \mathbb{N}}$ & Solution to the time-evolution equation \eqref{E-sexual-reproduction-time-discrete} & \pageref{E-sexual-reproduction-time-discrete}\\
\hline
$(\blambda_\alpha,\bF_\alpha)$ & Gaussian solution to the non-linear eigenproblem \eqref{E-eigenproblem-time-discrete} & \pageref{E-sexual-reproduction-Gaussian-solution}\\
\hline
$\bsigma_\alpha^2$ & Variance of Gaussian eigenfunction $\bF_\alpha$ & \pageref{E-sexual-reproduction-Gaussian-solution-variance-equation}\\
\hline
$\bar F=\frac{e^m F}{\bF_{\alpha=0}}$ & Normalized profile associated to $F\in \mathcal{M}_+(\mathbb{R}^d)$ & \pageref{D-rescaled-F}\\
\hline
$\Tree^n$, $\Tree^n_*$, $\widehat{\Tree}^n$ & Perfect binary tree, rootless tree and leafless tree & \pageref{SS-trees}\\
\hline
$\Leaves^n_m$, $\Leaves^n$ & Level $m$ and leaves (level $n$) of the tree & \pageref{SS-trees}\\
\hline
$i1,i2$ & Parents of a node $i\in \widehat{\Tree}^n$ of the tree & \pageref{SS-trees}\\
\hline
$\hspace{-1ex}\begin{array}{l}
\bx_n=(x_i)_{i\in \Tree^n_*}\\
\by_n=(y_i)_{i\in \Tree^n_*}
\end{array}$ & Variables indexed by the rootless tree & \pageref{R-tree-indexed-variables}\\
\hline
$\bz_n=(z_j)_{j\in \Leaves^n}$ & Variables indexed by leafs & \pageref{R-tree-indexed-variables}\\
\hline
$\Vert \bx\Vert=\left(\sum_{i=1}^n|x_i|^2\right)^{1/2}$ & $\ell_2$ sum of Euclidean norms of $\bx=(x_1,\ldots,x_n)\in \mathbb{R}^{dn}$ & \pageref{E-sexual-reproduction-barX-sum}\\
\hline
$\Phi_n^j(x;\by_n)$ & Lineage map from leaf $j\in \Leaves^n$ to root value $x\in \mathbb{R}^d$ & \pageref{D-sexual-reproduction-Gaussian-quadratic-lineage-maps}\\
\hline
$u\otimes v\in \mathbb{R}^{d\times d}$ & Kronecker product $(u_iv_j)_{1\leq i,j\leq N}$ of vectors $u,v\in \mathbb{R}^d$ & \pageref{R-correlations}\\
\hline
$k_n$, $\kappa_n$ & Sequences of coefficients in the change of variables & \pageref{D-sexual-reproduction-coefficients}\\
\hline
$(2\bk_\alpha)^2$ & Convergence rate of $\mathcal{D}_{\rm KL}$ & \pageref{E-sexual-reproduction-coefficients-limit}\\
\hline
$2\bk_\alpha$ & Convergence rate of the log-Lipschitz norm & \pageref{L-log-Lipschitz-estimate}\\
\hline
$\br_\alpha$ & Relaxation rates of variances recursion & \pageref{L-sexual-reproduction-Gaussian-solution-time-discrete-variance-relaxation}\\
\hline
\end{tabular}
\caption{List of notations}
\label{tab:list-notations}
\end{table}

\section{Motivation and strategy of the proof of Theorem \ref{T-main}}\label{S-observations}

In this section, we discuss the incompatibility of the quadratic Wasserstein distance to quantify directly contractivity under the joint effect of the reproduction operator $\mathcal{B}$ in \eqref{E-sexual-reproduction-operator} and a generic multiplicative selection $e^{-m}$. Although a small perturbation of the case of flat selection ($m\equiv 0$) could still be considered via a perturbative argument (see discussion above, \cite{R-21-arxiv} and also \cite{R-17-arxiv}), our novel approach is able to tackle a purely non-perturbative setting. We end the section by briefly discussing the strategy of our proof.

\subsection{Some properties of the sexual reproduction operator}\label{SS-properties-B}
We start by recalling some of the main properties of the sexual reproduction operator $\mathcal B$. Since the fecundity rate has been normalized to 1 (see Appendix \ref{Appendix-nondimensionalization}), then $\mathcal{B}$ preserves the mass and center of mass, namely,
\begin{equation}\label{eq:conserved quantities}
\Vert \mathcal{B}[F]\Vert_{L^1(\mathbb{R}^d)} = \Vert F\Vert_{\mathcal{M}(\mathbb{R}^d)},\quad \int_{\mathbb{R}^d}x\,  \mathcal{B}[F](x)\, dx = \int_{\mathbb{R}^d} x \,F(dx),
\end{equation}
for any $F\in \mathcal{M}_+(\mathbb{R}^d)$. Furthermore, it is contractive in the space of probability measures with a common center of mass, endowed with the quadratic Wasserstein metric. As discussed above, this property has been used fruitfully by {\sc G. Raoul} ({\em cf.} \cite{R-21-arxiv}) to analyse the long term behavior of the time-continuous problem in the regime of weak (and compactly supported) selection acting on fecundity. For the sake of clarity, we recall this fact and its proof below: 
\begin{lemma}\label{lem:contraction W2} Assume that $F_1,F_2\in \mathcal{P}_2(\mathbb{R}^d)$ and they have the same center of mass. Then, 
\begin{equation*}
\mathcal W_2^2(\mathcal B[F_1],\mathcal B[F_2]) \leq \frac{1}{2}\,\mathcal W_2^2(F_1,F_2) \,.
\end{equation*}
\end{lemma}
\begin{proof}
Recall the following dual characterizations of the quadratic Wasserstein distance ({\em cf.} \cite{AGS-08, V-09}): 
\begin{align}\label{E-W2-characterizations}
\begin{split}
\mathcal W_2^2(F_1,F_2)  & = \inf\left\{\int_{\mathbb{R}^{2d}} |x-y|^2 \, \,\gamma(dx,dy):\,\gamma\in \mathcal{P}(\mathbb{R}^d\times \mathbb{R}^d),\ \pi_{1\#}\gamma=F_1,\ \pi_{2\#}\gamma=F_2\right\}\\
&= \sup\left\{ \int_{\mathbb{R}^d} \phi_1\,dF_1 + \int_{\mathbb{R}^d} \phi_2 \,dF_2 :\,\left| \phi_1(x) + \phi_2(y)\right| \leq |x-y|^2\quad  \forall\,x,y\in \mathbb{R}^d \right\},
\end{split}
\end{align}
where $\pi_i:\mathbb{R}^d\times \mathbb{R}^d\rightarrow \mathbb{R}^d$ is the projection onto the $i$-th component for $i=1,2$. Taking any couple $(\phi_1,\phi_2)$ as above and using the specific form of the operator $\mathcal B$ we obtain:
\begin{align*}
&\int_{\mathbb{R}^d} \mathcal{B}[F_1](x) \,\phi_1(x)\, dx +\int_{\mathbb{R}^d} \mathcal{B}[F_2] (y)\,\phi_2(y)\, dy\\
& = \int_{\mathbb{R}^{3d}} G\left(x-\frac{x_1+x_2}{2}\right) \,\phi_1(x)\, F_1(dx_1)\,F_1(dx_2) \,dx + \int_{\mathbb{R}^{3d}} G\left(y-\frac{y_1+y_2}{2}\right) \,\phi_2(y)\, F_2(dy_1)\,F_2(dy_2) \,dy \\
& = \int_{\mathbb{R}^{3d}} G(z) \,\phi_1\left(z+\frac{x_1+x_2}{2}\right)\, F_1(dx_1)\,F_1(dx_2)\,dz + \int_{\mathbb{R}^{3d}} G(z) \,\phi_2\left(z+\frac{y_1+y_2}{2}\right)\, F_2(dy_1) \, F_2(dy_2)\, dz.
\end{align*}
Consider any transference plan $\gamma\in \mathcal{P}(\mathbb{R}^d\times \mathbb{R}^d)$ between $(F_1,F_2)$ as above. Specifically, we have that $\pi_{1\#}\gamma=F_1$ and $\pi_{2\#}\gamma=F_2$. Then, we can gather both integrals as follows:
\begin{multline*}
\int_{\mathbb{R}^d} \mathcal{B}[F_1](x) \,\phi_1(x)\, dx + \int_{\mathbb{R}^d} \mathcal{B}[F_2](y)\,\phi_2(y)\, dy \\ 
= \int_{\mathbb{R}^{5d}} G(z) \left( \phi_1\left(z+\frac{x_1+x_2}{2}\right) +  \phi_2\left(z+\frac{y_1+y_2}{2}\right) \right)  \,  \gamma(dx_1,dy_1) \, \gamma(dx_2,dy_2)\, dz.
\end{multline*}
Since the condition $\left| \phi_1(x) + \phi_2(y)\right| \leq |x-y|^2$ is verified for all $x,y\in \mathbb{R}^d$, then we find
\begin{align*}
\int_{\mathbb{R}^d} \mathcal{B}[F_1](x) &\,\phi_1(x)\, dx  + \int_{\mathbb{R}^d} \mathcal{B}[F_2] (y)\,\phi_2(y)\, dy\\
& \leq \frac{1}{4}  \int_{\mathbb{R}^{5d}} G(z) \left| (x_1+x_2) -  (y_1+y_2)\right|^2  \,  \gamma(dx_1,dy_1)\,\gamma(dx_2,dy_2)\,dz \\
& = \frac{1}{4} \int_{\mathbb{R}^{4d}}  \left| (x_1+x_2) -  (y_1+y_2)\right|^2  \,  \gamma(dx_1,dy_1) \, \gamma(dx_2,dy_2) \\
& = \frac{1}{4}\int_{\mathbb{R}^{2d}}  \left| x_1-y_1\right|^2 \, \gamma(dx_1,dy_1) + \frac{1}{4} \int_{\mathbb{R}^{2d}}  \left| x_2-y_2\right|^2 \, \gamma(dx_2,dy_2),
\end{align*}
where in the third line we have used that $G$ is a probability density and in the last line we have used the crucial fact that $F_1$ and $F_2$ share the same center of mass in order to cancel the cross-terms (otherwise the estimate would boil down to a non-expansiveness estimate). Taking supremum over $(\phi_1,\phi_2)$ and infimum over $\gamma$ and using the dual characterizations \eqref{E-W2-characterizations} yields the result.
\end{proof}

Note that the fact that both $F_1$ and $F_2$ have the same center of mass has been crucially used to cancel the crossed term. Otherwise, by the Cauchy--Schwarz inequality we would merely obtain non-expansiveness:
\begin{equation}\label{eq:non expansivity W2}
\mathcal{W}_2(\mathcal{B}[F_1],\mathcal{B}[F_2])\leq \mathcal{W}_2(F_1,F_2).
\end{equation}
Using the contractivity property of Lemma \ref{lem:contraction W2} along with the conservation of mass and center of mass in \eqref{eq:conserved quantities} yields the long term dynamics of \eqref{E-sexual-reproduction-time-discrete} in the special case $\alpha=0$.

\begin{corollary}\label{cor:contraction W2}
Set any initial datum $F_0\in \mathcal{M}_+(\mathbb{R}^d)$ such that $\int_{\mathbb{R}^d}\vert x\vert^2 F_0(dx)<\infty$ and consider the solution $\{F_n\}_{n\in \mathbb{N}}$ to the time-discrete problem \eqref{E-sexual-reproduction-time-discrete} with $\alpha=0$, {\em i.e.}, $F_n=\mathcal{B}[F_{n-1}]$ for all $n\in \mathbb{N}$. Then,
$$
\frac{\Vert F_n\Vert_{L^1(\mathbb{R}^d)}}{\Vert F_{n-1}\Vert_{L^1(\mathbb{R}^d)}}=1,\quad \mathcal{W}_2\left(\frac{F_n}{\Vert F_n\Vert_{L^1(\mathbb{R}^d)}},G_{\mu_0,2}\right)\lesssim \frac{1}{2^{n/2}},
$$
for every $n\in \mathbb{N}$, where $\mu_0:=\int_{\mathbb{R}^d}x\,F_0(dx)$. In particular, the set of stationary distributions under $\mathcal{B}$ is $\{\bF_{\alpha=0}(\cdot-\mu):\,\mu\in \mathbb{R}^d\}$, where $\bF_{\alpha=0}=G_{0,2}$ will denote here on the Gaussian centered at the origin with variance equals $2$ in agreement with the notation \eqref{E-sexual-reproduction-Gaussian-solution} in Theorem \ref{T-main}.
\end{corollary}

The preceding result might be regarded as the counterpart of Theorem \ref{T-main} for $\alpha=0$. Indeed, by Talagrand transportation inequality for a Gaussian measure \cite{T-96} we obtain the relation 
$$\mathcal{W}_2^2\left(\frac{F_n}{\Vert F_n\Vert_{L^1(\mathbb{R}^d)}},G_{\mu_0,2}\right)\leq 4\,\mathcal{D}_{\rm KL}\left(\left.\frac{F_n}{\Vert F_n\Vert_{L^1(\mathbb{R}^d)}}\right\Vert G_{\mu_0,2}\right).$$
However, Theorem \ref{T-main} does not hold when $\alpha=0$, as mentioned in Remark \ref{R-positivity-alpha}, due to two fundamental reasons. First, when $\alpha=0$ there is translation invariance and therefore, for generic $F_0\in \mathcal{M}_+(\mathbb{R}^d)$ one cannot expect that the normalized profiles $F_n/\Vert F_n\Vert_{L^1(\mathbb{R}^d)}$ always converge to the Gaussian $\bF_{\alpha=0}$ centered at the origin (contrarily to what happens when $\alpha>0$). Otherwise, the centers of mass must get shifted toward the origin, thus breaking the translation invariance. Indeed, as mentioned in Corollary \ref{cor:contraction W2}, the equilibria are not unique when $\alpha=0$ (contrarily to the case $\alpha>0$), and the center of mass of the resulting equilibria must stay equal to the initial one. Second, even if we set the initial center of mass at the origin so that we kill the translation invariance, our method of proof of Theorem \ref{T-main} leads to estimates that blow up as $\alpha\rightarrow 0$ since we have $2\bk_{\alpha=0}=1$.

\subsection{Incompatibility of Wasserstein metric with multiplicative operators}\label{SS-incompatibility-wasserstein-multiplicative} Note that by definition \eqref{E-sexual-reproduction-T-operator}, our operator $\mathcal{T}$ is the composition of the sexual resproduction operator $\mathcal{B}$ with the multiplicative operator by the survival probability $e^{-m}$. Then, it might be natural to study perturbations of the previous Lemma \ref{lem:contraction W2} including the following conservative multiplicative operator.

\begin{definition}[Normalization of multiplicative operator]\label{D-operator-M}
$$
\mathcal{M}[F]:=\dfrac{e^{-m} F}{\|e^{-m} F\|_{\mathcal{M}(\mathbb{R}^d)}},\quad F\in \mathcal{M}_+(\mathbb{R}^d)\setminus \{0\}\, .
$$
\end{definition}

However, the latter is not Lipschitz continuous with respect to the quadratic Wasserstein metric. Hence the composition of $\mathcal B$ and $\mathcal M$ is not expected to be contractive, even in the case of weak selection, without any additional restriction. We illustrate such a Lipschitz discontinuity of $\mathcal M$ in the following example.

\begin{example}\label{ex:incompatibility}
Suppose that the $m\in C^1_+(\mathbb{R})$ is radially symmetric around the origin and set $F_1, F_2\in \mathcal{P}_2(\mathbb{R})$ as the sum of two Dirac masses, $F_1$ being symmetric, and $F_2$ nearly symmetric. More precisely,
$$F_1 = \frac{1}{2} \delta_{-h} +  \frac{1}{2} \delta_h,\quad F_2  = \frac{1}{2} \delta_{-h+\varepsilon} +  \frac{1}{2} \delta_{h+\varepsilon},$$
where $h>0$ is fixed so that $m'(h)\neq 0$ and $\varepsilon>0$ is small. Note that 
$$\mathcal M[F_1] = F_1,\quad \mathcal M[F_2] = (1-p_\varepsilon) \delta_{-h+\varepsilon} +  p_\varepsilon\delta_{h+\varepsilon},$$
where $p_\varepsilon\in (0,1)$ is given by
$$p_\varepsilon= \dfrac{e^{-m(h+\varepsilon)}}{e^{-m(-h+\varepsilon)}+ e^{-m(h+\varepsilon)}}.$$
On the one hand, we have $\mathcal{W}_2(F_1,F_2) = \varepsilon$ because $F_2$ is simply deduced from $F_1$ by a translation of size $\varepsilon$. On the other hand, assuming $m'(h)>0$ for simplicity (a similar argument holds if $m'(h)<0$) and taking $\varepsilon>0$ sufficiently small we obtain $p_\varepsilon<\frac{1}{2}$ and
$$\mathcal{W}_2(\mathcal{M}[F_1],\mathcal{M}[F_2]) = \sqrt{\varepsilon^2+4h(h-\varepsilon)\left(\frac{1}{2}-p_\varepsilon\right)}\,\sim\, \varepsilon+h\,m'(h)^{1/2}\varepsilon^{1/2},$$
as $\varepsilon\rightarrow 0$ by symmetry and the mean value theorem. In a sense, the leading order term corresponds to the cost of moving a piece of mass $\frac{1}{2}-p_\varepsilon$ from $-h+\varepsilon$ to $h$ in order to equilibrate the Dirac masses in the transport plan. Since $m'(h)\neq 0$ this leads to the Lipschitz-discontinuity of the operator $\mathcal{M}$.
\end{example}

In \cite{R-21-arxiv}, this case is ruled out by assuming that the densities are uniformly bounded below on compact intervals. In our case, we will avoid relying on the above quantification in Lemma \ref{lem:contraction W2} for the contractivity of $\mathcal{B}$ and will propose a different strategy that we discuss in the sequel.

\subsection{Log-Lipschitz contraction estimate}\label{SS-log-lipschitz-contraction} In this paper, we explore an alternative approach to derive some suitable contraction of the operator $\mathcal{T}$ in appropriate norms in the regime of strong, but quadratic, selection. As anticipated in Remark \ref{R-choice-metric}, our computations suggest using a log-Lipschitz norm, which measures the uniform deviation of tails of any profile $F$ relative to the Gaussian tails of $\bF_\alpha$. More specifically, we obtain the following result for any $\alpha\geq 0$ (but only useful when $\alpha>0$).

\begin{lemma}[Log-Lipschitz estimate]\label{L-log-Lipschitz-estimate}
Let $m$ be the quadratic selection function in \eqref{E-selection-quadratic} and consider any value $\alpha\geq 0$. Then, the following estimate holds true
$$\left\Vert \nabla\log\frac{\mathcal{T}[F]}{\bF_\alpha}\right\Vert_{L^\infty(\mathbb{R}^d)}\leq 2\bk_\alpha \left\Vert \nabla\log\frac{F}{\bF_\alpha}\right\Vert_{L^\infty(\mathbb{R}^d)},$$
for any $F\in L^1_+(\mathbb{R}^d)\cap C^1(\mathbb{R}^d)$ such that $\Vert \nabla\log\frac{F}{\bF_\alpha}\Vert_{L^\infty(\mathbb{R}^d)}<\infty$. The coefficient $\bk_\alpha$ above is given by the numerical value \eqref{E-sexual-reproduction-coefficients-limit} in Theorem \ref{T-main}.
\end{lemma}

\begin{figure}[t]
\centering
\begin{tikzpicture}
\begin{axis}[
  axis x line=middle, axis y line=middle,
  xmin=0, xmax=15, xtick={0,5,10,15}, xlabel=$\alpha$,
  ymin=0, ymax=0.6, ytick={0,0.25,0.5}, ylabel=$\bk_\alpha$,
]
\addplot [
    domain=0:15, 
    samples=200, 
    color=black,
    line width=0.4mm,
]
{0.5*pow(2*(((3+2*x)-pow(pow(3+2*x,2)-8,1/2))/4),1)};
\addplot [mark=*, mark size=1.5pt] coordinates {(0,0.5)};
\end{axis}
\end{tikzpicture}
\caption{Parameter $\bk_\alpha$ against parameter $\alpha$.}
\label{fig:contraction-parameter}
\end{figure}

When applied to $F=F_n$ for any solution $\{F_n\}_{n\in \mathbb{N}}$ of the time-evolution problem \eqref{E-sexual-reproduction-time-discrete}, the above inequality in Lemma \ref{L-log-Lipschitz-estimate} allows propagating a control of growth/decrease of the log-Lipschitz norm under the flow of the equation.
In particular, when $\alpha>0$, we have $2\bk_\alpha<1$ ({\it cf.} Figure \ref{fig:contraction-parameter}), and therefore we obtain an actual contractivity estimate. There is one main drawback though: typically $\nabla \log \frac{F_n}{\bF_\alpha}$ are genuinely unbounded unless $F_0$ and $\bF_{\alpha}$ have the same Gaussian decay for large $x$ (which is much too restrictive). A major point of our work will precisely be to circumvent this unbounded factors.

However, when $\alpha=0$ we have $2\bk_\alpha=1$, and we simply obtain non-expansiveness. We emphasize that in the derivation of such a rough estimate, we do not use carefully the preservation of the center of mass. This results in a non-expansiveness estimate, which is conceptually no better than the previous non-expansiveness property \eqref{eq:non expansivity W2} in the quadratic Wasserstein metric, obtained as in Lemma \ref{lem:contraction W2} without assuming that the centers of mass are the same. For this reason, our method of proof will finally not yield satisfying results in the case $\alpha=0$, but we find this idea illuminating to address the case of a non-trivial effect of selection when $\alpha>0$. A refinement of Lemma \ref{L-log-Lipschitz-estimate} leading to real contractivity would require tackling more carefully the center of mass. However, as mentioned in Remark \ref{R-positivity-alpha}, we will not address this in the current paper, and we refer to Section \ref{S-conclusions} for some perspectives and future works in this line. 

\begin{proof}[Proof of Lemma \ref{L-log-Lipschitz-estimate}]
Our starting point resides in the following normalization of the operator $\mathcal{T}$:
\begin{equation}\label{E-normalized-T}
\frac{\mathcal{T}[F](x)}{ \bF_\alpha(x)}=\frac{\blambda_\alpha}{\int_{\mathbb{R}^d}\frac{F(x')}{\bF_\alpha(x')}\bF_\alpha(x')\,dx'}\iint_{\mathbb{R}^{2d}}\mathbf{P}(x;x_1,x_2)\frac{F(x_1)}{\bF_\alpha(x_1)}\frac{F(x_2)}{\bF_\alpha(x_2)}\,dx_1\,dx_2,
\end{equation}
for all $x\in \mathbb{R}^d$, where we have
\begin{equation}\label{E-normalized-T-P}
\mathbf{P}(x;x_1,x_2)=\frac{e^{\frac{1}{2}(\frac{1}{\bsigma_\alpha^2}-\alpha)|x|^2}}{\blambda_\alpha (2\pi\bsigma_\alpha)^d}\exp\left[-\frac{1}{2}\left\vert x-\frac{x_1+x_2}{2}\right\vert^2-\frac{1}{2\bsigma_\alpha^2}(|x_1|^2+|x_2|^2)\right].
\end{equation}
Above we have exploited the explicit Gaussian shape of $\bF_\alpha$ to find $\mathbf{P}$ explicitly. Note that $\mathbf{P}$ consists in a one-step Markov transition kernel representing the probability that parental traits $(x_1,x_2)$ lead to a descendant trait $x$. Indeed, $\iint_{\mathbb{R}^{2d}}\mathbf{P}(x;x_1,x_2)\,dx_1\,dx_2=1$ for all $x\in \mathbb{R}^d$ because $(\blambda_\alpha,\bF_\alpha)$ solves the non-linear eigenproblem \eqref{E-eigenproblem-time-discrete}. We remark that the quadratic form in the exponential in \eqref{E-normalized-T-P} reaches its maximum value at $(x_1,x_2) = (k x, k x)$ for $k=\bsigma_\alpha^2/(2+\bsigma_\alpha^2)$ and by definition \eqref{E-sexual-reproduction-coefficients-limit} of $\bk_\alpha$, we infer $k=\bk_\alpha$. This motivates using the change of variables
\begin{equation}\label{E-change-variables-log-lipschitz-estimate}
x_1=\bk_\alpha x+y_1,\quad x_2=\bk_\alpha x+y_2,
\end{equation}
which appropriately centers the quadratic form at the minimum. Specifically, we obtain
\begin{align*}
&\frac{1}{2}\left\vert x-\frac{x_1+x_2}{2}\right\vert^2-\frac{1}{2\bsigma_\alpha^2}(|x_1|^2+|x_2|^2)\\
&\, =-\frac{1}{2}\left\vert\frac{y_1+y_2}{2}\right\vert^2-\frac{1}{2\bsigma_\alpha^2}(|y_1|^2+|y_2|^2)-\frac{1}{2}\left((1-\bk_\alpha)^2+\frac{2\bk_\alpha^2}{\bsigma_\alpha^2}\right)|x|^2+\frac{1}{2}\left((1-\bk_\alpha)-\frac{2\bk_\alpha}{\bsigma_\alpha^2}\right)x\cdot (y_1+y_2),\\
&\, =-\frac{1}{2}\left\vert\frac{y_1+y_2}{2}\right\vert^2-\frac{1}{2\bsigma_\alpha^2}(|y_1|^2+|y_2|^2)-\frac{1}{2}\left(\frac{1}{\bsigma_\alpha^2}-\alpha\right)|x|^2,
\end{align*}
where in the second line we have noticed that the coefficient of the crossed term $x\cdot (y_1+y_2)$ vanishes due to the relation $\bk_\alpha=\frac{\bsigma_\alpha^2}{2+\bsigma_\alpha^2}$, and the coefficient of the $|x|^2$ factor can be reformulated as
$$(1-\bk_\alpha)^2+\frac{2\bk_\alpha^2}{\bsigma_\alpha^2}=\frac{2}{2+\bsigma_\alpha^2}=\frac{1}{1+\frac{\bsigma_\alpha^2}{2}}=\frac{1}{\bsigma_\alpha^2}-\alpha,$$
thanks to the implicit equation \eqref{E-sexual-reproduction-Gaussian-solution-variance-equation} satisfied by $\bsigma_\alpha^2$. Therefore, \eqref{E-normalized-T}-\eqref{E-normalized-T-P} transform into
\begin{align}
&\frac{\mathcal{T}[F](x)}{\bF_\alpha(x)}=\frac{\blambda_\alpha}{\int_{\mathbb{R}^d}\frac{F(x')}{\bF_\alpha(x')}\bF_\alpha(x')\,dx'}\iint_{\mathbb{R}^{2d}}\widetilde{\mathbf{P}}(y_1,y_2)\frac{F(\bk_\alpha x+y_1)}{\bF_\alpha(\bk_\alpha x+y_1)}\frac{F(\bk_\alpha x+y_2)}{\bF_\alpha(\bk_\alpha x+y_2)}\,dy_1\,dy_2,\label{E-normalized-T-2}\\
&\widetilde{\mathbf{P}}(y_1,y_2)=\frac{1}{\blambda_\alpha (2\pi\bsigma_\alpha)^d}\exp\left[-\frac{1}{2}\left\vert \frac{y_1+y_2}{2}\right\vert^2-\frac{1}{2\bsigma_\alpha^2}(|y_1|^2+|y_2|^2)\right].\label{E-normalized-T-P-2}
\end{align}
We remark that the new Markov transition kernel $\widetilde{\mathbf{P}}=\widetilde{\mathbf{P}}(y_1,y_2)$ does not depend on $x$ thanks to the explicit cancellation of the $|x|^2$ dependent factors. At this level, we start to observe the ergodicity phenomenon since the dependence of $x$ on the right hand side of \eqref{E-normalized-T-2} has shrunk by a factor $\bk_\alpha$. A possible strategy to see if there is a quantitative degradation of the dependence on $x$ is to take logarithmic derivatives and try to relate the log-Lipschitz norms of $\mathcal{T}[F]$ and $F$. Specifically, we have
\begin{equation}\label{E-normalized-T-log-derivative}
\nabla\log\frac{\mathcal{T}[F](x)}{\bF_\alpha(x)}=\bk_\alpha\int_{\mathbb{R}^{2d}}\left(\nabla\log\frac{F(\bk_\alpha x+y_1)}{\bF_\alpha(\bk_\alpha x+y_1)}+\nabla\log\frac{F(\bk_\alpha x+y_2)}{\bF_\alpha(\bk_\alpha x+y_2)}\right)\,\nu(x;dy_1,dy_2),
\end{equation}
where the $x$-dependent measures $\nu(x;dy_1,dy_2)$ on the variables $(y_1,y_2)$ have the following density with respect to the Lebesgue measure:
$$\frac{\nu(x;dy_1,dy_2)}{dy_1\,dy_2}=\frac{\displaystyle\widetilde{\mathbf{P}}(y_1,y_2)\frac{F(\bk_\alpha x+y_1)}{\bF_\alpha(\bk_\alpha x+y_1)}\frac{F(\bk_\alpha x+y_2)}{\bF_\alpha(\bk_\alpha x+y_2)}}{\displaystyle\iint_{\mathbb{R}^{2d}}\widetilde{\mathbf{P}}(y_1',y_2')\frac{F(\bk_\alpha x+y_1')}{\bF_\alpha(\bk_\alpha x+y_1')}\frac{F(\bk_\alpha x+y_2')}{\bF_\alpha(\bk_\alpha x+y_2')}\,dy_1'\,dy_2'}.$$
Since the integrands of \eqref{E-normalized-T-log-derivative} are uniformly bounded by our assumptions, we end the proof by taking $L^\infty$ bounds and using that $\nu$ are probability measures on the variables $(y_1,y_2)$.
\end{proof}

\subsection{Brief description of our strategy}\label{SS-strategy} As advanced before, our strategy is based on a finer understanding of the iterations of $\mathcal T$ across generations. Specifically, it relies on a suitable reformulation of the solutions $\{F_n\}_{n\in \mathbb{N}}$ solving the recursion \eqref{E-sexual-reproduction-time-discrete} for the special quadratic selection function $m$ in \eqref{E-selection-quadratic}. At this stage we intentionally keep notation simple and intuitive, since our goal is to briefly present the main strategy. However, a rigorous approach with more descriptive notation for the trees of ancestors, which arise from the recursion, is developed in detail in Section \ref{S-restating-iterations}.

The first step consists in choosing the appropriate normalization extending the previous normalization $F_n/\bF_\alpha$ in the proof of Lemma \ref{L-log-Lipschitz-estimate}. In this paper, we have opted for 
$$\bar F_n = \frac{e^m F_n}{\bF_{\alpha=0}},$$ ({\it cf.} Definition \ref{D-rescaled-F}) but there is some freedom here. In particular, the term $e^m$ is not mandatory, but it is convenient as such for an easier sorting of the various terms. Indeed, we may typically consider any normalization $F_n/G_{0,\sigma^2}$ with $\sigma^2$ larger but arbitrarily close to $\bsigma_\alpha^2$ (being $\bsigma_\alpha^2$ the variance \eqref{E-sexual-reproduction-Gaussian-solution-variance-equation} of the expected equilibrium $\bF_\alpha$). In other words, we could take any ``close-to-optimal'' normalization as compared to the ``optimal'' normalization used in the strategy in Section \ref{SS-log-lipschitz-contraction}. However, for general selection functions $m$ we do not know the expected equilibrium. To account for a more robust viewpoint, we  follow a more robust approach so that we do not ``optimize'' the normalization.

In a nutshell: we iterate the operator $\mathcal{T}$ in the recursion for $F_n$ up to the initial datum and we obtain (modulo a multiplicative constant) a formula of the following type
\begin{align*}
&F_n(x)\propto \int_{\mathbb{R}^{(2^{n+1}-2)d}} \mathbf{P}_n(x;\{x_i\})\,\prod_j  \bar F_0(x_j)\, d\{x_i\},
\end{align*}
for an explicit $n$-step transition kernel $\mathbf{P}_n$ with Gaussian shape:
$$\mathbf{P}_n(x;\{x_i\}):=\exp(-\mathbf{Q}_n(x;\{x_i\})).$$
Here, the index $i$ spans the $2^{n+1}-2$ members of the genealogical tree with $n$ generations of ancestors from the root (excluded) to the leaves, whilst $j$ only spans the $2^n$ members at the leaves. In addition, $\mathbf{Q}_n=\mathbf{Q}_n(x;\{x_i\})$ is a quadratic form taking as arguments all traits $x$ and $\{x_i\}$ ({\em i.e.}, all the $2^{n+1}-1$ ancestors in the family chart including the root $x$).  As compared to Section \ref{SS-log-lipschitz-contraction}, we have iterated $n$ times, which raises the dimension of the quadratic form to $2^{n+1}-1$. In addition, the change of normalization leads to additional quadratic contributions $\frac{\gamma}{2} |x_i|^2$ (with $\gamma=\frac{\bsigma_\alpha^2}{2+\bsigma_\alpha^2})$). Therefore, the minimizers of $\mathbf{Q}_n$ get shifted ({\em cf.} Section \ref{S-restating-iterations}). As a result of successive backwards changes of variables from leaves to the root very much in the spirit of \eqref{E-change-variables-log-lipschitz-estimate}, the following alternative expression is obtained:
\begin{equation}\label{eq:iterations alpha>0}
F_n(x)\propto \widetilde{\bF}_n(x) \int_{\mathbb{R}^{(2^{n+1}-2)d}} \widetilde{\boldsymbol{P}}_n(\{y_i\})\,\prod_j  \bar F_0(\Phi_n^j(x;\{y_i\}))\, d\{y_i\}\,.
\end{equation}
Above, for every leaf $j$ the maps $\Phi^j_n$ defined by
$$\Phi^j_n(x;\{y_i\}):=\kappa_n x+\Lambda^j(\{y_i\}),$$
where $\Lambda^j_n$ are affine transformations with respect to the variables $\{y_i\}$. We shall call them the lineage maps ({\it cf.} Definition \ref{D-sexual-reproduction-Gaussian-quadratic-lineage-maps}), since they contain precise information of the effect of the leaf $j$ of the genealogical tree on the resulting trait $x$. In addition, $\boldsymbol{P}_n$ are $n$-step transition kernel with Gaussian shape:
$$\widetilde{\boldsymbol{P}}_n(\{y_i\})=\exp(-\widetilde{\mathbf{Q}}_n(\{y_i\})),$$
for a $x$-independent quadratic form $\widetilde{\mathbf{Q}}_n$ taking as arguments all traits $\{y_i\}$ (root excluded)}. In this case, the $|x|^2$ remainders coming from the changes of variables on $\mathbf{P}_n$ do not simplify due to our different normalization, and they contribute with an explicit Gaussian factor $\widetilde{\bF}_n(x)$. Then, the dependency upon $x$ is split into two parts: the explicit term $\widetilde{\bF}_n(x)$ outside the integral (which is proven to converge towards $\bF_\alpha$ by construction) and the contribution $\kappa_n x$ of the lineage maps $\Phi^j_n(x;\{y_i\})$ inside $\bar F_0$ for an appropriate sequence $\{\kappa_n\}_{n\in \mathbb{N}}\subseteq \mathbb{R}_+$. The major observation here is that $\kappa_n\ll 2^{-n}$ when $\alpha>0$ because of the strong shift towards the origin under selection. Otherwise, for $\alpha=0$ we only have $\kappa_n=2^{-n}$.

Formula \eqref{eq:iterations alpha>0} for $\alpha>0$ then suggests a strong form of ergodicity, reminiscent of the contraction of the log-Lipschitz norm in Lemma \ref{L-log-Lipschitz-estimate}, where the dependency on $x$ fades as $n\rightarrow\infty$. Indeed, there are still $2^n$ terms in the product indexed by $j$, but the contraction parameter $\kappa_n$ appears to decay fast enough to compensate them. To make this argument quantitative, let us differentiate \eqref{eq:iterations alpha>0} again to obtain
\begin{equation}\label{eq:iterations alpha>0 log-derivative}
\nabla \log F_n(x) = \nabla \log   \widetilde{\bF}_n (x) + \kappa_n \sum_j \int_{\mathbb{R}^{(2^{n+1}-2)d}} \nabla \log \bar F_0(\Phi^j_n(x;\{y_i\})) \, \nu_n(x;d\{y_i\}) \,, 
\end{equation}
where the $x$-dependent probability measures $\nu_n(x;d\{y_i\})$ on the variables $\{y_i\}$ have the following density with respect to the Lebesgue measure:
$$
\frac{\nu_n(x;d\{y_i\})}{d\{y_i\}}=\frac{\widetilde{\boldsymbol{P}}_n(\{y_i\})\,\prod_j \bar F_0(\Phi^j_n(x;\{y_i\}))}{\int_{\mathbb{R}^{(2^{n+1}-2)d}}\widetilde{\boldsymbol{P}}_n(\{y_i'\})\,\prod_j \bar F_0(\Phi^j_n(x;\{y_i'\}))\,d\{y_i'\}}.
$$

We remark that a naive repetition of the strategy in Lemma \ref{L-log-Lipschitz-estimate} under the additional assumption that $\nabla\log \bar F_0\in L^\infty(\mathbb{R}^d)$ would immediately retrieve an exponential decay of the second factor in \eqref{eq:iterations alpha>0 log-derivative} when $\alpha>0$ (or only a uniform bound if $\alpha=0$) since the large sum over $j$ would be bounded by $2^n\kappa_n  \|\nabla \log \bar F_0 \|_{L^\infty(\mathbb{R}^d)}$. However, as mentioned before, there is a strong drawback: the factors $\nabla \log \bar F_0$ are not bounded, not even after long enough times. In the current formal argument, we have not discussed about the precise shape of the linear components $\Lambda^j_n$ of the lineage maps $\Phi^j_n$, since it was not relevant thus far. However, we anticipate that to overcome the aforementioned complication, we shall require more precise estimates for large values of $\{y_i\}$ in the high-dimensional integral, which are guaranteed by the strong enough decay of $\kappa_n$, and which will be essential in the rigorous proof in Section \ref{S-ergodicity}.

More specifically, and this explains why our approach moves from one-step contraction estimates (in the spirit of Section \ref{SS-log-lipschitz-contraction}) to ergodicity results, note that the same reformulation of the iteration as above could be done exactly $n-k$ times up to an advanced enough time step $k$. By doing so, we can grasp on some natural Gaussian shaping under selection, which we expect to lead to unbounded $\nabla \log \bar F_k$, but growing sublinearly at infinity. Hence, we require precise compensations in the high-dimensional integral \eqref{eq:iterations alpha>0 log-derivative} which we find by deriving a suitable control of moments of $\nu_n$. By doing so, one has to irremediably move from uniform estimates to averaged estimates, and being able to propagate them for large times. This requires a thorough and highly technical analysis, which becomes the main objective of this paper.

For an easier readability, and to guide the reader along the various steps of the proof, we provide an overall map of it in Figure \ref{fig:overall-map-proof}, which allows interconnecting the main fundamental results.
\begin{figure}[t]
\def\svgwidth{1.2\textwidth}
\centering
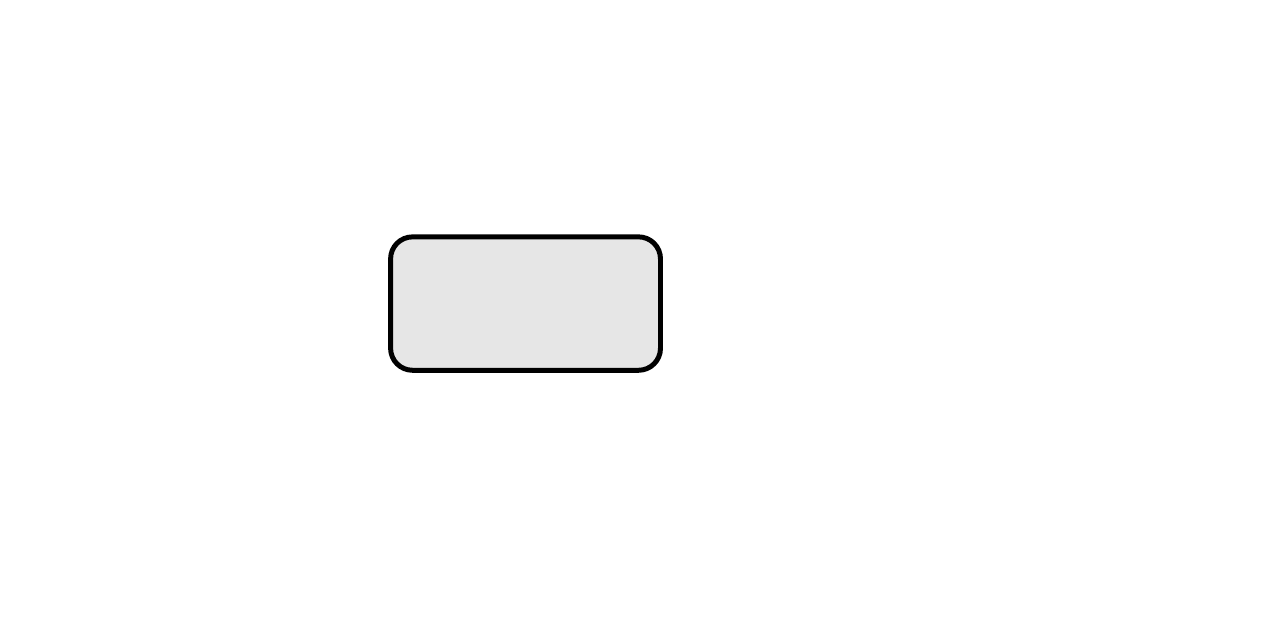
\caption{Overall map of the proof of Theorem \ref{T-main}}
\label{fig:overall-map-proof}
\end{figure}

\section{Preliminaries}\label{S-preliminaries}

In this part, we collect some necessary preliminary tools and results that will be used later on.

\subsection{Perfect binary trees}\label{SS-trees}

In this paper, we shall systematically use indices $i$ that do not range on discrete intervals $\{1,\ldots,n\}$ but rather on the vertices of a specific type of trees, which are called \textit{perfect binary trees}, see \cite{R-11,W-85}. First, we introduce the notion of \textit{binary trees}, which we present using their \textit{universal address system}, see \cite[Section 11.3.3]{R-11}. Namely, a binary tree $\Tree$ consists in a finite subset of words $\Tree\subseteq \Words_2:=\cup_{k=0}^\infty\{1,2\}^k$ with letters in the alphabet $\{1,2\}$ that verifies:
\begin{enumerate}[label=(\roman*)]
\item $\emptyset\in \Tree$.
\item If $i1\in \Tree$ or $i2\in \Tree$ for some word $i\in \Words_2$, then $i\in \Tree$.
\end{enumerate}
This implies that the root is the empty word $\empty$, and $\Tree$ is stable under chopping letters on the right of its words. Given a word $i\in \Tree$, we denote its length (number of letters) or \textit{height} in the tree by $\vert i\vert$.  In particular, $\vert \emptyset\vert=0$. A binary tree $\Tree$ is said to be \textit{perfect} if, in addition, the following properties hold:
\begin{enumerate}[label=(\roman*)]
\setcounter{enumi}{2}
\item Given a word $i\in \Words_2$, then $i1\in \Tree$ if, and only if, $i2\in \Tree$. 
\item There exists $n\in \mathbb{N}$ such that $\vert i\vert \leq n$ for every $i\in \Tree$ and
$$\#\{i\in \Tree:\,\vert i\vert=n\}=2^n.$$
\end{enumerate}
This implies that, except for the root, words in $\Tree$ appear in couples $\{i1,i2\}$ and paths in the tree arising from the root achieve the same maximal height $n$. The above four conditions determine a unique tree, which we call the \textit{perfect binary tree with height $n\in \mathbb{N}$} and we denote by $\Tree^n$. See Figure \ref{fig:trees} for a graphical representation of the perfect binary tree $\Tree^3$ of height $3$. Binary trees are often used to describe the different generations of offsprings after a given individual. In our case, we shall make a reverse use of trees. Namely, a binary tree will represent the \textit{family tree} or \textit{pedigree chart} of an individual, consisting of the different generations of \textit{ancestors} of such an individual. For this purpose, we shall establish the following terminology:
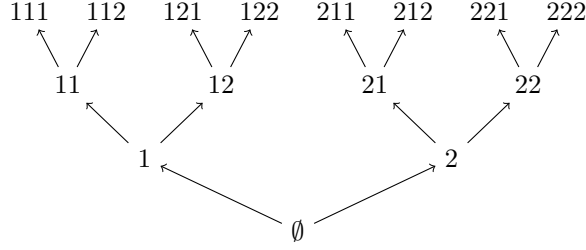
\begin{figure}[t]
\centering
\begin{forest}
for tree={grow=north,edge={->}}
[$\emptyset$ [2 [22 [222] [221]] [21 [212] [211]]] [1 [12 [122] [121]] [11 [112] [111]]]]
\end{forest}
\caption{Perfect binary tree $\Tree^3$}
\label{fig:trees}
\end{figure}

\begin{enumerate}[label=(\roman*)]
\item (\textbf{Leaves}) We say that a vertex $i\in \Tree^n$ is a leaf of the tree if $\vert i\vert=n$. The set of all leaves of $\Tree^n$ will be denoted by 
$$\Leaves^n:=\{i\in \Tree^n:\,\vert i\vert=n\}.$$
\item (\textbf{Levels}) We say that a vertex $i\in \Tree^n$ is on the $m$-th level of the tree if $\vert i\vert=m$ for some $1\leq m\leq n$. The set of all vertices on the $m$-th level will be denoted
$$\Level_m^n:=\{i\in \Tree^n:\,\vert i\vert=m\}.$$
In particular, note that $\Level_0^n=\{\emptyset\}$ is the root and $\Level_n^n=\Leaves^n$ are the leaves. For simplicity of notation, we denote the set of vertices of the root-free and leaves-free tree respectively by
$$\Tree^n_*:=\Tree^n\setminus\{\emptyset\},\quad \widehat{\Tree}^n:=\Tree^n\setminus\Leaves^n.$$
\item (\textbf{Child}) Given a vertex $i\in \Tree^n_*$, then there exists a unique word $j\in\Words_2$ such that $i=j1$ or $i=j2$. Such a vertex $j$ is called the \textit{child} of $i$ and will be denoted by $\child(i)=j.$
\item (\textbf{Parents}) Given a vertex $\in \widehat{\Tree}^n$, we define the \textit{parents} of $i$ as the subset of all vertices that have the same common child $i$, that is,
$$\Parents(i):=\{i1,i2\}.$$
\item (\textbf{Mate}) Given a vertex $i\in \Tree^n_*$, we define the \textit{mate} of $i$ and we denote it $\mate(i)$ as the only other vertex in $\Tree^n$ that has the same child as $i$, {\em i.e.},
$$\{i,\mate(i)\}=\{\child(i)1,\child(i)2\}.$$
\item (\textbf{Tree order}) Given two vertices $i,j\in \Tree^n$, we say that $i\leq j$ if the associated words are so ordered according to the lexicographical order of the set of words $\Words_2$ with two letters $\{1,2\}$.
\item (\textbf{Highest common descendant}) Given two vertices $i,j\in \widehat{\Tree}^n$, we define the \textit{highest common descendant} of $i$ and $j$, and we denote it by $i\wedge j$, as
$$i\wedge j:=\max\{l\in \Tree^n:\,l\leq i,\,l\leq j\},$$
where the maximum is considered with respect to above tree (lexicographic) order.
\end{enumerate}

\begin{remark}[Tree-indexed variables]\label{R-tree-indexed-variables}
Given $n\in \mathbb{N}$, we shall identify $\mathbb{R}^{2(2^n-1)d}\equiv (\mathbb{R}^d)^{\Tree^n_*}$ and $\mathbb{R}^{2^nd}\equiv (\mathbb{R}^d)^{\Leaves^n}$. Specifically, vectors $\bx_n,\by_n\in \mathbb{R}^{2(2^n-1)d}$ and $\bz_n\in\mathbb{R}^{2^nd}$ will be regarded as indexed families
\begin{equation}\label{E-tree-indexed-variables}
\bx_n=(x_i)_{i\in \Tree^n_*},\quad \by_n=(y_i)_{i\in \Tree^n_*},\quad \bz_n=(z_j)_{j\in \Leaves^n},
\end{equation}
where $x_i,y_i\in \mathbb{R}^d$ and $z_j\in \mathbb{R}^d$ for each $i\in \Tree^n_*$ and $j\in \Leaves^n$.
\end{remark}

\subsection{Gaussian solutions}\label{SS-Gausian-solutions}
In this part, we compute particular solutions of the time-discrete problem \eqref{E-sexual-reproduction-time-discrete} and the associated eigenproblem \eqref{E-eigenproblem-time-discrete}. As we advanced before, we shall exploit the explicit algebraic structure imposed by $G$ (given by the Gaussian mixing kernel \eqref{E-mixing-Gaussian}) and $m$ (given by the quadratic selection function \eqref{E-selection-quadratic}). Namely, explicit Gaussian solution will be obtained. We recall first the following stability property of Gaussians under convolutions.

\begin{lemma}[Stability of Gaussians]\label{L-convolution-gaussians}
The following relation holds true
\begin{equation}\label{E-convolution-gaussians}
G_{\mu_1,\sigma_1^2}*G_{\mu_2,\sigma_2^2}=G_{\mu_1+\mu_2,\sigma_1^2+\sigma_2^2},
\end{equation}
for any couple of means $\mu_1,\mu_2\in \mathbb{R}^d$ and variances $\sigma_1^2,\sigma_2^2>0$.
\end{lemma}

Using the above result, we can obtain the following explicit evaluation of the operator $\mathcal{T}$ in \eqref{E-sexual-reproduction-T-operator} over the class of Gaussian functions.

\begin{lemma}[Evaluation on Gaussians]\label{L-evaluation-gaussians}
Consider any $\mu\in \mathbb{R}^d$ and $\sigma^2\in \mathbb{R}_+^*$, then
$$\mathcal{T}[G_{\mu,\sigma^2}]=m_*\,G_{\mu_*,\sigma_*^2},$$
where the parameters $m_*$, $\mu_*$ and $\sigma_*^2$ are given by:
\begin{equation}\label{E-evaluation-gaussians-parameters}
m_*=\frac{e^{-\frac{1}{2}\frac{\alpha\vert \mu\vert^2}{1+\alpha(1+\frac{\sigma^2}{2})}}}{(1+\alpha(1+\frac{\sigma^2}{2}))^{d/2}},\quad \mu_*=\frac{\mu}{1+\alpha(1+\frac{\sigma^2}{2})},\quad \sigma_*^2=\frac{1+\frac{\sigma^2}{2}}{1+\alpha(1+\frac{\sigma^2}{2})}.
\end{equation}
\end{lemma}

\begin{proof}
On the one hand, note that by Lemma \ref{L-convolution-gaussians}
$$\mathcal{B}[G_{\mu,\sigma^2}](x)=\left(G\left(\frac{\cdot}{2}\right)*G_{\mu,\sigma^2}*G_{\mu,\sigma^2}\right)(2x)=G_{\mu,1+\frac{\sigma^2}{2}}(x),$$
for each $x\in \mathbb{R}^d$. Therefore, by definition of $\mathcal{T}$ in \eqref{E-sexual-reproduction-T-operator} we obtain
$$\mathcal{T}[G_{\mu,\sigma^2}](x)=e^{-\frac{\alpha}{2}\vert x\vert^2}G_{\mu,1+\frac{\sigma^2}{2}}(x)=\frac{1}{(2\pi (1+\frac{\sigma^2}{2}))^{d/2}}\exp\left(-\frac{\alpha}{2}\vert x\vert^2-\frac{1}{2}\frac{1}{1+\frac{\sigma^2}{2}}\vert x-\mu\vert^2\right),$$
for each $x\in \mathbb{R}^d$. By completing the square inside the exponential, we conclude our result.
\end{proof}

Consequently, the following explicit Gaussian solution of the eigenproblem \eqref{E-eigenproblem-time-discrete} is found.

\begin{proposition}[Gaussian solution of the eigenproblem]\label{P-sexual-reproduction-Gaussian-solution}
Assume that $\alpha\in \mathbb{R}_+^*$, then the eigenproblem \eqref{E-eigenproblem-time-discrete} has a unique Gaussian solution $(\blambda_\alpha,\bF_\alpha)$, determined by the relation \eqref{E-sexual-reproduction-Gaussian-solution} in Theorem \ref{T-main}.
\end{proposition}

\begin{proof}
We look for $\lambda\in\mathbb{R}_+^*$, $\mu\in \mathbb{R}^d$ and $\sigma^2\in\mathbb{R}_+^*$ such that $\lambda G_{\mu,\sigma^2}=\mathcal{T}[G_{\mu,\sigma^2}]$. By Lemma \ref{L-evaluation-gaussians} and bearing in mind parameters $m_*$, $\mu_*$ and $\sigma_*^2$ in \eqref{E-evaluation-gaussians-parameters} we obtain that $(\lambda,\mu,\sigma^2)$ must solve the equations:
$$\lambda=\frac{e^{-\frac{1}{2}\frac{\alpha\vert \mu\vert^2}{1+\alpha(1+\frac{\sigma^2}{2})}}}{(1+\alpha(1+\frac{\sigma^2}{2}))^{d/2}},\quad \mu=\frac{\mu}{1+\alpha(1+\frac{\sigma^2}{2})},\quad \sigma^2=\frac{1+\frac{\sigma^2}{2}}{1+\alpha(1+\frac{\sigma^2}{2})}.$$
Hence, the only solution is given by $\mu=0$ and $\lambda$ and $\sigma^2$ are determined by \eqref{E-sexual-reproduction-Gaussian-solution} and \eqref{E-sexual-reproduction-Gaussian-solution-variance-equation}.
\end{proof}

\begin{remark}[Dependency on selection]\label{R-sexual-reproduction-Gaussian-solution-variance}
The eigenvalue $\blambda_\alpha\in \mathbb{R}_+^*$ and the variance $\bsigma_\alpha^2\in \mathbb{R}_+^*$ determined by the relations \eqref{E-sexual-reproduction-Gaussian-solution} and \eqref{E-sexual-reproduction-Gaussian-solution-variance-equation} for the the unique Gaussian solution $(\blambda_\alpha,\bF_\alpha)$ to the eigenproblem \eqref{E-eigenproblem-time-discrete} are monotonically decreasing with the selection coefficient $\alpha$. Namely, the larger $\alpha$, the smaller $\bsigma_\alpha^2$ and $\blambda_\alpha$. Indeed, we obtain (see also Figure \ref{fig:behavior-eigenvalue-variance-Gaussian-solution})
\begin{align*}
\alpha \nearrow \infty \quad &\Longrightarrow\quad \bsigma_\alpha^2 \searrow 0,\quad \blambda_\alpha \searrow 0\\
\alpha \searrow 0 \quad &\Longrightarrow\quad \bsigma_\alpha^2 \nearrow 2,\quad \blambda_\alpha\nearrow 1.
\end{align*}
Therefore, if selection is strong, then $\bF_\alpha$ is very concentrated around the origin and, if selection ceases, then $\bF_\alpha$ is twice as spread as the mixing kernel $G$ in \eqref{E-mixing-Gaussian}, a famous result in quantitative genetics, see {\em e.g.} \cite{B-80}. Note that $\blambda_\alpha< 1$ for any $\alpha\in \mathbb{R}_+^*$ and, consequently, the special steady solutions $F_n=\blambda_\alpha^n \bF_\alpha$ coming from ansatz \eqref{E-ansatz-discrete} always get extinct for large $n$. This is an artificial consequence of our parameter reduction in Appendix \ref{Appendix-nondimensionalization}. In particular, if we maintain parameter $\beta$ in \eqref{E-sexual-reproduction-time-discrete-two-parameters} coming from the ratio between birth and mortality rates, then the above special steady solution only extincts if $\beta<(1+\alpha(1+\frac{\bsigma_\alpha^2}{2}))$.
\end{remark}

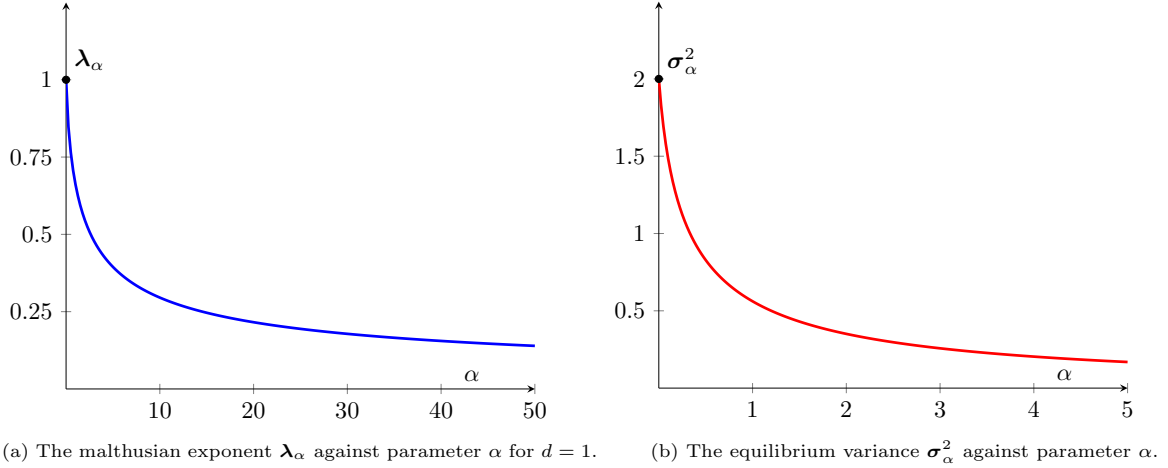
\begin{figure}
\centering
\begin{subfigure}{0.49\textwidth}
\begin{tikzpicture}[scale=0.9]
\begin{axis}[
  axis x line=middle, axis y line=middle,
  label/.style={at={(current axis.right of origin)},anchor=west},
  xmin=0, xmax=50, xtick={0,10,20,30,40,50}, xlabel=$\alpha$,
  ymin=0, ymax=1.25, ytick={0,0.25,0.5,0.75,1}, ylabel=$\blambda_\alpha$,
]
\addplot [
    domain=0:50, 
    samples=200, 
    color=blue,
    line width=0.4mm,
]
{1/pow(1+x*(1+(sqrt((2*x+1)^2+8*x)-(2*x+1))/(4*x)),1/2)};
\addplot [mark=*, mark size=1.5pt] coordinates {(0,1)};
\end{axis}
\end{tikzpicture}
\caption{The malthusian exponent $\blambda_\alpha$ against parameter $\alpha$ for $d=1$.}
\label{fig:behavior-eigenvalue-Gaussian-solution}
\end{subfigure}
\begin{subfigure}{0.49\textwidth}
\begin{tikzpicture}[scale=0.9]
\begin{axis}[
  axis x line=middle, axis y line=middle,
  xmin=0, xmax=5, xtick={0,...,5}, xlabel=$\alpha$,
  ymin=0, ymax=2.5, ytick={0,0.5,1,1.5,2}, ylabel=$\bsigma_\alpha^2$,
]
\addplot [
    domain=-0.001:5, 
    samples=200, 
    color=red,
    line width=0.4mm,
]
{(sqrt((2*x+1)^2+8*x)-(2*x+1))/(2*x)};
\addplot [mark=*, mark size=1.5pt] coordinates {(0,2)};
\end{axis}
\end{tikzpicture}
\caption{The equilibrium variance $\bsigma_\alpha^2$ against parameter $\alpha$.}
\label{fig:behavior-variance-Gaussian-solution}
\end{subfigure}
\caption{Plot of eigenvalue $\blambda_\alpha$ and variance at equilibrium $\bsigma_\alpha^2$ against $\alpha$ for $d=1$. The case $\alpha = 0$ corresponds to absence of selection, which is conservative ($\blambda_{\alpha=0} = 1$), and the variance at equilibrium $\bsigma_{\alpha=0}^2$ is twice the variance of the mixing kernel in $\mathcal B$.}
\label{fig:behavior-eigenvalue-variance-Gaussian-solution}
\end{figure}

\begin{remark}[Eigenproblem with $\alpha=0$]
The same ideas as in Proposition \ref{P-sexual-reproduction-Gaussian-solution} yield the Gaussian solutions to the eigenproblem \eqref{E-eigenproblem-time-discrete} in the absence of selection ({\em i.e.}, $\alpha=0$). However, there is no longer uniqueness due to the translation invariance of $\mathcal{B}$. Indeed, we obtain the Gaussian solutions:
$$\lambda=\blambda_{\alpha=0}=1,\quad F=\bF_{\alpha=0}(\cdot-\mu)=G_{\mu,2},$$
for any $\mu\in \mathbb{R}^d$. Indeed, it is a consequence of the contraction property stated in Section \ref{SS-properties-B} that these are all possible generic solutions (not only Gaussian). Specifically, by the conservation of mass \eqref{eq:conserved quantities} we obtain that the only possible eigenvalue is $\lambda=1$. Then, the eigenproblem \eqref{E-eigenproblem-time-discrete} reduces to a fixed point equation for $\mathcal{B}$. We can then conclude by Corollary \ref{cor:contraction W2}.
\end{remark}

We emphasize that the above presents a crucial difference of behavior between the nonlinear problem \eqref{E-eigenproblem-time-discrete} with sexual reproduction and the analogous linear version with asexual reproduction operator \eqref{E-asexual-reproduction-operator}. Namely, whilst in the former case there exists nontrivial solutions in the absence of selection, in the latter case it can be seen that such a nontrivial solution does not exists ({\em e.g.}, through Fourier arguments). This suggests that whilst in the linear problem both reproduction and selection stay balanced, in the nonlinear problem reproduction appears to dominates selection structurally.

Finally, we note that by iteration of the operator $\mathcal{T}$ and Lemma \ref{L-evaluation-gaussians} we can compute the explicit form of solutions of the time-discrete problem \eqref{E-sexual-reproduction-time-discrete} issued at Gaussian initial data. In fact, the Gaussian structure is preserved along generations, although mass, mean and variance are modified.

\begin{proposition}[Gaussian solutions of the time-discrete problem]\label{P-sexual-reproduction-Gaussian-solution-time-discrete}
Consider any $m_0\in \mathbb{R}_+^*$, $\mu_0\in \mathbb{R}^d$ and $\sigma_0^2\in \mathbb{R}_+$, and set the Gaussian initial datum $F_0=m_0\,G_{\mu_0,\sigma_0^2}$. Hence, the solution $\{F_n\}_{n\in \mathbb{N}}$ to the time-discrete problem \eqref{E-sexual-reproduction-time-discrete} takes the form
\begin{equation}\label{E-sexual-reproduction-Gaussian-solution-time-discrete}
F_n=m_n\,G_{\mu_n,\sigma_n^2},
\end{equation}
for $n\in \mathbb{N}$, where the parameters $m_n\in \mathbb{R}_+^*$, $\mu_n\in \mathbb{R}^d$ and $\sigma_n^2\in \mathbb{R}_+^*$ are governed by the recursions:
\begin{equation}\label{E-sexual-reproduction-Gaussian-solution-time-discrete-parameters}
m_n=m_{n-1}\frac{e^{-\frac{1}{2}\frac{\alpha\vert \mu_{n-1}\vert^2}{1+\alpha\big(1+\frac{\sigma_{n-1}^2}{2}\big)}}}{(1+\alpha(1+\frac{\sigma_{n-1}^2}{2}))^{d/2}},\quad \mu_n=\frac{\mu_{n-1}}{1+\alpha(1+\frac{\sigma^2_{n-1}}{2})}, \quad \frac{1}{\sigma_n^2}=\alpha+\frac{1}{1+\frac{\sigma_{n-1}^2}{2}}.
\end{equation}
\end{proposition}

\begin{lemma}[Relaxation of variance]\label{L-sexual-reproduction-Gaussian-solution-time-discrete-variance-relaxation}
Assume that $\alpha\in \mathbb{R}_+^*$, consider any $\sigma_0^2\in \mathbb{R}_+^*$ and define the sequence $\{\sigma_n^2\}_{n\in \mathbb{N}}$ by recursion according to the third recursion in \eqref{E-sexual-reproduction-Gaussian-solution-time-discrete-parameters}, {\em i.e.}
\begin{equation}\label{E-sexual-reproduction-Gaussian-solution-time-discrete-variance-recursion}
\frac{1}{\sigma_n^2}=\alpha+\frac{1}{1+\frac{\sigma_{n-1}^2}{2}},
\end{equation}
for every $n\in \mathbb{N}$. Hence, if $\sigma_0^2>\bsigma_\alpha^2$ then $\sigma_n\searrow \bsigma_\alpha^2$ as $n\rightarrow\infty$ and, if $\sigma_0^2<\bsigma_\alpha^2$ then $\sigma_n\nearrow \bsigma_\alpha^2$ as $n\rightarrow\infty$, where $\bsigma_\alpha^2$ is given by \eqref{E-sexual-reproduction-Gaussian-solution-variance-equation}. In addition, we obtain the convergence rates
\begin{equation}\label{E-sexual-reproduction-Gaussian-solution-time-discrete-variance-relaxation}
\vert \sigma_n^2-\bsigma_\alpha^2\vert\leq C_v \br_\alpha^n,
\end{equation}
for every $n\in \mathbb{N}$, where the constant $C_v\in\mathbb{R}_+$ depends on $\sigma_0^2$ and $\alpha$ (we obtain that $C_v=0$ if $\sigma_0^2=\bsigma_\alpha^2$), and the ratio $\br_\alpha\in (0,\frac{1}{2})$ satisfies $\br_\alpha=2\bk_\alpha^2$ and is given explicitly by
\begin{equation}\label{E-variance-contraction-ralpha}
\br_\alpha:=\frac{8}{\left((2\alpha+3)+\sqrt{(2\alpha+1)^2+8\alpha}\right)^2}.
\end{equation}
\end{lemma}

\begin{proof}
$\bullet$ {\sc Step 1}: Monotone convergence of $\{\sigma_n^2\}_{n\in \mathbb{N}}$.\\
Consider the function $f:\mathbb{R}_+\longrightarrow \mathbb{R}$ given by
$$f(x)=\alpha+\frac{1}{1+\frac{1}{2x}},\quad x\in \mathbb{R}_+,$$
and, for simplicity of notation, define
\begin{equation}\label{E-contraction-variance-change-variables}
x_*:=(\bsigma_\alpha^2)^{-1},\quad x_n:=(\sigma_n^2)^{-1},\quad n\in \mathbb{N}.
\end{equation}
Notice that $x_*$ is the unique fixed point of $f$ in $\mathbb{R}_+$ and $\{x_n\}_{n\in \mathbb{N}}$ determines the fixed-point iteration of the map $f$ issued at $x_0=(\sigma_0^2)^{-1}$, that is, $x_n=f(x_{n-1}),\quad n\in \mathbb{N}$. Our goal is to show that $\{x_n\}_{n\in \mathbb{N}}$ converges towards $x_*$ and find convergence rates. By direct computation we obtain $f'(x)=2/(1+2x)^2$, which is above $1$ near $x=0$, and therefore $f$ is not contractive. Hence, we cannot apply the usual Banach contraction principle and a different argument is provided. First, note that  $x\in \mathbb{R}_+\mapsto f(x)$ is strictly increasing and $x\in\mathbb{R}_+^*\mapsto \frac{f(x)}{x}$ is strictly decreasing. Then, we obtain that
\begin{align*}
\begin{aligned}
&x<f(x)<x_*, & & \mbox{if } 0\leq x<x_*,\\
& f(x)=x_*, & & \mbox{if }x=x_*,\\
&x_*<f(x)<x, &  &\mbox{if }x>x_*,
\end{aligned}
\end{align*}
which implies the aforementioned monotonicity properties of $\{x_n\}_{n\in \mathbb{N}}$ and, in addition,
\begin{equation}\label{E-contraction-variance-uniform-bound}
\min\{x_*,x_0\}\leq x_n\leq \max\{x_*,x_0\},
\end{equation}
for any $n\in \mathbb{N}$. In particular, this yields the monotonic convergence of $\{x_n\}_{n\in \mathbb{N}}$ towards $x_*$, or equivalently, the claimed monotonic convergence of $\{\sigma_n^2\}_{n\in \mathbb{N}}$ towards $\bsigma_\alpha^2$ as $n\rightarrow \infty$. 

\medskip

$\bullet$ {\sc Step 2:} Convergence rates.\\
Second, to compute the convergence rates we shall use the special algebraic structure of function $f$, which is the restriction to $\mathbb{R}_+$ of the M\"{o}bius function $M:\mathbb{R}\setminus\{-\frac{1}{2}\}\longrightarrow \mathbb{R}$ given by
$$M(x):=\frac{2(\alpha+1)x+\alpha}{2x+1},\quad x\in \mathbb{R}\setminus\left\{-\frac{1}{2}\right\}.$$
Note that $M$ has two fixed points:
$$
x_{\pm}:=\frac{2\alpha+1\pm\sqrt{(2\alpha+1)^2+8\alpha}}{4},
$$
where we note that $x_+=x_*\in \mathbb{R}_+^*$ and $x_-\in \mathbb{R}^-$. Then, by definition of $\{x_n\}_{n\in \mathbb{N}}$ in \eqref{E-contraction-variance-change-variables} we obtain
\begin{equation}\label{E-contraction-variance-Mobius-1}
x_n-x_+=\frac{2(\alpha+1)x_{n-1}+\alpha}{2x_{n-1}+1}-\frac{2(\alpha+1)x_++\alpha}{2x_++1}=\frac{2}{(2x_{n-1}+1)(2 x_++1)}(x_{n-1}-x_+).
\end{equation}
Since $x_0$ (thus $x_{n-1}$ for initial time steps $n$) can be chosen arbitrary close to $0$, then, the best a priori control that we can have on the prefactor in the right hand side of \eqref{E-contraction-variance-Mobius-1} is
$$0\leq \frac{2}{(2x_{n-1}+1)(2 x_++1)}\leq \frac{2\bsigma_\alpha^2}{2+\bsigma_\alpha^2}=2\bk_\alpha,$$
which gives true contraction in \eqref{E-contraction-variance-Mobius-1} when $\alpha>0$, since $2\bk_\alpha<1$. Nevertheless, such a rate is non-optimal, and we show an alternative approach to achieve a sharper one. Specifically, note that
\begin{equation}\label{E-contraction-variance-Mobius-2}
x_n-x_-=\frac{2(\alpha+1)x_{n-1}+\alpha}{2x_{n-1}+1}-\frac{2(\alpha+1)x_-+\alpha}{2x_-+1}=\frac{2}{(2x_{n-1}+1)(2 x_-+1)}(x_{n-1}-x_-).
\end{equation}
By dividing \eqref{E-contraction-variance-Mobius-1} by \eqref{E-contraction-variance-Mobius-2} and iterating the identity we obtain
\begin{equation}\label{E-contraction-variance-Mobius-3}
\frac{x_n-x_+}{x_n-x_-}=\left(\frac{2x_-+1}{2x_++1}\right)^n\frac{x_0-x_+}{x_0-x_-},
\end{equation}
for every $n\in \mathbb{N}$. In fact, the basis can be restated as follows
$$\frac{2x_-+1}{2x_++1}=\frac{(2\alpha+3)-\sqrt{(2\alpha+1)^2+8\alpha}}{(2\alpha+3)+\sqrt{(2\alpha+1)^2+8\alpha}}=\frac{8}{\left((2\alpha+3)+\sqrt{(2\alpha+1)^2+8\alpha}\right)^2}=\br_\alpha<\frac{1}{2}.$$
Therefore, using the uniform control \eqref{E-contraction-variance-uniform-bound} and \eqref{E-contraction-variance-Mobius-3}, we conclude that
\begin{equation}\label{E-contraction-variance-1}
\vert x_n-x_*\vert=\frac{\vert x_n-x_-\vert}{\vert x_0-x_-\vert}\vert x_0-x_*\vert\br_\alpha^n\leq  \frac{\max\{x_*,x_0\}-x_-}{\min\{x_*,x_0\}-x_-}\vert x_0-x_*\vert\,\br_\alpha^n,
\end{equation}
for every $n\in \mathbb{N}$, thus leading to an improved rate since $\br_\alpha=2\bk_\alpha^2<2\bk_\alpha$. Finally, notice that
\begin{equation}\label{E-contraction-variance-2}
\vert \sigma_n^2-\bsigma_\alpha^2\vert=\sigma_n^2\bsigma_\alpha^2\left\vert x_n-x_*\right\vert\leq \max\{\bsigma_\alpha^2,\sigma_0^2\}^2\left\vert x_n-x_*\right\vert,
\end{equation}
for every $n\in \mathbb{N}$. Hence, joining \eqref{E-contraction-variance-1}, \eqref{E-contraction-variance-2} along with \eqref{E-contraction-variance-uniform-bound} ends the proof.
\end{proof}

\begin{figure}[t]
\centering
\begin{tikzpicture}
\begin{axis}[
  axis x line=middle, axis y line=middle,
  every axis x label/.style={at={(current axis.right of origin)},anchor=west},
  xmin=0, xmax=5, xtick={0,...,5}, xlabel=$\quad\alpha$,
  ymin=0, ymax=0.6, ytick={0,0.25,0.5}, ylabel=$\br_\alpha$,
]
\addplot [
    domain=0:5, 
    samples=200, 
    color=black,
    line width=0.4mm,
]
{8/(pow((2*x+3)+pow(pow(2*x+1,2)+8*x,1/2),2))};
\addplot [mark=*, mark size=1.5pt] coordinates {(0,0.5)};
\end{axis}
\end{tikzpicture}
\caption{The ratio $\br_\alpha$ against parameter $\alpha$.}
\label{fig:decay-variances}
\end{figure}
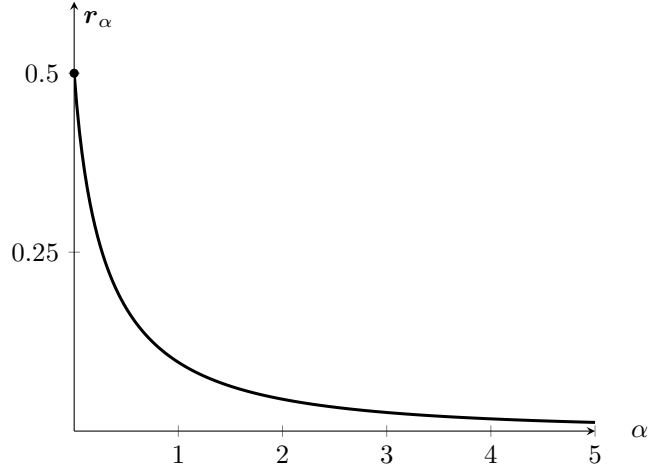

In particular, for Gaussian initial data we recover an explicit particular case of the general relaxation result in Theorem \ref{T-main}.

\begin{corollary}[Relaxation of Gaussian solutions]\label{C-sexual-reproduction-Gaussian-solution-time-discrete-relaxation}
Assume that $\alpha\in \mathbb{R}_+^*$, consider any $m_0\in \mathbb{R}_+$, $\mu_0\in \mathbb{R}^d$ and $\sigma_0^2\in \mathbb{R}_+$, and set the Gaussian initial datum $F_0=m_0\,G_{\mu_0,\sigma_0}$. Hence, the Gaussian solution $\{F_n\}_{n\in \mathbb{N}}$ to the time-discrete problem \eqref{E-sexual-reproduction-time-discrete} verifies that the growth rates $\Vert F_n\Vert_{L^1(\mathbb{R}^d)}/\Vert F_{n-1}\Vert_{L^1(\mathbb{R}^d)}$ relax towards $\blambda_\alpha$ and the normalized profiles $F_n/\Vert F_n\Vert_{L^1(\mathbb{R}^d)}$ relax towards $\bF_\alpha$, where $(\blambda_\alpha,\bF_\alpha)$ is the unique Gaussian solution \eqref{E-sexual-reproduction-Gaussian-solution} to the eigenproblem \eqref{E-eigenproblem-time-discrete} as proved in Proposition \ref{P-sexual-reproduction-Gaussian-solution}. Specifically, for any $\varepsilon\in \mathbb{R}_+^*$
\begin{align*}
\mathcal{D}_{\rm KL}\left(\left.\frac{F_n}{\Vert F_n\Vert_{L^1(\mathbb{R}^d)}}\right\Vert \bF_\alpha\right)&\leq C_{\mu,\varepsilon}(\blambda_\alpha^{4/d}+\varepsilon)^n+C_v\br_\alpha^n,\\
\left\vert\frac{\Vert F_n\Vert_{L^1(\mathbb{R}^d)}}{\Vert F_{n-1}\Vert_{L^1(\mathbb{R}^d)}}-\blambda_\alpha\right\vert&\leq C_{\mu,\varepsilon}(\blambda_\alpha^{4/d}+\varepsilon)^n+C_v\br_\alpha^n,
\end{align*}
where $\mathcal{D}_{\rm KL}$ is the Kullback-Leibler divergence (or relative entropy), $\br_\alpha$ is given by \eqref{E-variance-contraction-ralpha} in Lemma \ref{L-sexual-reproduction-Gaussian-solution-time-discrete-variance-relaxation}, $C_{\mu,\varepsilon}\in \mathbb{R}_+$ depends on $\alpha$, $\mu_0$, $\sigma_0^2$ and $\varepsilon$, and $C_v\in \mathbb{R}_+$ depends only $\alpha$ and $\sigma_0^2$. In fact, $C_{\mu,\varepsilon}=0$ if $\mu_0=0$ and $C_v=0$ if $\sigma_0^2=\bsigma_\alpha^2$.
\end{corollary}

Again, in this case the computations of the Kullback-Leibler divergence becomes explicit, and it is based on the general formula below for the divergence between two Gaussian density functions:
\begin{equation}\label{E-KL-Gaussian}
\mathcal{D}_{\rm KL}(G_{\mu_1,\Sigma_1}\Vert G_{\mu_2,\Sigma_2})=\frac{1}{2}(\mu_2-\mu_1)^\top\Sigma_2^{-1}(\mu_2-\mu_1)+\frac{1}{2}{\rm trace}(\Sigma_1\Sigma_2^{-1})-\frac{d}{2}+\frac{1}{2}\log\det (\Sigma_2\Sigma_1^{-1}),
\end{equation}
for every $\mu_1,\mu_2\in \mathbb{R}^d$ and any positive non-singular matrices $\Sigma_1,\Sigma_2\in \mathbb{R}^{d\times d}$.

\begin{proof}[Proof of Corollary \ref{C-sexual-reproduction-Gaussian-solution-time-discrete-relaxation}]
First, note that when $\alpha\in \mathbb{R}_+^*$, then the parameters $m_n$, $\mu_n$ and $\sigma_n^2$ in \eqref{E-sexual-reproduction-Gaussian-solution-time-discrete-parameters} verify
$$\lim_{n\rightarrow \infty}\frac{m_n}{m_{n-1}}=\blambda_\alpha,\quad \lim_{n\rightarrow \infty}\mu_n=0,\quad \lim_{n\rightarrow \infty}\sigma_n^2=\bsigma_\alpha^2.$$
In the sequel, we quantify the rates of convergence. Since $\sigma_n^2$ has already been studied in Lemma \ref{L-sexual-reproduction-Gaussian-solution-time-discrete-variance-relaxation} we shall focus only on $\mu_n$ and $m_n$. Let us fix any arbitrarily small $\varepsilon\in \mathbb{R}_+^*$. On the one hand note that
$$\frac{\vert \mu_n\vert }{\vert \mu_{n-1}\vert }-\blambda_\alpha^{2/d}=\frac{1}{1+\alpha\big(1+\frac{\sigma_{n-1}^2}{2}\big)}-\frac{1}{1+\alpha\big(1+\frac{\bsigma_\alpha^2}{2}\big)},$$
for every $n\in \mathbb{N}$. By the mean value theorem and Lemma \ref{L-sexual-reproduction-Gaussian-solution-time-discrete-variance-relaxation} we obtain
$$\left\vert\frac{\vert \mu_n\vert }{\vert \mu_{n-1}\vert }-\blambda_\alpha^{2/d}\right\vert\leq C_v\br_\alpha^n,$$
for an appropriate $C_v\in \mathbb{R}_+^*$ depending on $\alpha$ and $\sigma_0^2$. Since $\blambda_\alpha\in (0,1)$ then d'Alembert's ratio test readily shows that $\mu_n$ relaxes to zero as a geometric sequence. Indeed, note that
$$\vert \mu_n\vert=\left(\frac{\vert \mu_n\vert}{\vert \mu_{n-1}\vert}-\blambda_\alpha^{2/d}\right)\vert \mu_{n-1}\vert+\blambda_\alpha^{2/d}\vert \mu_{n-1}\vert\leq (\blambda_\alpha^{2/d}+C_v\br_\alpha^n)\vert \mu_{n-1}\vert ,$$
for any $n\in \mathbb{N}$. By an inductive argument, this yields
\begin{equation}\label{E-center-mass-contraction}
\vert \mu_n\vert\leq \prod_{k=1}^n(\blambda_\alpha^{2/d}+C_v\br_\alpha^n)\vert \mu_0\vert\lesssim C_\varepsilon(\blambda_\alpha^{2/d}+\varepsilon)^n,
\end{equation}
for sufficiently large $C_{\mu,\varepsilon}\in \mathbb{R}^d$, where we have used that $\br_\alpha\in (0,1)$ to absorb $C_v\br_\alpha^n$ in an $\varepsilon$-small term. On the other hand, note that
\begin{align*}
\frac{m_n}{m_{n-1}}-\blambda_\alpha&=\frac{e^{-\frac{1}{2}\frac{\alpha\vert \mu_{n-1}\vert^2}{1+\alpha\big(1+\frac{\sigma_{n-1}^2}{2}\big)}}}{(1+\alpha(1+\frac{\sigma_{n-1}^2}{2}))^{d/2}}-\frac{1}{(1+\alpha(1+\frac{\bsigma_\alpha^2}{2}))^{d/2}}\\
&=\frac{e^{-\frac{1}{2}\frac{\alpha\vert \mu_{n-1}\vert^2}{1+\alpha\big(1+\frac{\sigma_{n-1}^2}{2}\big)}}-1}{(1+\alpha(1+\frac{\sigma_{n-1}^2}{2}))^{d/2}}+\left(\frac{1}{(1+\alpha(1+\frac{\sigma_{n-1}^2}{2}))^{d/2}}-\frac{1}{(1+\alpha(1+\frac{\bsigma_\alpha^2}{2}))^{d/2}}\right).
\end{align*}
By the mean value theorem,
\begin{align*}
\left\vert\frac{m_n}{m_{n-1}}-\blambda_\alpha\right\vert\leq \frac{\alpha \vert \mu_{n-1}\vert^2}{2(1+\alpha)^{\frac{d}{2}+1}}+\frac{d\alpha}{4}\frac{\vert \sigma_{n-1}^2-\bsigma_\alpha^2\vert}{(1+\alpha)^{\frac{d}{2}+1}}.
\end{align*}
Therefore, using contraction of mean \eqref{E-center-mass-contraction} and contraction of variance \eqref{E-sexual-reproduction-Gaussian-solution-time-discrete-variance-relaxation} in Lemma \ref{L-sexual-reproduction-Gaussian-solution-time-discrete-variance-relaxation} entail the rate of convergence for the $m_n/m_{n-1}$. Finally, given that both $F_n/\Vert F_n\Vert_{L^1(\mathbb{R}^d)}$ and $\bF_\alpha$ are Gaussian functions, the relative entropy can be computed explicitly through \eqref{E-KL-Gaussian}. Specifically,
$$
\mathcal{D}_{\rm KL}\left(\left.\frac{F_n}{\Vert F_n\Vert_{L^1(\mathbb{R}^d)}}\right\Vert \bF_\alpha\right)=\frac{\vert \mu_n\vert^2}{2\bsigma_\alpha^2}+\frac{d}{2}\left(\frac{\sigma_n^2}{\bsigma_\alpha^2}-1\right)-\frac{d}{2}\log\left(\frac{\sigma_n^2}{\bsigma_\alpha^2}\right).
$$
Hence, using again \eqref{E-center-mass-contraction} and \eqref{E-sexual-reproduction-Gaussian-solution-time-discrete-variance-relaxation} concludes our result.
\end{proof}

The proof of Theorem \ref{T-main} for generic non-Gaussian initial data $F_0\in \mathcal{M}_+(\mathbb{R}^d)$ will be the core of this paper and we postpone it to Section \ref{S-ergodicity}.

\begin{remark}[Optimality of convergence rates]\label{R-rough-convergence-rates}
As illustrated in Figure \ref{fig:non-optimal-rates}, the convergence rate of the normalized profiles obtained in our main Theorem \ref{T-main} (blue line) is sharp, compared to the explicit one found in Corollary \ref{C-sexual-reproduction-Gaussian-solution-time-discrete-relaxation} (red line) for the special class of Gaussian solutions. In fact, one can easily check that the identity $\blambda_\alpha^{4/d}=(2\bk_\alpha)^2$ holds. However, there is a mismatch between the convergence rate of the rate of growth of mass in Theorem \ref{T-main} (orange line) and the sharp rate for Gaussian solutions found in Corollary \ref{C-sexual-reproduction-Gaussian-solution-time-discrete-relaxation} (again red line). On the one hand, it stands to reason that the relative entropy has a certain quadratic structure, whilst the rate of growth of mass is related to $L^1$ norms, and then it is not quadratic. We will see that more clearly in the proof of Theorem \ref{T-main} in Section \ref{S-ergodicity}, where we use explicitly the following relation
$$\left\vert \frac{\Vert F_n\Vert_{L^1(\mathbb{R}^d)}}{\Vert F_{n-1}\Vert_{L^1(\mathbb{R}^d)}}-\blambda_\alpha\right\vert\lesssim \sqrt{\mathcal{D}_{\rm KL}\left(\left.\frac{F_n}{\Vert F_n\Vert_{L^1(\mathbb{R}^d)}}\right\Vert \bF_\alpha\right)},$$
based on Pinsker's inequality. However, a certain quadratic structure is still present in the rate of convergence of $m_n/m_{n-1}$ in Corollary \ref{C-sexual-reproduction-Gaussian-solution-time-discrete-relaxation}, which may become explicit for generic solutions to \eqref{E-sexual-reproduction-time-discrete} if one finds the hidden cancellations. However, for simplicity we do not address this technical detail here.
\end{remark}

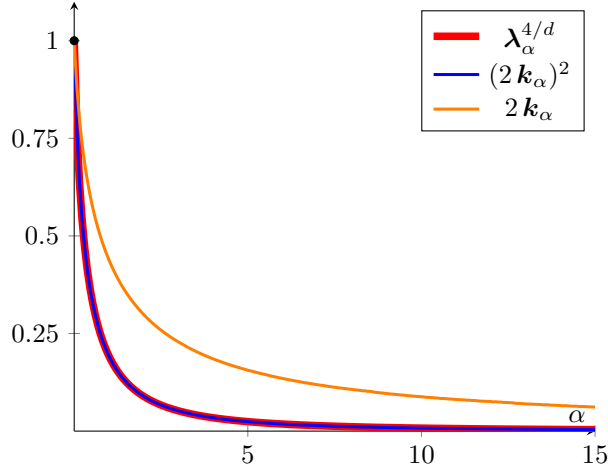
\begin{figure}[t]
\centering
\begin{tikzpicture}
\begin{axis}[
  axis x line=middle, axis y line=middle,
  xmin=0, xmax=15, xtick={0,5,10,15}, xlabel=$\alpha$,
  ymin=0, ymax=1.1, ytick={0,0.25,0.5,0.75,1},
]
\addplot [
    domain=0:15, 
    samples=200, 
    color=red,
    line width=1mm,
]
{1/pow(1+x*(1+(sqrt((2*x+1)^2+8*x)-(2*x+1))/(4*x)),4/2)};
\addlegendentry{$\blambda_\alpha^{4/d}$}
\addplot [
    domain=-0:15, 
    samples=200, 
    color=blue,
    line width=0.4mm,
]
{pow(2*(((3+2*x)-pow(pow(3+2*x,2)-8,1/2))/4),2)};
\addlegendentry{$(2\bk_\alpha)^2$}
\addplot [
    domain=0:15, 
    samples=200, 
    color=orange,
    line width=0.4mm,
]
{pow(2*(((3+2*x)-pow(pow(3+2*x,2)-8,1/2))/4),1)};
\addlegendentry{$2\bk_\alpha$}
\addplot [mark=*, mark size=1.5pt] coordinates {(0,1)};
\end{axis}
\end{tikzpicture}
\caption{Comparison of convergence rates in Theorem \ref{T-main} and Corollary \ref{C-sexual-reproduction-Gaussian-solution-time-discrete-relaxation}}
\label{fig:non-optimal-rates}
\end{figure}

\section{Some properties of the operator $\mathcal{T}$}\label{S-properties-operator-T}

First, note that for any $F\in \mathcal{M}_+(\mathbb{R}^d)$ we obtain that $\mathcal{T}[F]\in W^{k,1}(\mathbb{R}^d)\cap W^{k,\infty}(\mathbb{R}^d)$ for each $k\in \mathbb{N}$ thanks to the fact that $F$ has finite mass and the Gaussian mixing kernel $G$ in \eqref{E-mixing-Gaussian} is smooth and has bounded and integrable derivatives of any order. In particular, solutions $\{F_n\}_{n\in \mathbb{N}}$ to the time-discrete problem \eqref{E-sexual-reproduction-time-discrete} become instantaneously smooth after the first generation for any generic initial datum $F_0\in \mathcal{M}_+(\mathbb{R}^d)$. In the sequel, we will derive a quantitative control on the emergence of Gaussian tails for $\mathcal{T}[F]$. In addition, we will quantify the propagation of quadratic and exponential moments for the normalized profile $\mathcal{T}[F]/\Vert \mathcal{T}[F]\Vert_{L^1(\mathbb{R}^d)}$. Both {\em a priori} estimates will be required later in Section \ref{S-ergodicity}.

\subsection{Emergence of Gaussian tails}

\begin{lemma}[Emergence of tails]\label{L-sexual-reproduction-Gaussian-tails}
Assume that $\alpha\in \mathbb{R}_+^*$, consider $F\in \mathcal{M}_+(\mathbb{R}^d)$ and set $\overline{\sigma}^2,\underline{\sigma}^2\in \mathbb{R}_+^*$ by
$$\overline{\sigma}^2:=\frac{1}{\alpha},\quad \underline{\sigma}^2\in \left(0,\frac{1}{1+\alpha}\right).$$
Then, the following properties are fulfilled:
\begin{enumerate}[label=(\roman*)]
\item {\bf (Upper control of tails)} There exists $C=C(\alpha,F)>0$ such that
\begin{equation}\label{E-sexual-reproduction-operator-upper-bound-Gaussian}
\mathcal{T}[F](x)\leq C\,G_{0,\overline{\sigma}^2}(x),
\end{equation}
for every $x\in \mathbb{R}^d$.
\item {\bf (Lower control of tails)} There exists $c=c(\alpha,\underline{\sigma}^2,F)>0$ such that
\begin{equation}\label{E-sexual-reproduction-operator-lower-bound-Gaussian}
\mathcal{T}[F](x)\geq c\,G_{0,\underline{\sigma}^2}(x),
\end{equation}
for every $x\in \mathbb{R}^d$.
\item {\bf (Control of log-derivative)} Assume that $F$ is absolutely continuous with respect to the Lebesgue measure and it has Gaussian tail, {\em i.e.}, there exists $\sigma^2\in \mathbb{R}_+^*$ and $C'\in \mathbb{R}_+^*$ such that $F(x)\leq C' G_{0,\sigma^2}$ for every $x\in \mathbb{R}^d$. Then, there exists $C''=C''(\alpha,\underline{\sigma}^2,\sigma^2,C',F)>0$ such that
\begin{equation}\label{E-sexual-reproduction-operator-log-derivative}
\vert \nabla \log \mathcal{T}[F](x)\vert\leq C''(1+\vert x\vert),
\end{equation}
for every $x\in \mathbb{R}^d$.
\end{enumerate}
\end{lemma}

\begin{proof} First, notice that $F\in \mathcal{M}_+(\mathbb{R}^d)$ is any generic distribution but $G\in L^\infty(\mathbb{R}^d)$. Then, $\mathcal{B}[F]\in L^\infty(\mathbb{R}^d)$ and, by definition of $\mathcal{T}$ in \eqref{E-sexual-reproduction-T-operator}, we obtain
$$
\mathcal{T}[F](x)\leq \Vert \mathcal{B}[F]\Vert_{L^\infty(\mathbb{R}^d)} e^{-\frac{\alpha}{2}\vert x\vert^2}= \Vert \mathcal{B}[F]\Vert_{L^\infty(\mathbb{R}^d)}\left(2\pi\alpha^{-1}\right)^{d/2} G_{0,\overline{\sigma}^2}(x),
$$
for every $x\in \mathbb{R}^d$. Hence, \eqref{E-sexual-reproduction-operator-upper-bound-Gaussian} holds for appropriate $C>0$. Second, note that
\begin{align*}
\mathcal{T}[F](x)&=\frac{e^{-\frac{\alpha}{2}\vert x\vert^2}}{(2\pi)^{d/2}\Vert F\Vert_{\mathcal{M}_+(\mathbb{R}^d)}}\int_{\mathbb{R}^{2d}}\exp\left(-\frac{1}{2}\left\vert x-\frac{x_1+x_2}{2}\right\vert^2 \right) F(dx_1)F(dx_2)\\
&=\frac{e^{-\frac{\alpha}{2}\vert x\vert^2}}{(2\pi)^{d/2}\Vert F\Vert_{\mathcal{M}_+(\mathbb{R}^d)}}\int_{\mathbb{R}^{2d}}\exp\left(-\frac{1}{2}\vert x\vert^2-\frac{1}{8}\vert x_1+x_2\vert^2+\frac{1}{2}x\cdot (x_1+x_2)\right)F(dx_1)F(dx_2)\\
&\geq \frac{e^{-\frac{\alpha}{2}\vert x\vert^2}}{(2\pi)^{d/2}\Vert F\Vert_{\mathcal{M}_+(\mathbb{R}^d)}}\int_{\mathbb{R}^{2d}} \exp\left(-\frac{2+\varepsilon}{4}\vert x\vert^2-\frac{2+\varepsilon}{8\varepsilon}\vert x_1+x_2\vert^2\right)F(dx_1)F(dx_2)\\
&\geq \frac{e^{-\frac{\alpha}{2}\vert x\vert^2}}{(2\pi)^{d/2}\Vert F\Vert_{\mathcal{M}_+(\mathbb{R}^d)}}\int_{\mathbb{R}^{2d}} \exp\left(-\frac{2+\varepsilon}{4}\vert x\vert^2-\frac{2+\varepsilon}{4\varepsilon}\vert x_1\vert^2-\frac{2+\varepsilon}{4\varepsilon}\vert x_2\vert^2\right)F(dx_1)F(dx_2),
\end{align*}
for every $\varepsilon>0$, where in the third and last lines we have used Cauchy--Schwarz's and Young's inequalities. Consequently, we obtain that
$$
\mathcal{T}[F](x)\geq (2\pi)^{d/2}\left(\frac{2\varepsilon}{2+\varepsilon}\right)^d\frac{\Vert G_{0,\frac{2\varepsilon}{2+\varepsilon}}F\Vert_{\mathcal{M}_+(\mathbb{R}^d)}^2}{\Vert F\Vert_{\mathcal{M}_+(\mathbb{R}^d)}} e^{-\frac{\alpha}{2}\vert x\vert^2}e^{-\frac{2+\varepsilon}{4}\vert x\vert^2},
$$
for every $x\in \mathbb{R}^d$. Taking $\varepsilon>0$ small enough and an appropriate constant $c>0$ yields \eqref{E-sexual-reproduction-operator-lower-bound-Gaussian}. Finally, taking derivatives on \eqref{E-sexual-reproduction-T-operator} we obtain
\begin{align*}\nabla\mathcal{T}[F](x)&=-\alpha\,x\,\mathcal{T}[F](x)-\frac{e^{-\frac{\alpha}{2}|x|^2}}{\Vert F\Vert_{L^1(\mathbb{R}^d)}}\int_{\mathbb{R}^{2d}}\left(x-\frac{x_1+x_2}{2}\right)G\left(x-\frac{x_1+x_2}{2}\right)\,F(dx_1)\,F(dx_2)\\
&=-(1+\alpha)\,x\,\mathcal{T}[F](x)-\frac{e^{-\frac{\alpha}{2}|x|^2}}{\Vert F\Vert_{L^1(\mathbb{R}^d)}}\int_{\mathbb{R}^{2d}}\left(\frac{x_1+x_2}{2}\right)G\left(x-\frac{x_1+x_2}{2}\right)\,F(dx_1)\,F(dx_2).
\end{align*}
Therefore, dividing by $\mathcal{T}[F](x)$ and taking norms yield the following estimate
\begin{align}\label{E-sexual-reproduction-evolution-log-derivative-pre-1}
\vert \nabla \log \mathcal{T}[F](x)\vert\leq (1+\alpha)\vert x\vert +\frac{1}{\sqrt{2}\Vert F\Vert_{L^1(\mathbb{R}^d)}}\frac{e^{-\frac{\alpha}{2}\vert x\vert^2}}{\mathcal{T}[F](x)}\int_{\mathbb{R}^{2d}}\vert (x_1,x_2)\vert\,G\left(x-\frac{x_1+x_2}{2}\right) F(dx_1)F(dx_2),
\end{align}
for every $x\in \mathbb{R}^d$. Consider any $R>0$ and split the integral in the right hand side of \eqref{E-sexual-reproduction-evolution-log-derivative-pre-1} into the subsets $\{(x_1,x_2)\in \mathbb{R}^{2d}:\,\vert (x_1,x_2)\vert\leq \sqrt{2}R\vert x\vert\}$ and $\{(x_1,x_2)\in \mathbb{R}^{2d}:\,\vert (x_1,x_2)\vert> \sqrt{2}R\vert x\vert\}$. For the first subset, we use the definition \eqref{E-sexual-reproduction-T-operator} of the operator $\mathcal{T}$. For the second subset, we use the uniform bound of $G$ along with the assumptions on $F$ and the preceding lower bound \eqref{E-sexual-reproduction-operator-lower-bound-Gaussian} of $\mathcal{T}[F]$. Then, we obtain
\begin{align*}
\vert \nabla \log \mathcal{T}[F](x)\vert&\leq (1+\alpha+R)\vert x\vert+\frac{C^2\underline{\sigma}^d e^{-\frac{\alpha}{2}\vert x\vert^2}e^{\frac{1}{2\underline{\sigma}^2}\vert x\vert^2}}{\sqrt{2}c\Vert F\Vert_{L^1(\mathbb{R}^d)}} \int_{\vert (x_1,x_2)\vert >\sqrt{2}R\vert x\vert} \vert (x_1,x_2)\vert e^{-\frac{1}{2\sigma^2}\vert (x_1,x_2)\vert^2}\,dx_1\,dx_2\\
&=(1+\alpha+R)\vert x\vert +\frac{C^2\underline{\sigma}^d e^{-\frac{\alpha}{2}\vert x\vert^2}e^{\frac{1}{2\underline{\sigma}^2}\vert x\vert^2}}{\sqrt{2}c\Vert F\Vert_{L^1(\mathbb{R}^d)}} \int_{\sqrt{2}R\vert x\vert}^{+\infty}r^{2d} e^{-\frac{1}{2\sigma^2}r^2}\,dr\\
&\leq  (1+\alpha+R) \vert x\vert +\frac{C^2\underline{\sigma}^d e^{-\frac{\alpha}{2}\vert x\vert^2}e^{\frac{1}{2\underline{\sigma}^2}\vert x\vert^2}}{\sqrt{2}c\Vert F\Vert_{L^1(\mathbb{R}^d)}} \left(\frac{8d-4}{e}\right)^{d-\frac{1}{2}}\sigma^{2d-1}\int_{\sqrt{2}R\vert x\vert}^{+\infty}re^{-\frac{1}{4\sigma^2}r^2}\,dr\\
&= (1+\alpha+R)\vert x\vert +\frac{\sqrt{2}C^2\underline{\sigma}^d \sigma^{2d+1}}{c\Vert F\Vert_{L^1(\mathbb{R}^d)}}\left(\frac{8d-4}{e}\right)^{d-\frac{1}{2}} e^{-\frac{1}{2}\left(\alpha+\frac{R^2}{\sigma^2}-\frac{1}{\underline{\sigma}^2}\right)\vert x\vert^2},
\end{align*}
for every $x\in \mathbb{R}^d$, where in the third line we have used the inequality
$$r^{2d-1}\leq \left(\frac{2d-1}{e}\right)^{d-\frac{1}{2}}\left(\frac{\sigma}{\varepsilon}\right)^{2d-1} \exp\left(\frac{\varepsilon}{2\sigma^2}r^2\right),$$
for every $r>0$ and $\varepsilon>0$ in the particular case of $\varepsilon=1/2$. Taking $R>0$ large enough ends the proof.
\end{proof}

The previous result can be iterated to obtain similar properties for solutions $\{F_n\}_{n\in \mathbb{N}}$ to the time-discrete problem \eqref{E-sexual-reproduction-time-discrete} issued at a generic initial datum $F_0\in \mathcal{M}_+(\mathbb{R}^d)$. In particular, we note that Gaussian tails emerge instantaneously after the first generation $n=1$ and linear growth of the log-derivative is guaranteed after the second generation $n=2$.

\begin{corollary}[Propagation of tails]\label{C-sexual-reproduction-Gaussian-tails}
Assume that $\alpha\in \mathbb{R}_+^*$, consider the solution $\{F_n\}_{n\in \mathbb{N}}$ of \eqref{E-sexual-reproduction-time-discrete} issued at a generic initial datum $F_0\in \mathcal{M}_+(\mathbb{R}^d)$ and set $\overline{\sigma}^2_1,\underline{\sigma}^2_1\in \mathbb{R}_+^*$ by
$$\overline{\sigma}^2_1:=\frac{1}{\alpha},\quad \underline{\sigma}^2_1\in \left(0,\frac{1}{1+\alpha}\right).$$
Define the associated sequences of variances $\{\overline{\sigma}_n^2\}_{n\in \mathbb{N}}$ and $\{\underline{\sigma}_n^2\}_{n\in \mathbb{N}}$ by the following recursive relations
\begin{equation}\label{E-variances-barriers}
\frac{1}{\overline{\sigma}_{n+1}^2}:=\alpha+\frac{1}{1+\frac{\overline{\sigma}^2_n}{2}},\quad \frac{1}{\underline{\sigma}_{n+1}^2}:=\alpha+\frac{1}{1+\frac{\underline{\sigma}^2_n}{2}},
\end{equation}
for every $n\in \mathbb{N}$. Then, the following properties are fulfilled:
\begin{enumerate}[label=(\roman*)]
\item {\bf (Upper control of tails of $F_n$)} There exist $C_n=C_n(\alpha,F_0)>0$ such that
\begin{equation}\label{E-sexual-reproduction-evolution-upper-bound-Gaussian}
F_n(x)\leq C_n\,G_{0,\overline{\sigma}_n^2}(x),
\end{equation}
for each $x\in \mathbb{R}^d$ and every $n\in \mathbb{N}$.
\item {\bf (Lower control of tails of $F_n$)} There exist $c_n=c_n(\alpha,\underline{\sigma}_1^2,F_0)>0$ such that
\begin{equation}\label{E-sexual-reproduction-evolution-lower-bound-Gaussian}
F_n(x)\geq c_n\,G_{0,\underline{\sigma}_n^2}(x),
\end{equation}
for each $x\in \mathbb{R}^d$ and every $n\in \mathbb{N}$.
\item {\bf (Control of log-derivative of $F_n$)} There exists $D_n=D_n(\alpha,\underline{\sigma}_1^2,F_0)$ such that
\begin{equation}\label{E-sexual-reproduction-evolution-log-derivative}
\vert \nabla\log F_n(x)\vert\leq D_n(1+\vert x\vert),
\end{equation}
for each $x\in \mathbb{R}^d$ and every $n\geq 2$.
\end{enumerate}
\end{corollary}

\begin{proof} Note that once \eqref{E-sexual-reproduction-evolution-upper-bound-Gaussian} and \eqref{E-sexual-reproduction-evolution-lower-bound-Gaussian} are proved, we can readily apply \eqref{E-sexual-reproduction-operator-log-derivative} in Lemma \ref{L-sexual-reproduction-Gaussian-tails} to $F=F_{n-1}$ with $n\geq 2$  (which satisfies the required upper control by a Gaussian function) and we recover the estimate \eqref{E-sexual-reproduction-evolution-log-derivative} for the log-derivative of $F_n=\mathcal{T}[F_{n-1}]$. Then, we just focus on proving estimates \eqref{E-sexual-reproduction-evolution-upper-bound-Gaussian} and \eqref{E-sexual-reproduction-evolution-lower-bound-Gaussian} through an inductive argument. First, note that Lemma \ref{L-sexual-reproduction-Gaussian-tails} applied to $F=F_0$ yields  \eqref{E-sexual-reproduction-evolution-upper-bound-Gaussian} and \eqref{E-sexual-reproduction-evolution-lower-bound-Gaussian} with $n=1$. Let us assume that  \eqref{E-sexual-reproduction-evolution-upper-bound-Gaussian} and \eqref{E-sexual-reproduction-evolution-lower-bound-Gaussian} hold for some $n\in \mathbb{N}$ and let us prove it for $n+1$. Specifically, using the induction hypothesis we obtain that
\begin{align*}
F_{n+1}(x)&=e^{-\frac{\alpha}{2}\vert x\vert^2} \mathcal{B}[F_n](x)\\
&\leq \frac{C_n^2}{\Vert F_n\Vert_{L^1(\mathbb{R}^d)}}e^{-\frac{\alpha}{2}\vert x\vert^2} \left(G\left(\frac{\cdot}{2}\right)*G_{0,\overline{\sigma}_n^2}*G_{0,\overline{\sigma}_n^2}\right)(2x)\\
&= \frac{C_n^2}{\Vert F_n\Vert_{L^1(\mathbb{R}^d)}}e^{-\frac{\alpha}{2}\vert x\vert^2} G_{0,1+\frac{\overline{\sigma}_n^2}{2}}(x)\leq C_{n+1} G_{0,\overline{\sigma}_{n+1}^2}(x),
\end{align*}
for every $x\in \mathbb{R}^d$ and appropriate $C_{n+1}$, where we have used Lemma \ref{L-convolution-gaussians} for the stability of Gaussian under convolutions and the definition \eqref{E-variances-barriers} of $\overline{\sigma}_{n+1}^2$. This proves \eqref{E-sexual-reproduction-evolution-upper-bound-Gaussian} for $n+1$ and a similar argument yields the lower estimate \eqref{E-sexual-reproduction-evolution-lower-bound-Gaussian}.
\end{proof}

By Lemma \ref{L-sexual-reproduction-Gaussian-solution-time-discrete-variance-relaxation} we note that the sequence of variances $\{\overline{\sigma}_n^2\}_{n\in \mathbb{N}}$ and $\{\underline{\sigma}_n^2\}_{n\in \mathbb{N}}$ relax towards the asymptotic value $\bsigma_\alpha^2$ with a geometric convergence rate. This is consistent with our main result in Theorem \ref{T-main} and Corollary \ref{C-main}. In fact, this suggests that any solution $\{F_n\}_{n\in \mathbb{N}}$ to the time-discrete problem \eqref{E-sexual-reproduction-time-discrete} must relax towards the asymptotic profile $\bF_\alpha$. In addition, recall that for any solution $(\lambda,F)$ of the eigenproblem \eqref{E-eigenproblem-time-discrete} we recover a particular solution of the time-discrete problem \eqref{E-sexual-reproduction-time-discrete} via the ansatz \eqref{E-ansatz-discrete}, {\em i.e.}, $F_n=\lambda^n F$. Hence, we expect that $(\blambda_\alpha,\bF_\alpha)$ must indeed be the unique solution to the eigenvalue problem \eqref{E-eigenproblem-time-discrete}. However, we are still far from characterizing the full long-term dynamics for the solution $\{F_n\}_{n\in \mathbb{N}}$ in Theorem \ref{T-main} and the uniqueness result in Corollary \ref{C-main}. Namely, the above coefficients $C_n$ and $c_n$ contains crucial information about the balance of mass and their values in Corollary \ref{C-sexual-reproduction-Gaussian-tails} are not necessarily optimal. Indeed, note that they are given by explicit recursive formulas
$$C_{n+1}=(1-\alpha\overline{\sigma}_n^2)^{d/2}\frac{C_n^2}{\Vert F_n\Vert_{L^1(\mathbb{R}^d)}},\quad c_{n+1}=(1-\alpha\overline{\sigma}_n^2)^{d/2}\frac{c_n^2}{\Vert F_n\Vert_{L^1(\mathbb{R}^d)}},$$
but they require further knowledge about the behavior of $\Vert F_n\Vert_{L^1(\mathbb{R}^d)}$. In the next paragraph, we overcome this important issue of scaling factor by focusing on renormalized profiles, for which we prove uniform propagation of moments.

\subsection{Propagation of quadratic and exponential moments}\label{SS-propagation-moments}

In this section we analyze the propagation of quadratic and exponential moments along the normalized profiles $F_n/\Vert F_n\Vert_{L^1(\mathbb{R}^d)}$, where $\{F_n\}_{n\in \mathbb{N}}$ is a solution to \eqref{E-sexual-reproduction-time-discrete}. As we anticipated in Section \ref{SS-strategy}, we need to move from uniform estimates of the high-dimensional integral \eqref{eq:iterations alpha>0 log-derivative} to estimates in an average sense. This will be precisely the point in proof of the main Theorem \ref{T-main} in Section \ref{SS-proof-main-theorem} where a suitable control of moments will be crucial, as sketched in the overall map in Figure \ref{fig:overall-map-proof}. Before addressing the case of generic initial data $F_0\in \mathcal{M}_+(\mathbb{R}^d)$, we illustrate the particular explicit case of Gaussian initial data $F_0=m_0\,G_{\mu_0,\sigma_0^2}$.

\begin{remark}[Gaussian case]
Consider any $m_0\in \mathbb{R}_+$, $\mu_0\in \mathbb{R}^d$ and $\sigma_0^2\in \mathbb{R}_+$, set $\theta\in \mathbb{R}_+^*$ with $\theta<\frac{1}{2\max\{\sigma_0^2,\bsigma_\alpha^2\}}$ and consider the solution $\{F_n\}_{n\in \mathbb{N}}$ to the time-discrete problem \eqref{E-sexual-reproduction-time-discrete} with initial datum $F_0=m_0G_{\mu_0,\sigma_0^2}$. Then, by direct calculation we obtain
\begin{align*}
\int_{\mathbb{R}^d}\vert x\vert^2 \frac{F_n(x)}{\Vert F_n\Vert_{L^1(\mathbb{R}^d)}}\,dx&=\vert \mu_n\vert^2+\sigma_n^2,\\
\int_{\mathbb{R}^d}e^{\theta\vert x\vert^2}\frac{F_n(x)}{\Vert F_n\Vert_{L^1(\mathbb{R}^d)}}\,dx&=\frac{e^{\frac{\theta}{1-2\theta\sigma_n^2}\vert \mu_n\vert^2}}{(1-2\theta\sigma_n^2)^{d/2}},
\end{align*}
Here, $\{\mu_n\}_{n\in \mathbb{N}}$ and $\{\sigma_n^2\}_{n\in \mathbb{N}}$ are the mean and variance of $\{F_n\}_{n\in \mathbb{N}}$ according to formula \eqref{E-sexual-reproduction-Gaussian-solution-time-discrete-parameters} in Proposition \ref{P-sexual-reproduction-Gaussian-solution-time-discrete}. As studied in the proof of Corollary \ref{C-sexual-reproduction-Gaussian-solution-time-discrete-relaxation} we have the uniform bounds
$$\vert \mu_n\vert\leq \vert \mu_0\vert \quad\mbox{and}\quad \min\{\sigma_0^2,\bsigma_\alpha^2\}\leq \sigma_n^2\leq \max\{\sigma_0^2,\bsigma_\alpha^2\},$$
for every $n\in \mathbb{N}$. Then, we obtain that both the quadratic and exponential moments are uniformly bounded.
\end{remark}

In the sequel, we explore the case of generic initial data $F_0\in \mathcal{M}_+(\mathbb{R})$. For simplicity of notation, we define the following conservative operator associated with $\mathcal{T}$.

\begin{definition}[Normalization of $\mathcal{T}$]\label{D-normalized-operator-T}
$$
\mathcal{S}[F]:=\frac{\mathcal{T}[F]}{\Vert \mathcal{T}[F]\Vert_{L^1(\mathbb{R}^d)}},\quad F\in \mathcal{M}_+(\mathbb{R}^d)\setminus \{0\}\,.
$$
\end{definition}

Therefore, we obtain the following control of quadratic moments under $\mathcal{S}$.

\begin{lemma}[Control of quadratic moments]\label{L-quadratic-moments}
Assume that $\alpha\in \mathbb{R}_+^*$ and set any parameter $\eta\in \big(\frac{1}{2(1+\alpha)^2},1\big)$. Then, there exists a constant $M=M(\alpha,\eta)\in \mathbb{R}_+^*$ such that
$$\int_{\mathbb{R}^d}\vert x\vert^2 \mathcal{S}[F](x)\,dx\leq M+\eta \int_{\mathbb{R}^d}\vert x\vert^2\,\frac{F(dx)}{\Vert F\Vert_{\mathcal{M}(\mathbb{R}^d)}},$$
for any measure $F\in \mathcal{M}_+(\mathbb{R}^d)$.
\end{lemma}

Before we proceed with the proof, let us comment on some apparent difficulties. For this discussion, we assume normalized $F\in \mathcal{P}(\mathbb{R}^d)$ without loss of generality, as we are only concerned with the shape of profiles here. Note that the conservative operator $\mathcal S$ in Definition \ref{D-normalized-operator-T} associated with $\mathcal{T}$ is actually the composition of the conservative operator $\mathcal B$, and the conservative multiplicative operator $\mathcal M$ in Definition \ref{D-operator-M}, namely, $\mathcal S = \mathcal M\circ \mathcal B$. The former is invariant by translation, whereas the latter is not. However, we have the following inequality for the operator $\mathcal{M}$:
\begin{equation}\label{eq:variance M}
\int_{\mathbb{R}^d} |x|^2 \mathcal M[F](x)\,dx \leq \int_{\mathbb{R}^d}  |x|^2 F(dx)\,. 
\end{equation}
This inequality is intuitively clear: by applying the selection operator $\mathcal{M}$, the quadratic moment (that is, the Wasserstein distance to the Dirac distribution at the origin) is non-increasing. In fact, the proof is a straightforward consequence of the following identity:
$$
\int_{\mathbb{R}^d} |x|^2 \mathcal M[F](x)\,dx - \int_{\mathbb{R}^d} \vert x\vert^2 F(dx)  = \frac12  \int_{\mathbb{R}^{2d}} \left( |x|^2 - |y|^2 \right) \left(e^{-m(x)} - e^{-m(y)}\right) \dfrac{F(dx) F(dy)}{\int_{\mathbb{R}^d} e^{-m(z)}F(dz)}\,. 
$$
Note that the right-hand-side of this equality is indeed non-positive, provided that $m$ is radially non-decreasing as it is the case for our quadratic choice \eqref{E-selection-quadratic}. A similar estimate can also be tested on the operator $\mathcal B$. Specifically, by a simple change of variable the following identity holds true:
\begin{equation}\label{eq:variance B}
\int_{\mathbb{R}^d} |x|^2 \mathcal B[F](x)\,dx  = \mbox{Var}\,G +  \frac12 \int_{\mathbb{R}^d}  |x|^2\,F(dx) +\dfrac12 \left\vert\int_{\mathbb{R}^d} x\,F(dx)\right\vert^2\,. 
\end{equation}
Then, it might seem natural to combine these two relationships in order to control the quadratic moments of the composition $\mathcal{S}=\mathcal{M}\circ \mathcal{B}$. Suppose, for instance, that the center of mass in the last term of \eqref{eq:variance B} can be bounded {\em a priori} uniformly on $F$. Then, the combination of \eqref{eq:variance M} and \eqref{eq:variance B} yields the contraction property in Lemma \ref{L-quadratic-moments} due to the reduction by half in the size of the quadratic moment of $F$. Of course, without such a uniform control on the center of mass, the last two terms in \eqref{eq:variance B} would contribute together (by Jensen's inequality) with a merely non-expansive factor $\int_{\mathbb{R}^d}\vert x\vert^2\, F(dx)$ that is not enough for our purpose (we will iterate Lemma \ref{L-quadratic-moments} on the $F_n$ later). This {\em a priori} estimate of the center of mass should then be derived mainly from the selection component $\mathcal M$, as the reproduction component $\mathcal B$ is invariant by translation. However, it cannot be derived solely from $\mathcal{M}$, because the latter can have a dramatic effect on nearly symmetric distributions. This can be seen on the same configuration as in Example \ref{ex:incompatibility} above: for a sum of nearly symmetrical Dirac masses, the center of mass is close to zero before selection, and close to one of the endpoints after selection (when $h$ is large). Nevertheless, such a configuration is destroyed by $\mathcal{B}$, so that the combination of $\mathcal M$ and $\mathcal B$ is crucial to guarantee the uniform boundedness of the quadratic moments as shown in the argument below.

\begin{proof}[Proof of Lemma \ref{L-quadratic-moments}] Set any $F\in \mathcal{M}_+(\mathbb{R}^d)$. First, note that
$$\int_{\mathbb{R}^d}\vert x\vert^2\mathcal{S}[F](x)\,dx=\frac{\displaystyle\int_{\mathbb{R}^{3d}}\vert x\vert^2 e^{-m(x)}G\left(x-\frac{x_1+x_2}{2}\right)\frac{F(dx_1)}{\Vert F\Vert_{\mathcal{M}(\mathbb{R}^d)}}\frac{F(dx_2)}{\Vert F\Vert_{\mathcal{M}(\mathbb{R}^d)}}dx}{\displaystyle\int_{\mathbb{R}^{3d}} e^{-m(x)}G\left(x-\frac{x_1+x_2}{2}\right)\frac{F(dx_1)}{\Vert F\Vert_{\mathcal{M}(\mathbb{R}^d)}}\frac{F(dx_2)}{\Vert F\Vert_{\mathcal{M}(\mathbb{R}^d)}}\,dx}.$$
Since $G$ is a Gaussian function given by  \eqref{E-mixing-Gaussian} and $m$ is a quadratic function given by \eqref{E-selection-quadratic}, then we can compute explicitly the integrals with respect to $x$ in the numerator and denominator. Specifically, for fixed $x_1,x_2\in \mathbb{R}^d$ we define $\bar x:=\frac{x_1+x_2}{2}$ and we obtain
\begin{align}
\int_{\mathbb{R}^d}\vert x\vert^2 e^{-m(x)}G(x-\bar x)\,dx&=\frac{1}{(1+\alpha)^{d/2}}\exp\left(-\frac{1}{2}\frac{\alpha}{1+\alpha}\vert \bar x\vert^2\right)\left(\frac{1}{1+\alpha}+\frac{\vert \bar x\vert^2}{(1+\alpha)^2}\right),\label{E-computation-explicit-integral-1}\\
\int_{\mathbb{R}^d} e^{-m(x)}G(x-\bar x)\,dx&=\frac{1}{(1+\alpha)^{d/2}}\exp\left(-\frac{1}{2}\frac{\alpha}{1+\alpha}\vert \bar x\vert^2\right).\label{E-computation-explicit-integral-2}
\end{align}
Therefore, we can write alternatively
\begin{equation}\label{E-quadratic-moment-x-integration}
\int_{\mathbb{R}^d}\vert x\vert^2\mathcal{S}[F](x)\,dx=\frac{1}{1+\alpha}+\frac{1}{(1+\alpha)^2}\frac{\displaystyle\int_{\mathbb{R}^{2d}}\vert \bar x\vert^2 \exp\left(-\frac{1}{2}\frac{\alpha}{1+\alpha}\vert \bar x\vert^2\right)\frac{F(dx_1)}{\Vert F\Vert_{\mathcal{M}(\mathbb{R}^d)}}\frac{F(dx_2)}{\Vert F\Vert_{\mathcal{M}(\mathbb{R}^d)}}}{\displaystyle\int_{\mathbb{R}^{2d}}\exp\left(-\frac{1}{2}\frac{\alpha}{1+\alpha}\vert \bar x\vert^2\right)\frac{F(dx_1)}{\Vert F\Vert_{\mathcal{M}(\mathbb{R}^d)}}\frac{F(dx_2)}{\Vert F\Vert_{\mathcal{M}(\mathbb{R}^d)}}}.
\end{equation}
Our ultimate goal is to bound the second term in \eqref{E-quadratic-moment-x-integration}. To this end, we provide a lower bound for the denominator and an upper bound for the numerator in two separate steps.\\

$\bullet$ {\sc Step 1}: Consider an arbitrary $R_1\in \mathbb{R}_+^*$ to be determined later. Then, we obtain
\begin{align}\label{E-quadratic-moment-denominator}
\begin{aligned}
\int_{\mathbb{R}^{2d}}&\exp\left(-\frac{1}{2}\frac{\alpha}{1+\alpha}\vert \bar x\vert^2\right)\frac{F(dx_1)}{\Vert F\Vert_{\mathcal{M}(\mathbb{R}^d)}}\frac{F(dx_2)}{\Vert F\Vert_{\mathcal{M}(\mathbb{R}^d)}}\\
&\geq \int_{\vert \bar x\vert\leq R_1}\exp\left(-\frac{1}{2}\frac{\alpha}{1+\alpha}\vert \bar x\vert^2\right)\frac{F(x_1)}{\Vert F\Vert_{L^1(\mathbb{R}^d)}}\frac{F(x_2)}{\Vert F\Vert_{L^1(\mathbb{R}^d)}}\,dx_1\,dx_2\\
&\geq \exp\left(-\frac{1}{2}\frac{\alpha}{1+\alpha}R_1^2\right)\left(1-\int_{\vert \bar x\vert>R_1}\frac{F(dx_1)}{\Vert F\Vert_{\mathcal{M}(\mathbb{R}^d)}}\frac{F(dx_2)}{\Vert F\Vert_{\mathcal{M}(\mathbb{R}^d)}}\right)\\
&\geq \exp\left(-\frac{1}{2}\frac{\alpha}{1+\alpha}R_1^2\right)\left(1-\frac{1}{2R_1^2}\int_{\mathbb{R}^d}\vert x\vert^2\frac{F(dx)}{\Vert F\Vert_{\mathcal{M}(\mathbb{R}^d)}}\right),
\end{aligned}
\end{align}
where in last line we have used Young's and Chebyshev's inequalities.

\medskip

$\bullet$ {\sc Step 2}: Consider an arbitrary $R_2\in \mathbb{R}_+^*$ to be determined later. Then, we obtain
$$\int_{\mathbb{R}^{2d}}\vert \bar x\vert^2 \exp\left(-\frac{1}{2}\frac{\alpha}{1+\alpha}\vert \bar x\vert^2\right)\frac{F(dx_1)}{\Vert F\Vert_{\mathcal{M}(\mathbb{R}^d)}}\frac{F(dx_2)}{\Vert F\Vert_{\mathcal{M}(\mathbb{R}^d)}}=I_1+I_2,$$
where each term reads
\begin{align*}
I_1&:=\int_{\vert \bar x\vert\leq R_2}\vert \bar x\vert^2 \exp\left(-\frac{1}{2}\frac{\alpha}{1+\alpha}\vert \bar x\vert^2\right)\frac{F(dx_1)}{\Vert F\Vert_{\mathcal{M}(\mathbb{R}^d)}}\frac{F(dx_2)}{\Vert F\Vert_{\mathcal{M}(\mathbb{R}^d)}},\\
I_2&:=\int_{\vert \bar x\vert >R_2}\vert \bar x\vert^2 \exp\left(-\frac{1}{2}\frac{\alpha}{1+\alpha}\vert \bar x\vert^2\right)\frac{F(dx_1)}{\Vert F\Vert_{\mathcal{M}(\mathbb{R}^d)}}\frac{F(dx_2)}{\Vert F\Vert_{\mathcal{M}(\mathbb{R}^d)}}.
\end{align*}
On the one hand, the first term can be readily estimated by
\begin{equation}\label{E-quadratic-moment-numerator-1}
I_1\leq R_2^2\int_{\mathbb{R}^{2d}}\exp\left(-\frac{1}{2}\frac{\alpha}{1+\alpha}\vert \bar x\vert^2\right)\frac{F(dx_1)}{\Vert F\Vert_{\mathcal{M}(\mathbb{R}^d)}}\frac{F(dx_2)}{\Vert F\Vert_{\mathcal{M}(\mathbb{R}^d)}}.
\end{equation}
On the other hand, the second term can be estimated by
\begin{align}\label{E-quadratic-moment-numerator-2}
\begin{aligned}
I_2&\leq \frac{2}{e\gamma}\frac{1+\alpha}{\alpha}\int_{\vert \bar x\vert>R_2}\exp\left(-\frac{1-\gamma}{2}\frac{\alpha}{1+\alpha}\vert \bar x\vert^2\right)\frac{F(dx_1)}{\Vert F\Vert_{\mathcal{M}(\mathbb{R}^d)}}\frac{F(dx_2)}{\Vert F\Vert_{\mathcal{M}(\mathbb{R}^d)}}\\
&\leq \frac{2}{e\gamma}\frac{1+\alpha}{\alpha}\exp\left(-\frac{1-\gamma}{2}\frac{\alpha}{1+\alpha}R_2^2\right),
\end{aligned}
\end{align}
for any $\gamma\in (0,1)$, where in the first line we have used that $e^{\theta s}\geq \theta e s$ with $s=|\bar x|^2$ and $\theta=\frac{\gamma}{2}\frac{\alpha}{1+\alpha}$, and in the second line we have exploited that $F/\Vert F\Vert_{\mathcal{M}(\mathbb{R}^d)}$ are probability measures, thus normalized.

\medskip 

$\bullet$ {\sc Step 3}: Putting \eqref{E-quadratic-moment-denominator}, \eqref{E-quadratic-moment-numerator-1} and \eqref{E-quadratic-moment-numerator-2} into \eqref{E-quadratic-moment-x-integration} yields
$$\int_{\mathbb{R}^d}\vert x\vert^2\mathcal{S}[F](x)\,dx\leq \frac{1}{1+\alpha}+\frac{R_2^2}{(1+\alpha)^2}+\frac{2}{e\gamma}\frac{1}{\alpha(1+\alpha)}\frac{\exp\left(\frac{1}{2}\frac{\alpha}{1+\alpha}R_1^2\right)\exp\left(-\frac{1-\gamma}{2}\frac{\alpha}{1+\alpha}R_2^2\right)}{1-\frac{1}{2R_1^2}\int_{\mathbb{R}^d}\vert x\vert^2\frac{F(dx)}{\Vert F\Vert_{\mathcal{M}(\mathbb{R}^d)}}},$$
for any $R_1,R_2\in \mathbb{R}_+^*$ and $\gamma\in (0,1)$. We set the values
$$R_1^2=\frac{1}{2(1-\delta)}\int_{\mathbb{R}^d}\vert x\vert^2\frac{F(dx)}{\Vert F\Vert_{\mathcal{M}(\mathbb{R}^d)}},\qquad R_2^2=\frac{R_1^2}{1-\gamma},$$
with $\gamma,\delta\in (0,1)$. Then, we obtain
$$\int_{\mathbb{R}^d}\vert x\vert^2\mathcal{S}[F](x)\,dx\leq \frac{1}{1+\alpha}+\frac{2}{e\gamma\delta}\frac{1}{\alpha(1+\alpha)}+\frac{1}{2(1-\delta)(1-\gamma)(1+\alpha)^2}\int_{\mathbb{R}^d}\vert x\vert^2\frac{F(dx)}{\Vert F\Vert_{\mathcal{M}(\mathbb{R}^d)}},$$
thus ending the proof by the arbitrariness of $\gamma,\delta\in (0,1)$.
\end{proof}

By iteration, we then obtain the following contractivity of the variance of generic solutions of \eqref{E-sexual-reproduction-time-discrete}.

\begin{corollary}[Propagation of quadratic moments]\label{C-quadratic-moments}
Assume that $\alpha\in \mathbb{R}_+^*$, set any parameter $\eta\in \big(\frac{1}{2(1+\alpha)^2},1\big)$ and consider the solution $\{F_n\}_{n\in \mathbb{N}}$ of \eqref{E-sexual-reproduction-time-discrete} issued at a generic initial datum $F_0\in \mathcal{M}_+(\mathbb{R}^d)$ with $\int_{\mathbb{R}^d}\vert x\vert^2 F_0(dx)<\infty$. Then, there exists $M=M(\alpha,\eta)$ such that
$$\int_{\mathbb{R}^d}\vert x\vert^2\frac{F_n(x)}{\Vert F_n\Vert_{L^1(\mathbb{R}^d)}}\,dx\leq \frac{M}{1-\eta}+\eta^n\int_{\mathbb{R}^d}\vert x\vert^2 \frac{F_0(dx)}{\Vert F_0\Vert_{\mathcal{M}(\mathbb{R}^d)}},$$
for any $n\in \mathbb{N}.$
\end{corollary}

Let us note that, on the one hand the factor $\eta^n$ makes the dependence on the variance of the initial datum $F_0$ negligible as the amount of generations increases. On the other hand, the constant $M$ does not depend on the chosen initial profile $F_0$. This is compatible with our ergodicity property in Theorem \ref{T-main}.

\begin{lemma}[Control of exponential moments]\label{L-exponential-moments}
Assume that $\alpha\in \mathbb{R}_+^*$, set $\theta\in \mathbb{R}_+^*$ so that $\theta<\frac{\alpha}{2}$ and any parameter $\chi>\frac{\alpha}{1+\alpha}\frac{\theta}{2(1+\alpha-2\theta)(\alpha-2\theta)}$. Then, there exists a constant $C=C(\alpha,\theta,\chi)$ such that
$$\int_{\mathbb{R}^d}e^{\theta\vert x\vert^2} \mathcal{S}[F](x)\,dx\leq C\left\{1+\exp\left(\chi \int_{\mathbb{R}^d}\vert x\vert^2\frac{F(dx)}{\Vert F\Vert_{\mathcal{M}(\mathbb{R}^d)}}\right)\right\},$$
for any measure $F\in \mathcal{M}_+(\mathbb{R}^d)$.
\end{lemma}

\begin{proof}
Since $G$ is a Gaussian function given by  \eqref{E-mixing-Gaussian} and $m$ is a quadratic function given by \eqref{E-selection-quadratic}, then we can compute explicitly the integrals with respect to $x$ in the numerator and denominator like formulas \eqref{E-computation-explicit-integral-1}-\eqref{E-computation-explicit-integral-2} in the proof of Lemma \ref{L-quadratic-moments} and we obtain
\begin{equation}\label{E-exponential-moment-x-integration}
\int_{\mathbb{R}^d}e^{\theta \vert x\vert^2}\mathcal{S}[F](x)\,dx=\left(\frac{1+\alpha}{1+\alpha-2\theta}\right)^{d/2}\frac{\displaystyle\int_{\mathbb{R}^{2d}}\exp\left(-\frac{1}{2}\frac{\alpha-2\theta}{1+\alpha-2\theta}\vert \bar x\vert^2\right)\frac{F(dx_1)}{\Vert F\Vert_{\mathcal{M}(\mathbb{R}^d)}}\frac{F(dx_2)}{\Vert F\Vert_{\mathcal{M}(\mathbb{R}^d)}}}{\displaystyle\int_{\mathbb{R}^{2d}}\exp\left(-\frac{1}{2}\frac{\alpha}{1+\alpha}\vert \bar x\vert^2\right)\frac{F(dx_1)}{\Vert F\Vert_{\mathcal{M}(\mathbb{R}^d)}}\frac{F(dx_2)}{\Vert F\Vert_{\mathcal{M}(\mathbb{R}^d)}}},
\end{equation}
where $\bar x=\frac{x_1+x_2}{2}$. Again, we provide bounds for the numerator and denominator in \eqref{E-exponential-moment-x-integration}. Regarding the denominator, we recover estimate \eqref{E-quadratic-moment-denominator} for any $R_1\in \mathbb{R}_+^*$. Thus, we focus on the bound of the numerator
$$\int_{\mathbb{R}^{2d}}\exp\left(-\frac{1}{2}\frac{\alpha-2\theta}{1+\alpha-2\theta}\vert \bar x\vert^2\right)\frac{F(dx_1)}{\Vert F\Vert_{\mathcal{M}(\mathbb{R}^d)}}\frac{F(dx_2)}{\Vert F\Vert_{\mathcal{M}(\mathbb{R}^d)}}=I_1+I_2,$$
where each term reads
\begin{align*}
I_1&:=\int_{\vert \bar x\vert\leq R_2}\exp\left(-\frac{1}{2}\frac{\alpha-2\theta}{1+\alpha-2\theta}\vert \bar x\vert^2\right)\frac{F(dx_1)}{\Vert F\Vert_{\mathcal{M}(\mathbb{R}^d)}}\frac{F(dx_2)}{\Vert F\Vert_{\mathcal{M}(\mathbb{R}^d)}},\\
I_2&:=\int_{\vert \bar x\vert> R_2}\exp\left(-\frac{1}{2}\frac{\alpha-2\theta}{1+\alpha-2\theta}\vert \bar x\vert^2\right)\frac{F(dx_1)}{\Vert F\Vert_{\mathcal{M}(\mathbb{R}^d)}}\frac{F(dx_2)}{\Vert F\Vert_{\mathcal{M}(\mathbb{R}^d)}}.
\end{align*}
On the one hand, note that
\begin{equation}\label{E-exponential-moment-numerator-1}
I_1\leq \exp\left(\frac{\theta R_2^2}{(1+\alpha-2\theta)^2}\right)\int_{\mathbb{R}^{2d}}\exp\left(-\frac{1}{2}\frac{\alpha}{1+\alpha}\vert \bar x\vert^2\right)\frac{F(dx_1)}{\Vert F\Vert_{\mathcal{M}(\mathbb{R}^d)}}\frac{F(dx_2)}{\Vert F\Vert_{\mathcal{M}(\mathbb{R}^d)}},
\end{equation}
where we have used the mean value theorem applied to the function $r\in \mathbb{R}_+^*\mapsto \frac{r}{1+r}$ to derive the relation
$$\frac{\alpha}{1+\alpha}-\frac{\alpha-2\theta}{1+\alpha-2\theta}\leq \frac{2\theta}{(1+\alpha-2\theta)^2}.$$
On the other hand,
\begin{equation}\label{E-exponential-moment-numerator-2}
I_2\leq \exp\left(-\frac{1}{2}\frac{\alpha-2\theta}{1+\alpha-2\theta}R_2^2\right).
\end{equation}
Putting \eqref{E-quadratic-moment-denominator}, \eqref{E-exponential-moment-numerator-1} and \eqref{E-exponential-moment-numerator-2} into \eqref{E-exponential-moment-x-integration} we obtain
$$\int_{\mathbb{R}^d}e^{\theta\vert x\vert^2}\mathcal{S}[F](x)\,dx\leq \left(\frac{1+\alpha}{1+\alpha-2\theta}\right)^{d/2}\left\{\exp\left(\frac{\theta R_2^2}{(1+\alpha-2\theta)^2}\right)+\frac{\exp\left(-\frac{1}{2}\frac{\alpha-2\theta}{1+\alpha-2\theta}R_2^2\right)\exp\left(\frac{1}{2}\frac{\alpha}{1+\alpha}R_1^2\right)}{1-\frac{1}{2R_1^2}\int_{\mathbb{R}^d}\vert x\vert^2 \frac{F(dx)}{\Vert F\Vert_{\mathcal{M}(\mathbb{R}^d)}}}\right\},$$
for any $R_1,R_2\in \mathbb{R}_+^*.$ As in the proof of Lemma \ref{L-quadratic-moments} we choose them as follows
$$R_1^2=\frac{1}{2(1-\delta)}\int_{\mathbb{R}^d}\vert x\vert^2\frac{F(dx)}{\Vert F\Vert_{\mathcal{M}(\mathbb{R}^d)}},\qquad R_2^2=\frac{\alpha}{1+\alpha}\frac{1+\alpha-2\theta}{\alpha-2\theta}R_1^2,$$
where $\delta\in (0,1)$. Then, we obtain
\begin{multline*}
\int_{\mathbb{R}^d}e^{\theta\vert x\vert^2}\mathcal{S}[F](x)\,dx\\
\leq \left(\frac{1+\alpha}{1+\alpha-2\theta}\right)^{d/2}\left\{\frac{1}{\delta}+\exp\left(\frac{\alpha}{1+\alpha}\frac{\theta}{(1+\alpha-2\theta)(\alpha-2\theta)}\frac{1}{2(1-\delta)}\int_{\mathbb{R}^d}\vert x\vert^2\frac{F(dx)}{\Vert F\Vert_{\mathcal{M}(\mathbb{R}^d)}}\right)\right\},
\end{multline*}
and we end the proof by arbitrariness of $\delta\in (0,1)$.
\end{proof}

Therefore, small exponential moments can be controlled by quadratic moments. We can then couple Lemma \ref{L-exponential-moments} and Corollary \ref{C-quadratic-moments} to obtain the propagation of exponential moments in the time-discrete problem \eqref{E-sexual-reproduction-time-discrete}.

\begin{corollary}[Propagation of exponential moments]\label{C-exponential-moments}
Assume that $\alpha\in \mathbb{R}_+^*$, set $\theta\in \mathbb{R}_+^*$ with $\theta<\frac{\alpha}{2}$ and any parameter $\eta\in \big(\frac{1}{2(1+\alpha)^2},1\big)$. Consider the solution $\{F_n\}_{n\in \mathbb{N}}$ of \eqref{E-sexual-reproduction-time-discrete} issued at a generic initial datum $F_0\in \mathcal{M}_+(\mathbb{R}^d)$ with $\int_{\mathbb{R}^d}\vert x\vert^2 F_0(dx)<\infty$. Then, there exist $C=C(\alpha,\theta,\eta)$ and $C'=C'(\alpha,\theta,\eta)$ such that
$$\int_{\mathbb{R}^d}e^{\theta\vert x\vert^2}\frac{F_n(x)}{\Vert F_n\Vert_{L^1(\mathbb{R}^d)}}\,dx\leq C\left\{1+\exp\left(C'\left(1+\eta^{n-1}\int_{\mathbb{R}^d}\vert x\vert^2\frac{F_0(dx)}{\Vert F_0\Vert_{\mathcal{M}(\mathbb{R}^d)}}\right)\right)\right\},$$
for any $n\in \mathbb{N}$.
\end{corollary}

\section{Reformulating the recursion}\label{S-restating-iterations}
In this section we find an appropriate reformulation of the recursion for the solutions $\{F_n\}_{n\in \mathbb{N}}$ of the time-discrete problem \eqref{E-sexual-reproduction-time-discrete}, as already discussed in Section \ref{SS-strategy}. Note that by iterating $\mathcal{T}$ $n$-times we find that $F_n$ depends on the initial datum $F_0\in \mathcal{M}_+(\mathbb{R}^d)$ via a high-dimensional integral parametrized by variables indexed on the binary tree $\Tree^n_*$ (representing the traits of the ancestors in the pedigree chart). Specifically, for any initial datum $F_0\in \mathcal{M}_+(\mathbb{R}^d)$ we obtain:

\medskip

\noindent $\diamond$ {\em Iteration 1}:\\
By the explicit definition of $\mathcal{T}$ in \eqref{E-sexual-reproduction-T-operator} we readily obtain
\begin{equation}\label{E-sexual-reproduction-iterations_n=1}
F_1(x)=\frac{1}{\Vert F_0\Vert_{\mathcal{M}(\mathbb{R}^d)}}\int_{\mathbb{R}^{2d}}e^{-m(x)}G\left(x-\frac{x_1+x_2}{2}\right)F_0(x_1)F_0(x_2)\,d\mathbf{x}_1,
\end{equation}
where we denote $\bx_1=(x_1,x_2)\in \mathbb{R}^{2d}$.

\medskip

\noindent $\diamond$ {\em Iteration 2}:\\
By \eqref{E-sexual-reproduction-iterations_n=1} with $F_0$ replaced by $F_1$ and using $F_1=\mathcal{T}[F_0]$ in the right hand side we get
\begin{multline}\label{E-sexual-reproduction-iterations_n=2}
F_2(x)=\frac{1}{\Vert F_1\Vert_{L^1(\mathbb{R}^d)}\Vert F_0\Vert_{\mathcal{M}(\mathbb{R}^d)}^2}\int_{\mathbb{R}^{6d}} e^{-(m(x)+m(x_1)+m(x_2))}G\left(x-\frac{x_1+x_2}{2}\right)\\
\times G\left(x_1-\frac{x_{11}+x_{12}}{2}\right)G\left(x_2-\frac{x_{21}+x_{22}}{2}\right) F_0(x_{11})F_0(x_{12})F_0(x_{21})F_0(x_{22})\,d\mathbf{x}_2,
\end{multline}
where we denote $\bx_2=(x_1,x_2,x_{11},x_{12},x_{21},x_{22})\in \mathbb{R}^{6d}$.

\medskip

\noindent $\diamond$ {\em Iteration $n\in \mathbb{N}$}:\\
By a clear inductive process, we can iterate the operator $\mathcal{T}$ as many times $n\in \mathbb{N}$ as needed to recover an explicit dependence of $F_n$ on the initial datum $F_0$. Of course, this generates a high-dimensional integral involving  $2(2^n-1)$ variables that we can label along the vertices of the perfect binary tree $\Tree^n$ according to the \textit{universal address system} notation in Section \ref{SS-trees}. Namely, we obtain
\begin{multline}\label{E-sexual-reproduction-iterations}
F_n(x)=\frac{1}{\prod_{m=0}^{n-1}\Vert F_m\Vert_{L^1(\mathbb{R}^d)}^{2^{n-1-m}}}\\
\times \int_{\mathbb{R}^{2(2^n-1)d}} \exp\left(-\sum_{i\in \widehat{\Tree}^n}m(x_i)\right)\prod_{i\in \widehat{\Tree}^n} G\left(x_i-\frac{x_{i1}+x_{i2}}{2}\right)\prod_{i\in \Leaves^n}F_0(x_i)\,d\mathbf{x}_n,
\end{multline}
where we have used the tree-indexed notation $\bx_n=(x_i)_{i\in \Tree^n_*}\in \mathbb{R}^{2(2^n-1)d}$ in Remark \ref{R-tree-indexed-variables} and $x_\emptyset:=x$. 

\medskip

The ultimate goal of this section is to find a suitable change of variables as discussed in Section \ref{SS-strategy}. We shall see that the effect of such a change of variable on the above high-dimensional integral in \eqref{E-sexual-reproduction-iterations} will reveal that actual dependence of $\{F_n\}_{n\in \mathbb{N}}$ on the shape of $F_0$. We will see that such a dependence is indeed weak and the initial datum is rapidly ``forgotten'' across the different levels of the tree $\Tree^n$. This ergodicity property will be crucial in our analysis and will be exploited later in Section \ref{S-ergodicity} to derive our main result in Theorem \ref{T-main}.

\subsection{A change of variables across the binary tree}
The goal of this section is to appropriately reformulate the high-dimensional integral \eqref{E-sexual-reproduction-iterations} by exploiting the special form of the Gaussian mixing kernel $G$ in \eqref{E-mixing-Gaussian} and the quadratic selection function $m$ in \eqref{E-selection-quadratic}. Namely, we shall derive a suitable change of variables according to a $n$-step backwards process starting at the leaf variables $x_j$ with $j\in \Leaves^n$ and ending at the root variable $x_\emptyset=x$. More specifically, at each level $m$ in the tree we shall change the reference frame of the variables $x_i$ indexed with $i\in \Level^n_m$ so that the quadratic form involving those variables in the exponential of \eqref{E-sexual-reproduction-iterations} gets appropriately centered at its minimum. By doing so we find that the minimum is located at a contracted value of the variable $x_{\child(i)}\in L^n_{m-1}$ at the level $m-1$ below, which is indexed by its children $\child(i)$ in the binary tree (recall notation in Section \ref{SS-trees}). By appropriately propagating the backwards process across the tree, and by tracking the accumulated contraction of variables, we shall see that the eventual dependence on the root variable $x$ is weakly gradually forgotten as $n\rightarrow \infty$. This suggests an apparent form of ergodicity, which we will be crucial later in Section \ref{S-ergodicity}. 

First, for the sake of clarity, we illustrate our method in the simpler case $n=2$. Later, we address the general case with $n>2$ driven by the same strategy. As proposed in Section \ref{SS-strategy}, we define the following choice for the rescaled profiles that will be used along the change of variables.

\begin{definition}[Rescaled distributions]\label{D-rescaled-F}
Consider any $F\in \mathcal{M}_+(\mathbb{R}^d)$. Then, we define
$$
\bar{F}(dx):=e^{m(x)} \frac{F(dx)}{\bF_{\alpha=0}(x)},\quad x\in \mathbb{R}^d.
$$
\end{definition}

Here, we remind that $\bF_{\alpha=0}=G_{0,2}$ is the Gaussian eigenfunction corresponding to \eqref{E-sexual-reproduction-Gaussian-solution} with $\alpha=0$. We also recall that there is some freedom in the normalization, and in particular the term $e^m$ is not mandatory but it is convenient for an easier sorting of the various terms, as we anticipated in Section \ref{SS-strategy}. By substituting \eqref{E-mixing-Gaussian} and \eqref{E-selection-quadratic} into \eqref{E-sexual-reproduction-iterations_n=1} and writting the integral in the right hand side in terms of the rescaled function $\bar{F}_0$ in Definition \ref{D-rescaled-F}, we obtain that
\begin{align}\label{E-sexual-reproduction-iterations_n=1_Gaussian-quadratic}
\begin{aligned}
&(e^m F_2)(x)=\frac{1}{(2\pi)^{3d/2}}\frac{1}{(4\pi)^{2d}}\frac{1}{\Vert F_1\Vert_{L^1(\mathbb{R}^d)}\Vert F_0\Vert_{\mathcal{M}(\mathbb{R}^d)}^2}\\
&\times \int_{\mathbb{R}^{6d}} \exp\left(-\frac{\alpha}{2}\big(\vert x_1\vert^2+\vert x_2\vert^2\big)-\frac{1}{2}\left\vert x-\frac{x_1+x_2}{2}\right\vert^2\right)\\
&\qquad\times \exp\left(-\frac{\alpha_0}{2}\big(\vert x_{11}\vert^2+\vert x_{12}\vert^2+\vert x_{21}\vert^2+\vert x_{22}\vert^2\big)-\frac{1}{2}\left\vert x_1-\frac{x_{11}+x_{12}}{2}\right\vert^2-\frac{1}{2}\left\vert x_2-\frac{x_{21}+x_{22}}{2}\right\vert^2\right)\\
&\qquad\times\bar{F}_0(x_{11})\bar{F}_0(x_{12})\bar{F}_0(x_{21})\bar{F}_0(x_{22})\,d\mathbf{x}_2.
\end{aligned}
\end{align}
Here, the parameter $\alpha_0\in \mathbb{R}_+^*$ has been defined by
$$\alpha_0:=\frac{1}{2}+\alpha,$$
so that the quadratic terms in the fist term of the third line of \eqref{E-sexual-reproduction-iterations_n=1_Gaussian-quadratic} correct the rescaling $\bar{F}_0$ of $F_0$. The new formulation will be obtained by an appropriate change of variables, so that the quadratic forms inside the above exponential get appropriately centered. To do so, notice that this formula involves as many variables $x_i$ as indices $i$ in the perfect binary tree $\Tree^2$. Indeed, we have sorted the different terms in \eqref{E-sexual-reproduction-iterations_n=1_Gaussian-quadratic} in such a way that the second line only involves variables $x_1$ and $x_2$ in the first level $\Level_1^2$ of the tree, whilst the third line contains the terms involving the variables $x_{11}$, $x_{12}$, $x_{21}$ and $x_{22}$ in the second level of the tree $\Level_2^2$ ({\em i.e.}, the leaves $\Leaves^2$). We will divide the method into two steps. First, we will address the change of the variables indexed by the leaves $\Leaves^2$. Second, we will perform the change of variables indexed by the first level of the tree $\Level_1^2$.

\medskip

$\bullet$ {\sc Step 1:} Change of variables for $x_{11}$, $x_{12}$, $x_{21}$ and $x_{22}$.\\
Consider some coefficient $k_1>0$ to be determined later and define the change of variables $x_{11}\rightarrow y_{11}$, $x_{12}\rightarrow y_{12}$, $x_{21}\rightarrow y_{21}$ and $x_{22}\rightarrow y_{22}$ given by
\begin{align*}
x_{11}&=k_1x_1+y_{11}, & x_{12}&=k_1x_1+y_{12},\\
x_{21}&=k_1x_2+y_{21}, & x_{22}&=k_1x_2+y_{12}.
\end{align*}
On the one hand, using the change of variables in the terms inside the exponential of \eqref{E-sexual-reproduction-iterations_n=1_Gaussian-quadratic} which involve $x_{11}$ and $x_{12}$, we obtain
\begin{align*}
\frac{\alpha_0}{2}\vert x_{11}\vert^2+\frac{\alpha_0}{2}\vert x_{12}\vert^2&+\frac{1}{2}\left\vert x_1-\frac{x_{11}+x_{12}}{2}\right\vert^2\\
&=\frac{\alpha_0}{2}\vert k_1x_1+y_{11}\vert^2+\frac{\alpha_0}{2}\vert k_1 x_1+y_{12}\vert^2+\frac{1}{2}\left\vert (1-k_1)x_1-\frac{y_{11}+y_{12}}{2}\right\vert^2\\
&=\frac{1}{2}\left(2\alpha_0 k_1^2+(1-k_1)^2\right)\vert x_1\vert^2+\frac{\alpha_0}{2}\vert y_{11}\vert^2+\frac{\alpha_0}{2}\vert y_{12}\vert^2+\frac{1}{8}\vert y_{11}+y_{12}\vert^2\\
&\hspace{2cm} +\left(\alpha_0 k_1-\frac{1-k_1}{2}\right)x_1\cdot y_{11}+\left(\alpha_0 k_1-\frac{1-k_1}{2}\right)x_1\cdot y_{12}.
\end{align*}
On the other hand, using the change of variables in the terms which involve $x_{21}$ and $x_{22}$, we get
\begin{align*}
\frac{\alpha_0}{2}\vert x_{21}\vert^2+\frac{\alpha_0}{2}\vert x_{22}\vert^2&+\frac{1}{2}\left\vert x_2-\frac{x_{21}+x_{22}}{2}\right\vert^2\\
&=\frac{\alpha_0}{2}\vert k_1x_2+y_{21}\vert^2+\frac{\alpha_0}{2}\vert k_1 x_2+y_{22}\vert^2+\frac{1}{2}\left\vert (1-k_1)x_2-\frac{y_{21}+y_{22}}{2}\right\vert^2\\
&=\frac{1}{2}\left(2\alpha_0 k_1^2+(1-k_1)^2\right)\vert x_2\vert^2+\frac{\alpha_0}{2}\vert y_{21}\vert^2+\frac{\alpha_0}{2}\vert y_{22}\vert^2+\frac{1}{8}\vert y_{21}+y_{22}\vert^2\\
&\hspace{2cm} +\left(\alpha_0 k_1-\frac{1-k_1}{2}\right)x_2\cdot y_{21}+\left(\alpha_0 k_1-\frac{1-k_1}{2}\right)x_2\cdot y_{22}.
\end{align*}
Notice that one can eliminate the crossed terms in both expressions by choosing
$$k_1:=\frac{1}{1+2\alpha_0}=\frac{1}{2(1+\alpha)}.$$
In that case, adding both terms we obtain that
\begin{align*}
\frac{\alpha_0}{2}\vert x_{11}\vert^2+\frac{\alpha_0}{2}\vert x_{12}\vert^2&+\frac{1}{2}\left\vert x_1-\frac{x_{11}+x_{12}}{2}\right\vert^2+\frac{\alpha_0}{2}\vert x_{21}\vert^2+\frac{\alpha_0}{2}\vert x_{22}\vert^2+\frac{1}{2}\left\vert x_2-\frac{x_{21}+x_{22}}{2}\right\vert^2\\
&=\frac{1}{2}(1-k_1)\vert x_1\vert^2+\frac{1}{4 k_1}\vert y_{11}\vert^2+\frac{1}{4 k_1}\vert y_{12}\vert^2-\frac{1}{8}\vert y_{11}-y_{12}\vert^2\\
&+\frac{1}{2}(1-k_1)\vert x_2\vert^2+\frac{1}{4 k_1}\vert y_{21}\vert^2+\frac{1}{4 k_1}\vert y_{22}\vert^2-\frac{1}{8}\vert y_{21}-y_{22}\vert^2.
\end{align*}
Putting everything together into \eqref{E-sexual-reproduction-iterations_n=1_Gaussian-quadratic} yields
\begin{align}\label{E-sexual-reproduction-iterations_n=1_Gaussian-quadratic-change-1}
\begin{aligned}
&(e^{m}F_2)(x)=\frac{1}{(2\pi)^{3d/2}}\frac{1}{(4\pi)^{2d}}\frac{1}{\Vert F_1\Vert_{L^1(\mathbb{R}^d)}\Vert F_0\Vert_{\mathcal{M}(\mathbb{R}^d)}^2}\\
&\times \int_{\mathbb{R}^{6d}} \exp\left(-\frac{\alpha_1}{2}\big(\vert x_1\vert^2+\vert x_2\vert^2\big)-\frac{1}{2}\left\vert x-\frac{x_1+x_2}{2}\right\vert^2\right)\\
&\qquad\times \exp\left(-\frac{1}{4 k_1}\big(\vert y_{11}\vert^2+\vert y_{12}\vert^2+\vert y_{21}\vert^2+\vert y_{22}\vert^2\big)+\frac{1}{8}\big(\vert y_{11}-y_{12}\vert^2+\vert y_{21}-y_{22}\vert^2\big)\right)\\
&\qquad\times \bar{F}_0(k_1x_1+y_{11})\bar{F}_0(k_1x_1+y_{12})\bar{F}_0(k_1x_2+y_{21})\bar{F}_0(k_1x_2+y_{22})\,dx_1\,dx_2\,dy_{11}\,dy_{12}\,dy_{21}\,dy_{22},
\end{aligned}
\end{align}
where the parameter $\alpha_1\in \mathbb{R}_+^*$ has been defined by
$$\alpha_1:=1-k_1+\alpha,$$
in order to recombine the initial terms in the second line of \eqref{E-sexual-reproduction-iterations_n=1_Gaussian-quadratic} involving variables indexed by the first level $\Level_1^2$ with the new remainders that have appeared from the previous step.

\medskip

$\bullet$ {\sc Step 2:} Change of variables for $x_1$ and $x_2$.\\
Consider some coefficient $k_2>0$ to be determined later and define the change of variables $x_1\rightarrow y_1$ and $x_2\rightarrow y_2$ given by
$$
x_1=k_2x+y_1,\quad x_2=k_2x+y_2.
$$
This time, using the change of variables in the terms of the exponential in the second line of \eqref{E-sexual-reproduction-iterations_n=1_Gaussian-quadratic-change-1} yields
\begin{align*}
\frac{\alpha_1}{2}\vert x_1\vert^2+\frac{\alpha_1}{2}\vert x_2\vert^2&+\frac{1}{2}\left\vert x-\frac{x_1+x_2}{2}\right\vert^2\\
&=\frac{\alpha_1}{2}\vert k_2x+y_1\vert^2+\frac{\alpha_1}{2}\vert k_2 x+y_2\vert^2+\frac{1}{2}\left\vert (1-k_2)x-\frac{y_1+y_2}{2}\right\vert^2\\
&=\frac{1}{2}\left(2\alpha_1 k_2^2+(1-k_2)^2\right)\vert x\vert^2+\frac{\alpha_1}{2}\vert y_1\vert^2+\frac{\alpha_1}{2}\vert y_2\vert^2+\frac{1}{8}\vert y_1+y_2\vert^2\\
&\hspace{2cm} +\left(\alpha_1 k_2-\frac{1-k_2}{2}\right)x\cdot y_1+\left(\alpha_1 k_2-\frac{1-k_2}{2}\right)x\cdot y_2.
\end{align*}
Again, we can cancel the crossed term by choosing
$$k_2:=\frac{1}{1+2\alpha_1}=\frac{1}{3+2\alpha-2k_1}.$$
Namely, we obtain that
\begin{align*}
\frac{\alpha_1}{2}\vert x_1\vert^2&+\frac{\alpha_1}{2}\vert x_2\vert^2+\frac{1}{2}\left\vert x-\frac{x_1+x_2}{2}\right\vert^2\\
&=\frac{1}{2}(1-k_2)\vert x\vert^2+\frac{1}{4 k_2}\vert y_1\vert^2+\frac{1}{4 k_2}\vert y_2\vert^2-\frac{1}{8}\vert y_1-y_2\vert^2.
\end{align*}
Then, putting everything together into \eqref{E-sexual-reproduction-iterations_n=1_Gaussian-quadratic-change-1} yields
\begin{align}\label{E-sexual-reproduction-iterations_n=1_Gaussian-quadratic-change-2}
\begin{aligned}
&F_2(x)=\frac{1}{(2\pi)^{3d/2}}\frac{1}{(4\pi)^{2d}}\frac{1}{\Vert F_1\Vert_{L^1(\mathbb{R}^d)}\Vert F_0\Vert_{\mathcal{M}(\mathbb{R}^d)}^2}e^{-\frac{1+\alpha-k_2}{2}\vert x\vert^2}\\
&\times \int_{\mathbb{R}^{6d}} \exp\left(-\frac{1}{4 k_2}\big(\vert y_1\vert^2+\vert y_2\vert^2\big)+\frac{1}{8}\vert y_1-y_2\vert^2\right)\\
&\quad\times \exp\left(-\frac{1}{4 k_1}\big(\vert y_{11}\vert^2+\vert y_{12}\vert^2+\vert y_{21}\vert^2+\vert y_{22}\vert^2\big)+\frac{1}{8}\big(\vert y_{11}-y_{12}\vert^2+\vert y_{21}-y_{22}\vert^2\big)\right)\\
&\quad\times \bar{F}_0(k_1k_2x+k_1y_1+y_{11})\bar{F}_0(k_1k_2x+k_1 y_1+y_{12})\\
&\quad\times \bar{F}_0(k_1k_2x+k_1y_2+y_{21})\bar{F}_0(k_1k_2x+k_1y_2+y_{22})\,d\mathbf{y}_2,
\end{aligned}
\end{align}
where we denote again $\by_2=(y_1,y_2,y_{11},y_{12},y_{21},y_{22})\in \mathbb{R}^{6d}$. 

\medskip

Before stating the main result for general $n\in\mathbb{N}$, we collect some natural notation according to the preceding computations, that will be useful here on. First, we define the following sequences of coefficients.

\begin{definition}[Coefficients]\label{D-sexual-reproduction-coefficients}
Consider any $\alpha\in \mathbb{R}_+$.
\begin{enumerate}
\item The coefficients $\{k_n\}_{n\in \mathbb{N}}$ are defined by the recursive formula
\begin{align}\label{E-sexual-reproduction-coefficients}
\begin{aligned}
k_1&:=\frac{1}{2(1+\alpha)},\\
k_n&:=\frac{1}{3+2\alpha-2k_{n-1}},\quad \text{for }n\geq 2.
\end{aligned}
\end{align}
\item The coefficients $\{\kappa_n\}_{n\in \mathbb{N}}$ are defined by the recursive formula
\begin{align}\label{E-sexual-reproduction-coefficients-kappa}
\begin{aligned}
\kappa_0&:=1,\\
\kappa_n&:=k_1\cdots k_n,\quad \text{for }n\geq 1.
\end{aligned}
\end{align}
\end{enumerate}
\end{definition}

We note that for $n=2$ the above coefficients $k_1,k_2$ reduce to those appearing in the previous reformulation \eqref{E-sexual-reproduction-iterations_n=1_Gaussian-quadratic-change-2} of the recursion. Since it will be used later, we study the asymptotic behavior of such sequences of coefficients as $n\rightarrow \infty$.
	
\begin{lemma}[Asymptotic behavior of the coefficients]\label{L-sexual-reproduction-coefficients-asymptotics}
Consider the sequences $\{k_n\}_{n\in \mathbb{N}}$ and $\{\kappa_n\}_{n\in \mathbb{N}}$ in Definition \ref{D-sexual-reproduction-coefficients}. Then, the following properties hold true:
\begin{enumerate}[label=(\roman*)]
\item {\bf (Coefficients $k_n$)} The coefficients $\{k_n\}_{n\in \mathbb{N}}$ are positive numbers and $k_n\searrow \bk_\alpha$ as $n\rightarrow \infty$, where $\bk_\alpha\in \mathbb{R}_+^*$ is the smallest root of the equation
\begin{equation}\label{E-sexual-reproduction-coefficients-root}
\frac{1}{3+2\alpha-2\bk_\alpha}=\bk_\alpha.
\end{equation}
Specifically, $\bk_\alpha$ is given explicitly by formula \eqref{E-sexual-reproduction-coefficients-limit} and it is related to the coefficient $\br_\alpha$ in formula \eqref{E-variance-contraction-ralpha} of Lemma \ref{L-sexual-reproduction-Gaussian-solution-time-discrete-variance-relaxation} by $\br_\alpha=2\bk_\alpha^2$. In addition,
\begin{equation}\label{E-sexual-reproduction-coefficients-asymptotics}
k_n-\bk_\alpha\leq C\br_\alpha^{n-1},
\end{equation}
for every $n\in \mathbb{N}$, where $\br_\alpha$ is and $C\in \mathbb{R}_+^*$ depends only on $\alpha$. 
\item {\bf (Coefficients $\kappa_n$)} The coefficients $\{\kappa_n\}_{n\in \mathbb{N}}$ are positive numbers and decay geometrically with maximal rate $\bk_\alpha$. Specifically, for every $\varepsilon\in \mathbb{R}_+^*$ there exists $C_\varepsilon\in \mathbb{R}_+^*$ such that
\begin{equation}\label{E-sexual-reproduction-coefficients-asymptotics-2}
\kappa_n\leq C_\varepsilon (\bk_\alpha+\varepsilon)^n,
\end{equation}
for every $n\in \mathbb{N}$.
\end{enumerate}
\end{lemma}

\begin{proof}
On the one hand, the properties of $\{\kappa_n\}_{n\in \mathbb{N}}$ readily follow from those of $\{k_n\}_{n\in \mathbb{N}}$ and the relation \eqref{E-sexual-reproduction-coefficients-kappa} in Definition \ref{D-sexual-reproduction-coefficients}. Specifically, for any given $q\in \mathbb{N}$ and any $n\geq q$ we get the decomposition
$$\kappa_n=\left(\prod_{m=0}^q k_m\right)\left(\prod_{m=q+1}^n k_m\right)\leq k_1^q k_{q+1}^{n-q}=\frac{k_1^q}{k_{q+1}^q}k_{q+1}^n\leq \frac{k_1^q}{k_{q+1}^q} (C\br_\alpha^{q+1}+\bk_\alpha)^n.$$
Therefore, taking $q$ sufficiently large so that $C\br_\alpha^{q+1}\leq \varepsilon$ we conclude \eqref{E-sexual-reproduction-coefficients-asymptotics-2}. We then focus on the properties of $\{k_n\}_{n\in \mathbb{N}}$ and we use a similar strategy like in Lemma \ref{L-sexual-reproduction-Gaussian-solution-time-discrete-variance-relaxation}.

\medskip

$\bullet$ {\sc Step 1}: Monotone convergence of $\{k_n\}_{n\in \mathbb{N}}$.\\
Define the following function
$$f(x):=\frac{1}{3+2\alpha-2x},\quad x\in \big(-\infty,\frac{3}{2}+\alpha\big).$$
Then, by \eqref{E-sexual-reproduction-coefficients} $\{k_n\}_{n\in \mathbb{N}}$ obeys the following recursive relation
$$k_n=f(k_{n-1}),\quad n>1.$$
Let us consider the fixed points $x_-<x_+$ of $f$, i.e,
$$x_{\pm}=\frac{(3+2\alpha)\pm\sqrt{(1+2\alpha)^2+8\alpha}}{4}.$$
By inspection, it is clear that $0<\bk_\alpha=x_-<x_+<\frac{3}{2}+\alpha$ and
$$
x_-<f(x)<x, \quad \text{if }x\in (x_-,x_+),
$$
where $\bk_\alpha$ is given in \eqref{E-sexual-reproduction-coefficients-limit}. Since $k_1\in (x_-,x_+)$, then we conclude that $\{k_n\}_{n\in \mathbb{N}}$ is a well defined, positive and monotonically decreasing sequence contained in the compact interval $[x_-,x_+]$. Therefore, it must converge towards some limit $\ell$, which is a solution of $f(x)=x$ {\em i.e.}, $\ell\in \{x_-,x_+\}$. Since the full sequence $\{k_n\}_{n\in \mathbb{N}}$ is below $x_+$ and decreasing, we then conclude that $\ell=x_-=\bk_\alpha$. In particular, we obtain
\begin{equation}\label{E-contraction-coefficients-uniform-bound}
\bk_\alpha<k_n\leq k_1,
\end{equation}
for any $n\in \mathbb{N}$.

\medskip

$\bullet$ {\sc Step 2}: Convergence rates.\\
As for Lemma \ref{L-sexual-reproduction-Gaussian-solution-time-discrete-variance-relaxation},  we shall use the special algebraic structure of the function $f$, which is the restriction to $\big(-\infty,\frac{3}{2}+\alpha\big)$ of the M\"{o}bius function $M:\mathbb{R}\setminus\{\frac{3}{2}+\alpha\}\longrightarrow \mathbb{R}$ given by
$$M(x):=\frac{1}{3+2\alpha-2x},\quad x\in \mathbb{R}\setminus\left\{\frac{3}{2}+\alpha\right\}.$$
By definition of $\{k_n\}_{n\in \mathbb{N}}$ in \eqref{E-sexual-reproduction-coefficients} we obtain
$$k_n-x_+=\frac{1}{3+2\alpha-2k_{n-1}}-\frac{1}{3+2\alpha-2x_+}=\frac{2}{(3+2\alpha-2k_{n-1})(3+2\alpha-2x_+)}(k_{n-1}-x_+).$$
Similarly, we obtain 
$$k_n-x_-=\frac{1}{3+2\alpha-2k_{n-1}}-\frac{1}{3+2\alpha-2x_-}=\frac{2}{(3+2\alpha-2k_{n-1})(3+2\alpha-2x_-)}(k_{n-1}-x_-).$$
By dividing both expressions and iterating such an identity we get
\begin{equation}\label{E-contraction-coefficients-Mobius}
\frac{k_n-x_-}{k_n-x_+}=\left(\frac{3+2\alpha-2x_+}{3+2\alpha-2x_-}\right)^{n-1}\frac{k_1-x_-}{k_1-x_+},
\end{equation}
for every $n\in \mathbb{N}$. In fact, the basis can be restated as follows
$$\frac{3+2\alpha-2x_+}{3+2\alpha-2x_-}=\frac{3+2\alpha-\sqrt{(1+2\alpha)^2+8\alpha}}{3+2\alpha+\sqrt{(1+2\alpha)^2+8\alpha}}=\frac{8}{\left(3+2\alpha+\sqrt{(1+2\alpha)^2+8\alpha}\right)^2}\equiv \br_\alpha<1,$$
where $\br_\alpha$ is determined by \eqref{E-variance-contraction-ralpha} in Lemma \ref{L-sexual-reproduction-Gaussian-solution-time-discrete-variance-relaxation}. Therefore, using \eqref{E-contraction-coefficients-uniform-bound} and \eqref{E-contraction-coefficients-Mobius} we conclude that
$$k_n-\bk_\alpha=\frac{k_1-x_-}{x_+-k_1}(x_+-k_n)\br_\alpha^{n-1}\leq \frac{k_1-x_-}{x_+-k_1}(x_+-x_-)\br_\alpha^{n-1},$$
for any $n\in \mathbb{N}$, thus ending the proof.
\end{proof}

\begin{definition}[Quadratic forms]\label{D-sexual-reproduction-Gaussian-quadratic-Qn-form}
Consider the sequence $\{k_n\}_{n\in \mathbb{N}}$ in Definition \ref{D-sexual-reproduction-coefficients}. Then, we define the quadratic form $Q_n=Q_n(\by_n)$ by
\begin{equation}\label{E-sexual-reproduction-Gaussian-quadratic-Qn-form}
Q_n(\by_n):=\sum_{m=0}^{n-1}\sum_{i\in \Level_m^n}\left(\frac{1}{4 k_{n-m}}(\vert y_{i1}\vert^2+\vert y_{i2}\vert^2)-\frac{1}{8}\vert y_{i1}-y_{i2}\vert^2\right),\quad \by_n\in \mathbb{R}^{2(2^n-1)d},
\end{equation}
where we are using the tree-indexed notation $\by_n=(y_i)_{i\in \Tree^n_*}\in \mathbb{R}^{2(2^n-1)d}$ in Remark \ref{R-tree-indexed-variables}.
\end{definition}

Again, note that when $n=2$ the above quadratic form $Q_2(\by_2)$ reduces to the one inside the exponential of \eqref{E-sexual-reproduction-iterations_n=1_Gaussian-quadratic-change-2}. We now show the following uniform control of the quadratic forms $Q_n$ as $n\rightarrow\infty$.

\begin{lemma}[Uniform positive definite quadratic forms]\label{L-sexual-reproduction-control-Qn}
Consider the quadratic form $Q_n=Q_n(\by_n)$ in Definition \ref{D-sexual-reproduction-Gaussian-quadratic-Qn-form} and define the couple of coefficients $\beta_{\min},\beta_{\max}\in \mathbb{R}_+^*$ by
$$\beta_{\min}:=\frac{1+2\alpha}{4},\quad \beta_{\max}:=\frac{1}{4\bk_\alpha},$$
where $\bk_\alpha$ is given by formula \eqref{E-sexual-reproduction-coefficients-limit}. Then, we obtain that
$$\beta_{\min}\Vert \by_n\Vert^2\leq Q_n(\by_n)\leq \beta_{\max}\Vert \by_n\Vert^2,$$
where we are using the tree-indexed notation $\by_n=(y_i)_{i\in \Tree^n_*}\in \mathbb{R}^{2(2^n-1)d}$ in Remark \ref{R-tree-indexed-variables}.
\end{lemma}

\begin{proof}
By virtue of Young's inequality, we obtain the following lower and upper bound for $Q_n$
$$\sum_{m=0}^{n-1}\sum_{i\in \Level_m^n}\frac{1}{4}\left(\frac{1}{k_{n-m}}-1\right)(\vert y_{i1}\vert^2+\vert y_{i2}\vert^2)\leq Q_n(\by_n)\leq \sum_{m=0}^{n-1}\sum_{i\in \Level_m^n}\frac{1}{4}\frac{1}{k_{n-m}}(\vert y_{i1}\vert^2+\vert y_{i2}\vert^2),$$
for every $\by_n=(y_i)_{i\in \Tree^n_*}\in \mathbb{R}^{2(2^n-1)d}$. Then, the result follows from Lemma \ref{L-sexual-reproduction-coefficients-asymptotics}, which guarantees the uniform control $\bk_\alpha< k_{m-n}\leq k_1$ for every $m=0,\ldots,n-1$.
\end{proof}

\begin{definition}[Lineage maps]\label{D-sexual-reproduction-Gaussian-quadratic-lineage-maps}
We define the \textit{lineage maps} $\Phi_n^j=\Phi_n^j(x;\by_n)$ associated to the leave $j\in \Leaves^n$ (see Figure \ref{fig:lineage-map}) as follows
\begin{equation}\label{E-sexual-reproduction-Gaussian-quadratic-lineage-maps}
\Phi_n^j(x;\by_n):=\kappa_n x+\sum_{m=0}^{n-1} \kappa_m y_{\child^m(j)},\quad \by_n\in \mathbb{R}^{2(2^n-1)d},
\end{equation}
where $x\in \mathbb{R}^d$ represents the root value, $\by_n=(y_i)_{i\in \Tree^n_*}\in \mathbb{R}^{2(2^n-1)d}$ is represented according to the tree-indexed notation in Remark \ref{R-tree-indexed-variables}, $\{\kappa_n\}_{n\in \mathbb{N}}$ is given in Definition \ref{D-sexual-reproduction-coefficients}, and $\child^m=\child\circ \cdots\circ\child$ is the $m$ times iterated map $\child:\Tree^n_*\rightarrow \widehat{\Tree}^n$ which, to any vertex $i\in \Tree^n_*$, it associates its child $\child(i)\in \widehat{\Tree}^n$ ({\it cf.} Section \ref{SS-trees}).
\end{definition}

\begin{figure}[t]
\centering
\begin{forest}
for tree={grow=north,edge={->}}
[$\emptyset$ [2 [22 [222] [221]] [21 [212] [211]]] [1 [12 [122] [121, colour my roots=red]] [11 [112] [111]]]]
\end{forest}
\caption{Unique path joining the leaf $j=121$ to the root in the perfect binary tree $\Tree^3$. The corresponding lineage map takes the form $\Phi_n^{121}(x;{\mathbf y}_3)=\kappa_3x+\kappa_2 y_{1}+\kappa_1 y_{12}+y_{121}$.}
\label{fig:lineage-map}
\end{figure}
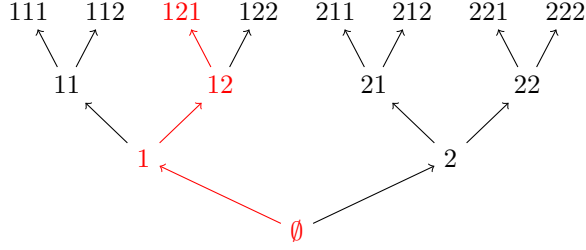

We are now ready to state the main result of this section extending formula \eqref{E-sexual-reproduction-iterations_n=1_Gaussian-quadratic-change-2} to any $n\in\mathbb{N}$. The starting point is again formula \eqref{E-sexual-reproduction-iterations}. In the particular case of Gaussian mixing kernel \eqref{E-mixing-Gaussian} and quadratic selection function \eqref{E-selection-quadratic}, it takes the explicit form
\begin{align}\label{E-sexual-reproduction-iterations-Gaussian-quadratic}
\begin{aligned}
&(e^m F_n)(x)=\frac{1}{(2\pi)^{(2^n-1)d/2}}\frac{1}{(4\pi)^{2^{n-1}d}}\frac{1}{\prod_{m=0}^{n-1}\Vert F_m\Vert_{L^1(\mathbb{R}^d)}^{2^{n-1-m}}}\\
&\quad\times \int_{\mathbb{R}^{2(2^n-1)d}}\exp\left(-\sum_{m=0}^{n-2}\sum_{i\in \Leaves^n_m}\left[\frac{\alpha}{2}(\vert x_{i1}\vert^2+\vert x_{i2}\vert^2)+\frac{1}{2}\left\vert x_i-\frac{x_{i1}+x_{i2}}{2}\right\vert^2\right]\right)\\
&\qquad\qquad\qquad\times \exp\left(-\sum_{i\in \Leaves^n_{n-1}}\left[\frac{1}{2}(\alpha+\frac{1}{2})(\vert x_{i1}\vert^2+\vert x_{i2}\vert^2)+\frac{1}{2}\left\vert x_i-\frac{x_{i1}+x_{i2}}{2}\right\vert^2\right]\right)\\
&\qquad\qquad\qquad\times \prod_{j\in \Leaves^n}\bar{F}_0(x_j)\,d\mathbf{x}_n,
\end{aligned}
\end{align}
for any $x\in \mathbb{R}^d$ and $n\in \mathbb{N}$. Above, we have used the tree-indexed notation $\bx_n=(x_i)_{i\in \Tree^n_*}\in \mathbb{R}^{2(2^n-1)d}$ in Remark \ref{R-tree-indexed-variables}, we have set $x_\emptyset:=x$ and we have considered the rescaled distribution $\bar{F}_0$ according to Definition \ref{D-rescaled-F}. Again, notice that the exponential terms in the fist term of the third line of \eqref{E-sexual-reproduction-iterations-Gaussian-quadratic} have been introduced to appropriately correct the rescaled distribution $\bar{F}_0$ of $F_0$. As a consequence, the factor $(4\pi)^{-2^{n-1}d}$ arises from the normalization by $\bF_{\alpha=0}(x_j)$ whilst the factor $(2\pi)^{-(2^n-1)d/2}$ comes from the repeated products of the Gaussian mixing kernel $G$. The main result then reads as follows

\begin{proposition}[Reformulation of the iterations I]\label{P-sexual-reproduction-iterations-Gaussian-quadratic-change-variables}
Assume that $\alpha\in \mathbb{R}_+$ and set any initial datum $F_0\in \mathcal{M}_+(\mathbb{R}^d)$. Hence, the solution $\{F_n\}_{n\in \mathbb{N}}$ to the time-discrete problem \eqref{E-sexual-reproduction-time-discrete} admits the following form
\begin{equation}\label{E-sexual-reproduction-iterations-Gaussian-quadratic-change-variables}
F_n(x)=\frac{e^{-\frac{1+\alpha-k_n}{2}\vert x\vert^2}}{(4\pi)^{2^{n-1}d}(2\pi)^{(2^n-1)d/2}}\frac{1}{\prod_{m=0}^{n-1}\Vert F_m\Vert_{L^1(\mathbb{R}^d)}^{2^{n-1-m}}}\int_{\mathbb{R}^{2(2^n-1)d}}e^{-Q_n(\by_n)}\prod_{j\in \Leaves^n}\bar{F}_0(\Phi_n^j(x;\by_n))\,d\mathbf{y}_n,
\end{equation}
for every $x\in \mathbb{R}^d$ and $n\in \mathbb{N}$, where we are using the tree-indexed notation $\by_n=(y_i)_{i\in \Tree^n_*}\in \mathbb{R}^{2(2^n-1)d}$ in Remark \ref{R-tree-indexed-variables} along with further notation from Definitions \ref{D-rescaled-F}, \ref{D-sexual-reproduction-coefficients}, \ref{D-sexual-reproduction-Gaussian-quadratic-Qn-form} and \ref{D-sexual-reproduction-Gaussian-quadratic-lineage-maps}.
\end{proposition}

As we anticipated for the particular case $n=2$ at the beginning of this section, the proof of Proposition \ref{P-sexual-reproduction-iterations-Gaussian-quadratic-change-variables} will be based on the following change of variables
$$
x_{i1}=k_{n-m}x_i+y_{i1},\quad x_{i2}=k_{n-m}x_i+y_{i2},
$$
for every $i\in \Level_m^n$, which we will apply in a backwards way starting at indices $i\in \Level_{n-1}^n$ and ending at $i=\emptyset$. Again, the objective of such a change of variables is to appropriately move the reference frame so that the above quadratic forms $\frac{1}{2}\left\vert x_i-\frac{x_{i1}+x_{i2}}{2}\right\vert^2$ inside the exponential of \eqref{E-sexual-reproduction-iterations-Gaussian-quadratic} get appropriately centered and no crossed terms remain when the selection parts $\frac{\alpha}{2}(\vert x_{i1}\vert^2+\vert x_{i2}\vert^2)$ are taken into account. For simplicity, along the proof we shall restrict to tracking how the various terms inside the exponentials in \eqref{E-sexual-reproduction-iterations-Gaussian-quadratic} get modified after the change of variables.

\begin{proof}[Proof of Proposition \ref{P-sexual-reproduction-iterations-Gaussian-quadratic-change-variables}]~\\

$\bullet$ {\sc Step 1}: Change of variables $x_{i1}\rightarrow y_{i1}$ and $x_{i2}\rightarrow y_{i2}$ for $i\in \Level_{n-1}^n$ .\\
We consider the change of variables given by
$$
x_{i1}=k_1x_i+y_{i1},\quad x_{i2}=k_1x_i+y_{i2},
$$
where $k_1$ is given by Definition \ref{D-sexual-reproduction-coefficients}. The collection of all the terms for $i\in \Level_{n-1}^n$ in \eqref{E-sexual-reproduction-iterations-Gaussian-quadratic} then read
\begin{align*}
\frac{\alpha_0}{2}\vert x_{i1}\vert^2+\frac{\alpha_0}{2}\vert x_{i2}\vert^2&+\frac{1}{2}\left\vert x_i-\frac{x_{i1}+x_{i2}}{2}\right\vert^2\\
&=\frac{\alpha_0}{2}\vert k_1x_i+y_{i1}\vert^2+\frac{\alpha_0}{2}\vert k_1 x_i+y_{i2}\vert^2+\frac{1}{2}\left\vert (1-k_1)x_i-\frac{y_{i1}+y_{i2}}{2}\right\vert^2\\
&=\frac{1}{2}\left(2\alpha_0 k_1^2+(1-k_1)^2\right)\vert x_i\vert^2+\frac{\alpha_0}{2}\vert y_{i1}\vert^2+\frac{\alpha_0}{2}\vert y_{i2}\vert^2+\frac{1}{8}\vert y_{i1}+y_{i2}\vert^2\\
&\hspace{2cm} +\left(\alpha_0 k_1-\frac{1-k_1}{2}\right)x_i\cdot y_{i1}+\left(\alpha_0 k_1-\frac{1-k_1}{2}\right)x_i\cdot y_{i2}\\
&=\frac{1}{2}\left(2\alpha_0 k_1^2+(1-k_1)^2\right)\vert x_i\vert^2+\frac{\alpha_0}{2}\vert y_{i1}\vert^2+\frac{\alpha_0}{2}\vert y_{i2}\vert^2+\frac{1}{8}\vert y_{i1}+y_{i2}\vert^2\\
&=\frac{1}{2}(1-k_1)\vert x_i\vert^2+\frac{1}{4 k_1}\vert y_{i1}\vert^2+\frac{1}{4 k_1}\vert y_{i2}\vert^2-\frac{1}{8}\vert y_{i1}-y_{i2}\vert^2,
\end{align*}
where we have defined the coefficient $\alpha_0\in \mathbb{R}_+^*$ by
$$\alpha_0:=\alpha+\frac{1}{2},$$
to recombine terms in the second and third lines of \eqref{E-sexual-reproduction-iterations-Gaussian-quadratic}, and we have used the following relation between $k_1$ and $\alpha_0$ in order to cancel the crossed terms, {\em i.e.},
$$k_1=\frac{1}{1+2\alpha_0}.$$

\medskip

$\bullet$ {\sc Step 2}: Change of variables $x_{i1}\rightarrow y_{i1}$ and $x_{i2}\rightarrow y_{i2}$ for $i\in \Level_{n-2}^n$.\\
Again, we consider the change of variables given by
$$
x_{i1}=k_2x_i+y_{i1},\quad x_{i2}=k_2x_i+y_{i2},
$$
where $k_2$ is given by Definition \ref{D-sexual-reproduction-coefficients}. Putting together the terms for $i\in \Level_{n-2}^n$ in \eqref{E-sexual-reproduction-iterations-Gaussian-quadratic} and the above $x_i$-dependent remainder in the above expression in {\sc Step 1} yields the following term under the above change of variables
\begin{align*}
\frac{\alpha_1}{2}\vert x_{i1}\vert^2+\frac{\alpha_1}{2}\vert x_{i2}\vert^2&+\frac{1}{2}\left\vert x_i-\frac{x_{i1}+x_{i2}}{2}\right\vert^2\\
&=\frac{\alpha_1}{2}\vert k_2x_i+y_{i1}\vert^2+\frac{\alpha_1}{2}\vert k_2 x_i+y_{i2}\vert^2+\frac{1}{2}\left\vert (1-k_2)x_i-\frac{y_{i1}+y_{i2}}{2}\right\vert^2\\
&=\frac{1}{2}\left(2\alpha_1 k_2^2+(1-k_2)^2\right)\vert x_i\vert^2+\frac{\alpha_1}{2}\vert y_{i1}\vert^2+\frac{\alpha_1}{2}\vert y_{i2}\vert^2+\frac{1}{8}\vert y_{i1}+y_{i2}\vert^2\\
&\hspace{2cm} +\left(\alpha_1 k_2-\frac{1-k_2}{2}\right)x_i\cdot y_{i1}+\left(\alpha_1 k_2-\frac{1-k_2}{2}\right)x_i\cdot y_{i2}\\
&=\frac{1}{2}\left(2\alpha_1 k_2^2+(1-k_2)^2\right)\vert x_i\vert^2+\frac{\alpha_1}{2}\vert y_{i1}\vert^2+\frac{\alpha_1}{2}\vert y_{i2}\vert^2+\frac{1}{8}\vert y_{i1}+y_{i2}\vert^2\\
&=\frac{1}{2}(1-k_2)\vert x_i\vert^2+\frac{1}{4 k_2}\vert y_{i1}\vert^2+\frac{1}{4 k_2}\vert y_{i2}\vert^2-\frac{1}{8}\vert y_{i1}-y_{i2}\vert^2,
\end{align*}
where we have defined again coefficient the $\alpha_1\in \mathbb{R}_+^*$ by
$$\alpha_1:=1+\alpha-k_1,$$
in order to absorb the above-mentioned $x_i$-dependent remainder. In addition, note that we have used again the following relation between $k_2$ and $\alpha_1$ to cancel the crossed term, {\em i.e.},
$$k_2=\frac{1}{1+2\alpha_1}.$$

\medskip

$\bullet$ {\sc Step 3}: Change of variables $x_{i1}\rightarrow y_{i1}$ and $x_{i2}\rightarrow y_{i2}$ for $i\in \Level_{n-3}^n$.\\
We consider the change of variables given by
\begin{align*}
x_{i1}&=k_3x_i+y_{i1},\\
x_{i2}&=k_3x_i+y_{i2},
\end{align*}
where $k_3$ is given by Definition \ref{D-sexual-reproduction-coefficients}. Putting together the terms for $i\in \Level_{n-3}^n$ in \eqref{E-sexual-reproduction-iterations-Gaussian-quadratic} and the above $x_i$-dependent remainder in the above expression in {\sc Step 2} yields the following term under the above change of variables
\begin{align*}
\frac{\alpha_2}{2}\vert x_{i1}\vert^2+\frac{\alpha_2}{2}\vert x_{i2}\vert^2&+\frac{1}{2}\left\vert x_i-\frac{x_{i1}+x_{i2}}{2}\right\vert^2\\
&=\frac{\alpha_2}{2}\vert k_3x_i+y_{i1}\vert^2+\frac{\alpha_2}{2}\vert k_3 x_i+y_{i2}\vert^2+\frac{1}{2}\left\vert (1-k_3)x_i-\frac{y_{i1}+y_{i2}}{2}\right\vert^2\\
&=\frac{1}{2}\left(2\alpha_2 k_3^2+(1-k_3)^2\right)\vert x_i\vert^2+\frac{\alpha_2}{2}\vert y_{i1}\vert^2+\frac{\alpha_2}{2}\vert y_{i2}\vert^2+\frac{1}{8}\vert y_{i1}+y_{i2}\vert^2\\
&\hspace{2cm} +\left(\alpha_2 k_3-\frac{1-k_3}{2}\right)x_i\cdot y_{i1}+\left(\alpha_2 k_3-\frac{1-k_3}{2}\right)x_i\cdot y_{i2}\\
&=\frac{1}{2}\left(2\alpha_2 k_3^2+(1-k_3)^2\right)\vert x_i\vert^2+\frac{\alpha_2}{2}\vert y_{i1}\vert^2+\frac{\alpha_2}{2}\vert y_{i2}\vert^2+\frac{1}{8}\vert y_{i1}+y_{i2}\vert^2\\
&=\frac{1}{2}(1-k_3)\vert x_i\vert^2+\frac{1}{4 k_3}\vert y_{i1}\vert^2+\frac{1}{4 k_3}\vert y_{i2}\vert^2-\frac{1}{8}\vert y_{i1}-y_{i2}\vert^2,
\end{align*}
where we have defined again coefficient the $\alpha_2\in \mathbb{R}_+^*$
$$\alpha_2:=1+\alpha-k_2,$$
in order to absorb the above-mentioned $x_i$-dependent remainder. In addition, note that we have used again the following relation between $k_3$ and $\alpha_2$ to cancel the crossed term, {\em i.e.},
$$k_3=\frac{1}{1+2\alpha_2}.$$
Following a similar recursive process, we readily identify the quadratic form $Q_n$ in Definition \ref{D-sexual-reproduction-Gaussian-quadratic-Qn-form} in the exponential of the integrand in the final expression \eqref{E-sexual-reproduction-iterations-Gaussian-quadratic-change-variables}.

\medskip

$\bullet$ {\sc Step 4}: Identifying the lineage maps $\Phi_n^j$ and the exponential factor.\\
On the one hand, the hierarchy of changes of variables immediately leads to
$$x_j=\Phi_n^j(x;\by_n),$$
for any leaf $j\in \Leaves^n$, thanks to Definition \ref{D-sexual-reproduction-coefficients} for the coefficients $\{\kappa_n\}_{n\in \mathbb{N}}$ and Definition \ref{D-sexual-reproduction-Gaussian-quadratic-lineage-maps} for the lineage maps $\Phi_n^j$. On the other hand, since the Jacobian determinant of each change of variables is $1$, no further factors appear during the iterative process except for the exponential $x$-dependent remainder $e^{-\frac{1+\alpha-k_n}{2}\vert x\vert^2}$ of the last step of the recurrence, which is not absorbed in the quadratic form $Q_n$.
\end{proof}

\subsection{Probabilistic reinterpretation}

For general purposes, in this part we reformulate the result in Proposition \ref{P-sexual-reproduction-iterations-Gaussian-quadratic-change-variables} in appropriate probabilistic terms. More precisely, this reformulation will not only provide shorter formulas, but it will also allow identifying a problematic key point when studying the asymptotics of the high-dimensional integral \eqref{E-sexual-reproduction-iterations-Gaussian-quadratic-change-variables} in Section \ref{S-ergodicity}, namely, the presence of non-negligible correlations between the various factors indexed with indexes over leaves $j\in \Leaves^n$ of the tree, which do not dissipate even for long time $n\rightarrow\infty$. We shall use the following notation.

\begin{definition}[High-dimensional normal distributions]\label{D-sexual-reproduction-multivariate-normal-Gn}
Consider any integer $n\in \mathbb{N}$ and
\begin{equation}\label{E-sexual-reproduction-multivariate-normal-Gn-covariances}
\Sigma_n^m:=\left(\begin{array}{cc}
\displaystyle \frac{(2-k_{n-m})k_{n-m}}{1-k_{n-m}} I_d & \displaystyle-\frac{k_{n-m}^2}{1-k_{n-m}} I_d\\
\displaystyle-\frac{k_{n-m}^2}{1-k_{n-m}} I_d & \displaystyle\frac{(2-k_{n-m})k_{n-m}}{1-k_{n-m}} I_d
\end{array}\right),
\end{equation}
for every $m=0,\ldots,n-1$. Then, we will define the random vector $\bY_n=(Y_i)_{i\in \Tree^n_*}$ distributed according to the following multivariate normal distribution
\begin{equation}\label{E-sexual-reproduction-multivariate-normal-Gn}
G_n(\by_n):=\prod_{m=0}^{n-1}\prod_{i\in \Level_m^n}\frac{1}{(2\pi)^d\sqrt{\det(\Sigma_n^m)}}\exp\left(-\frac{1}{2}(y_{i1}^T,y_{i2}^T)(\Sigma_n^m)^{-1}\binom{y_{i1}}{y_{i2}}\right),\quad \by_n\in\mathbb{R}^{2(2^n-1)d},
\end{equation}
where again we are using the tree-indexed notation $\by_n=(y_i)_{i\in \Tree^n_*}\in \mathbb{R}^{2(2^n-1)d}$ in Remark \ref{R-tree-indexed-variables}.
\end{definition}

Notice that for every index $i\in \Level_m^n$ with $m=0,\ldots,n-1$ the pair $(Y_{i1},Y_{i2})$ is normally distributed according to $\mathcal{N}(0,\Sigma_n^m)$ and $Y_{i1}$ and $Y_{i2}$ are certainly correlated. Nevertheless, the vector $(Y_{i1},Y_{i2})$ is independent on any other random vector $Y_j$ for indices $j\neq i1$ and $j\neq i2$. We are now ready to introduce the following probabilistic reformulation of  Proposition \ref{P-sexual-reproduction-iterations-Gaussian-quadratic-change-variables}.

\begin{proposition}[Reformulation of the iterations II]\label{P-iterative-method-sexual-reproduction}
Assume that $\alpha\in \mathbb{R}_+^*$ and set any initial datum $F_0\in \mathcal{M}_+(\mathbb{R}^d)$. Hence, the solution $\{F_n\}_{n\in \mathbb{N}}$ to the time-discrete problem \eqref{E-sexual-reproduction-time-discrete} admits the following form
\begin{equation}\label{E-sexual-Gaussian-quadratic-emF}
F_n(x)=\frac{e^{-\frac{1+\alpha-k_n}{2}\vert x\vert^2}}{(2\pi)^{d/2}2^{2^{n-1}d}}\left(\prod_{m=0}^{n-1}\frac{\left(\frac{4k_{n-m}^2}{1-k_{n-m}}\right)^{2^{m-1}d}}{\Vert F_m\Vert_{L^1(\mathbb{R}^d)}^{2^{n-1-m}}}\right)\mathbb{E}\left[\prod_{j\in \Leaves^n} \bar{F}_0(\Phi^j_n(x;\bY_n))\right],
\end{equation}
for every $x\in \mathbb{R}^d$ and $n\in \mathbb{N}$, where we are using the high-dimensional random vector $\bY_n=(Y_i)_{i\in \Tree^n_*}$ in Definition \ref{D-sexual-reproduction-multivariate-normal-Gn} along with further notation from Definitions \ref{D-rescaled-F}, \ref{D-sexual-reproduction-coefficients}, \ref{D-sexual-reproduction-Gaussian-quadratic-Qn-form} and \ref{D-sexual-reproduction-Gaussian-quadratic-lineage-maps}.
\end{proposition}

\begin{proof}
By Proposition \ref{P-sexual-reproduction-iterations-Gaussian-quadratic-change-variables} we obtain
\begin{align*}
F_n(x)&=\frac{e^{-\frac{1+\alpha-k_n}{2}\vert x\vert^2}}{(4\pi)^{2^{n-1}d}(2\pi)^{(2^n-1)d/2}}\frac{1}{\prod_{m=0}^{n-1}\Vert F_m\Vert_{L^1(\mathbb{R}^d)}^{2^{n-1-m}}}\int_{\mathbb{R}^{2(2^n-1)d}}e^{-Q_n(\by_n)}\prod_{j\in \Leaves^n}\bar{F}_0(\Phi_n^j(x;\by_n))\,d\mathbf{y}_n\\
&=\frac{(2\pi)^{(2^n-1)d/2}}{(4\pi)^{2^{n-1}d}}e^{-\frac{1+\alpha-k_n}{2}\vert x\vert^2}\left(\prod_{m=0}^{n-1}\frac{\left(\frac{4k_{n-m}^2}{1-k_{n-m}}\right)^{2^{m-1}d}}{\Vert F_m\Vert_{L^1(\mathbb{R}^d)}^{2^{n-1-m}}}\right)\int_{\mathbb{R}^{2(2^n-1)d}}G_n(\by_n)\prod_{j\in \Leaves^n}\bar{F}_0(\Phi_n^j(x;\by_n))\,d\mathbf{y}_n,
\end{align*}
for every $x\in \mathbb{R}^d$ and each $n\in \mathbb{N}$, where in the last line we have used that the inverse and determinant of the covariance matrices $\Sigma_n^m$ in \eqref{E-sexual-reproduction-multivariate-normal-Gn-covariances} take the form
$$
(\Sigma_n^m)^{-1}:=\left(\begin{array}{cc}
\displaystyle\frac{2-k_{n-m}}{4 k_{n-m}} I_d & \displaystyle\frac{1}{4}I_d\\
\displaystyle\frac{1}{4}I_d & \displaystyle\frac{2-k_{n-m}}{4 k_{n-m}} I_d
\end{array}\right),\quad \det(\Sigma_n^m)=\left(\frac{4 k_{n-m}^2}{1-k_{n-m}}\right)^d.
$$
Then, the result follows by applying the law of the unconscious statistician (LOTUS) to relate the preceding integral to the expectation of the random variable $\prod_{j\in \Leaves^n}\bar{F}_0(\Phi_n^j(x;\bY_n))$.
\end{proof}

As illustrated in the overall map in Figure \ref{fig:overall-map-proof}, the reformulation of the iterations in Propositions \ref{P-sexual-reproduction-iterations-Gaussian-quadratic-change-variables} and \ref{P-iterative-method-sexual-reproduction} will be fundamental in order to characterize the long-time behavior of the solution $\{F_n\}_{n\in \mathbb{N}}$ to the time-discrete problem \eqref{E-sexual-reproduction-time-discrete}. More specifically, we need to unravel the asymptotic behavior of the high-dimensional integral encoded in the expectation term \eqref{E-sexual-Gaussian-quadratic-emF}. As discussed in Sections \ref{SS-log-lipschitz-contraction} and \ref{SS-strategy}, one is able to control it in the log-Lipschitz norm under stringent constraints on the initial data, which are not fully satisfactory. In view of the structure of the expectation term, alternatively one may expect to be able to interchange expectations with products and control each factor separately. However, as we show below, this naive idea is bound to fail since the involved random variables are correlated, and they do not uncorrelate even in the limit $n\rightarrow\infty$. This complicated structure imposes severe problems when handling formula \eqref{E-sexual-Gaussian-quadratic-emF}, and one needs more powerfull ideas, which we introduce later in Lemma \ref{L-sexual-reproduction-weak-dependency}.

\begin{remark}[Non-negligible correlations]\label{R-correlations}
Since $\bY_n$ in Definition \ref{D-sexual-reproduction-multivariate-normal-Gn} is normally distributed and the lineage map $\Phi^j_n(x;\by_n)$ in Definition \ref{D-sexual-reproduction-Gaussian-quadratic-lineage-maps} are affine transformations of $\by_n$, then all $(\Phi^j_n(x;\bY_n))_{j\in \Leaves^n}$ are also normally distributed. Unfortunately, the components of $(\Phi^j_n(x;\bY_n))_{j\in \Leaves^n}$ are correlated since, in particular, all the components $(i1,i2)$ are correlated in view of the structure \eqref{E-sexual-reproduction-multivariate-normal-Gn-covariances} of their covariance matrices. Specifically, a straightforward computation shows that the covariance matrix 
$$K_{ij}=\mathbb{E} \left[(\Phi_n^i(x;\bY_n)-\mathbb{E}\,\Phi_n^i(x;\bY_n))\otimes (\Phi_n^j(x;\bY_n)-\mathbb{E}\,\Phi_n^j(x;\bY_n))\right],$$ 
between any couple $(i,j)\in \Leaves^n\times \Leaves^n$ of components (see Table \ref{tab:list-notations} for the notation $\otimes$ of the Kronecker product) takes the following explicit form
$$
K_{ij}=\left\{\begin{array}{ll}
\displaystyle\left(\sum_{q=0}^{n-1}\kappa_q^2\frac{(2-k_{q+1})k_{q+1}}{1-k_{q+1}}\right)I_d, & \mbox{if}\quad \vert i\wedge j\vert=n,\\
\displaystyle\left(-\kappa_{n-\vert i\wedge j\vert-1}^2\frac{k_{n-\vert i\wedge j\vert}^2}{1-k_{n-\vert i\wedge j\vert}}+\sum_{q=n-\vert i\wedge j\vert}^{n-1} \kappa_q^2\frac{(2-k_{q+1})k_{q+1}}{1-k_{q+1}}\right) I_d, & \mbox{if}\quad 0<\vert i\wedge j\vert<n,\\
\displaystyle\left(-\kappa_{n-1}^2\frac{k_n^2}{1-k_n}\right) I_d, & \mbox{if}\quad \vert i\wedge j\vert=0.
\end{array}\right.
$$
Here, $\vert i\wedge j\vert$ represents the level of the highest common descendant $i\wedge j$ of $i$ and $j$ ({\em cf.} Section \ref{SS-trees}). Note that correlations only disappears for leaves $(i,j)$ such that $n-\vert i\wedge j\vert\rightarrow\infty$, that is, when the length of the path from $i$ (or $j$) to the highest common descendant $i\wedge j$ diverges as the height $n$ of the tree increases. Unfortunately, correlations are non-negligible in all other cases even for $n\rightarrow\infty$.
\end{remark}

\section{Uniqueness of equilibria and quantitative ergodicity property}\label{S-ergodicity}
In this part we shall prove our main results, namely, Theorem \ref{T-main} and Corollary \ref{C-main}. First, we will derive a local convergence result as $n\rightarrow\infty$ of the profiles $\{F_n\}_{n\in \mathbb{N}}$ solving \eqref{E-sexual-reproduction-time-discrete} with generic initial datum $F_0\in \mathcal{M}_+(\mathbb{R}^d)$ towards the Gaussian profile $\bF_\alpha$ in \eqref{E-sexual-reproduction-Gaussian-solution}. This weaker local analogue of Theorem \ref{T-main} will be enough to derive the uniqueness of the solution $(\blambda_\alpha,\bF_\alpha)$ of the eigenproblem \eqref{E-eigenproblem-time-discrete} in Corollary \ref{C-main}. Second, we will extend our local result into our main global result in Theorem \ref{T-main} by quantifying the relaxation of the growth rate $\Vert F_n\Vert_{L^1(\mathbb{R}^d)}/\Vert F_{n-1}\Vert_{L^1(\mathbb{R}^d)}$ and the normalized profiles $F_n/\Vert F_n\Vert_{L^1(\mathbb{R}^d)}$ towards the unique solution $(\blambda_\alpha,\bF_\alpha)$ of the eigenproblem \eqref{E-eigenproblem-time-discrete}. Our strategy will exploit the high-dimensional iterative formula \eqref{E-sexual-Gaussian-quadratic-emF} in Proposition \ref{P-iterative-method-sexual-reproduction}.

\subsection{Local convergence result}\label{SS-local-convergence-result}

In this part we shall prove our local convergence result. To this end, we will show that the dependency of the expectation term in \eqref{E-sexual-Gaussian-quadratic-emF} on the variable $x$ decays to zero as $n\rightarrow\infty$ uniformly over compact sets. To simplify notation, we shall define the following operator.

\begin{definition}[Expectation operators]\label{D-expectation-operator}
Under the notation in Definitions \ref{D-sexual-reproduction-Gaussian-quadratic-lineage-maps} and \ref{D-sexual-reproduction-multivariate-normal-Gn}, we define
\begin{equation}\label{E-expectation-operator}
\mathcal{E}_n[F](x):=\mathbb{E}\left[\prod_{j\in \Leaves^n}F(\Phi^j_n(x;\bY_n))\right],\quad x\in \mathbb{R}^d,
\end{equation}
for any $n\in \mathbb{N}$ and any non-negative measurable function $F:\mathbb{R}^d\rightarrow \mathbb{R}$.
\end{definition}

Note that $\mathcal{E}_n[F](x)$ possibly takes extended values in $(-\infty,+\infty]$ if the expectation does not exist. However, if $F$ is a bounded function (and it will be often the case), we obtain that $\mathcal{E}_n[F](x)$ always exists and is finite. The following technical lemma for the tail of the incomplete Gamma function will be required later to determine the asymptotic behavior of the above operator $\mathcal{E}_n$.

\begin{lemma}\label{L-estimate-incomplete-gamma-function}
Consider the (upper) incomplete Gamma function with parameter $a\in \mathbb{R}_+^*$, {\em i.e.},
\begin{equation}\label{E-incomplete-gamma-function}
\Gamma(a,x):=\int_x^{+\infty} s^{a-1} e^{-s}\,ds,\quad x\in \mathbb{R}_+^*.
\end{equation}
Then, the following estimate holds true
\begin{equation}\label{E-estimate-incomplete-gamma-function}
\Gamma(a+1,x)\leq 2e^{-x/2},
\end{equation}
for every $x,a\in \mathbb{R}_+^*$ verifying the constraints
\begin{equation}\label{E-estimate-incomplete-gamma-function-hypothesis}
x>e, \quad 2a<\frac{x}{\log(x)}.
\end{equation}
\end{lemma}

\begin{proof}
Let us restate the incomplete Gamma function as follows
$$\Gamma(a+1,x)=\int_x^{+\infty}\exp(-\psi_a(s) s)\,ds,$$
where $\psi_a(s)$ is the function
$$\psi_a(s):=1-a\frac{\log(s)}{s},\quad s>0.$$
Notice that $\psi_a$ is monotonically increasing for $s\in [e,\infty)$. Hence, we obtain that
$$\psi_a(s)\geq \psi_a(x)\geq \frac{1}{2},$$
for every $s\geq x$, where we have used \eqref{E-estimate-incomplete-gamma-function-hypothesis} on each of the above inequalities. Consequently, we get
$$\Gamma(a+1,x)\leq \int_x^{+\infty} e^{-\frac{s}{2}}\,ds,$$
and this ends the proof.
\end{proof}

Therefore, we obtain the following weak dependency of $\mathcal{E}_n[F]$ on the variable $x$ as $n\rightarrow \infty$.

\begin{lemma}[Asymptotics of $\mathcal{E}_n$]\label{L-sexual-reproduction-weak-dependency}
Assume that $\alpha\in \mathbb{R}_+^*$ and consider any $F\in C^1(\mathbb{R}^d)$ such that
\begin{align}\label{E-sexual-reproduction-F-hypothesis}
\begin{aligned}
0<F(x)&\leq C,\\
\vert \nabla \log F(x)\vert&\leq D(1+\vert x\vert),
\end{aligned}
\end{align}
for every $x\in \mathbb{R}^d$ and appropriate $C,D\in \mathbb{R}_+^*$. Then, for every $\varepsilon\in \mathbb{R}_+^*$, there exists a large enough constant $C_\varepsilon\in \mathbb{R}_+^*$ such that the following property holds true
\begin{equation}\label{E-sexual-reproduction-weak-dependency}
\vert \nabla \log \mathcal{E}_n[F](x)\vert\leq C_\varepsilon\left((2\bk_\alpha+\varepsilon)^n+(\br_\alpha+\varepsilon)^n\vert x\vert+e^{-\frac{(2+\varepsilon)^n}{\varepsilon}}e^{C_\varepsilon(2\br_\alpha+\varepsilon)^n\vert x\vert^2}\right),
\end{equation}
for every $x\in \mathbb{R}^d$ and any $n\in \mathbb{N}$, where $\mathcal{E}_n$ is the operator in Definition \ref{D-expectation-operator}.
\end{lemma}

\begin{proof}
Before proving the main result, note that the hypothesis \eqref{E-sexual-reproduction-F-hypothesis} guarantee that $F$ is bounded below by a Gaussian profile. Indeed, by the fundamental theorem of calculus we obtain the following relation
$$\log F(x)= \log F(0)+\int_0^1\nabla (\log F)(\theta x)\cdot x\,d\theta,$$
for any $x\in \mathbb{R}^d$. Using the linear growth assumption in \eqref{E-sexual-reproduction-F-hypothesis} for the log-derivative $\nabla \log F$ we obtain that
$$
\log F(x)\geq \log F(0)-D\vert x\vert\int_0^1(1+\theta\vert x\vert)\,d\theta \geq \log F(0)-\frac{D}{2}-D\vert x\vert^2,
$$
for any $x\in \mathbb{R}^d$, where in the last inequality we have integrated with respect to $\theta$ and we have used Young's inequality to interpolate $\vert x\vert$ by $\vert x\vert^2$. Therefore, we obtain the following Gaussian lower bound
\begin{equation}\label{E-sexual-reproduction-F-hypothesis-lower}
F(x)\geq c e^{-\beta\vert x\vert^2},
\end{equation}
for each $x\in \mathbb{R}^d$ and some $c,\beta\in \mathbb{R}_+^*$ depending only on $F(0)$ and $D$. 

For simplicity of notation we define 
$$E_n:=\mathcal{E}_n[F],$$
for any $n\in \mathbb{N}$. In the sequel, we shall study the behavior of $\nabla\log E_n$ as $n\rightarrow \infty$.

\medskip

$\bullet$ {\sc Step 1}: Control of $\sum_{j\in \Leaves^n}\vert\Phi^j_n(x;\by_n)\vert$.\\
By definition \eqref{E-sexual-reproduction-Gaussian-quadratic-lineage-maps} of the lineage map $\Phi^j_n$, we obtain that
$$\sum_{j\in \Leaves^n}\vert \Phi^j_n(x;\by_n)\vert=\sum_{j\in \Leaves^n}\left\vert \kappa_n x+\sum_{m=0}^{n-1}\kappa_m y_{\child^m(j)}\right\vert\leq 2^n\kappa_n\vert x\vert+\sum_{m=0}^{n-1}2^m\kappa_m \sum_{i\in \Level_{n-m}^n}\vert y_i\vert,$$
where we are using the tree-indexed notation $\by_n=(y_i)_{i\in \Tree^n_*}\in \mathbb{R}^{2(2^n-1)d}$ in Remark \ref{R-tree-indexed-variables} and $\child:\Tree^n_*\longrightarrow \widehat{\Tree}^n$ is defined in Section \ref{SS-trees}. Note that in the last inequality we have used that each index of the level $\Level_{n-m}^n$ in the tree contributes to the amount of $2^m$ leaves. Then, we obtain
\begin{align}\label{E-sexual-reproduction-barX-sum}
\begin{aligned}
\sum_{j\in \Leaves^n}\left\vert \Phi^j_n(x;\by_n)\right\vert & \leq 2^n\kappa_n|x|+\left(\sum_{m=0}^{n-1}\sum_{i\in \Leaves^n_{n-m}}(2^m\kappa_m)^2\right)^{1/2}\left(\sum_{m=0}^{n-1}\sum_{i\in \Leaves^n_{n-m}}|y_i|^2\right)^{1/2}\\
&= 2^n\kappa_n\vert x\vert+ \left(\sum_{m=0}^{n-1}2^{n-m}(2^m\kappa_m)^2\right)^{1/2}\Vert \by_n\Vert\\
&\leq 2^n\kappa_n\vert x\vert+\sqrt{2} \,2^{n/2} \Vert \by_n\Vert,
\end{aligned}
\end{align}
for any $\by_n\in \mathbb{R}^{2(2^n-1)d}$, where in the first line we have used the Cauchy--Schwarz inequality on the joint sum over indices $m$ and $i$, in the second line we have used the definition $\Vert \by_n\Vert=(\sum_{i\in \Tree^n_*}|y_i|^2)^{1/2}$ of the $\ell_2$ norm ({\it cf.} Table \ref{tab:list-notations}), and in the last line we have used that
$$\sum_{m=0}^{n-1} 2^{n-m}(2^m\kappa_m)^2\leq 2^n\sum_{m=0}^{n-1}2^m k_1^{2m}\leq 2^n\sum_{m=0}^\infty 2^m k_1^{2m}=\frac{2^n}{1-2k_1^2}\leq 2^{n+1}.$$
Above, we have used that $k_n\leq k_1$ for all $n\in \mathbb{N}$ by virtue of Lemma \ref{L-sexual-reproduction-coefficients-asymptotics}(i) and also that 
$$\frac{1}{1-2k_1^2}=\frac{1+4\alpha+2\alpha^2}{2(1+\alpha)^2}=\frac{2(1+\alpha^2)}{(1+\alpha)^2+2\alpha+\alpha^2}<2.$$

\medskip

$\bullet$ {\sc Step 2}: Derivative of $\log E_n$.\\
Taking derivatives on \eqref{E-expectation-operator}, using the second property in \eqref{E-sexual-reproduction-F-hypothesis} along with \eqref{E-sexual-reproduction-barX-sum} entails
\begin{align}\label{E-sexual-reproduction-log-derivative-an}
\begin{aligned}
\left\vert \nabla \log E_n(x)\right\vert& \leq \kappa_n\frac{\mathbb{E}\left[\left(\sum_{j\in \Leaves^n}\vert \nabla\log F(\Phi^j_n(x;\bY_n))\vert\right)\prod_{j\in \Leaves^n}F(\Phi^j_n(x;\bY_n))\right]}
{\mathbb{E}\left[\prod_{j\in \Leaves^n}F(\Phi^j_n(x;\bY_n))\right]}\\
&\lesssim\kappa_n\frac{\mathbb{E}\left[\left(\sum_{j\in \Leaves^n}(1+\vert \Phi^j_n(x;\bY_n)\vert)\right)\prod_{j\in \Leaves^n}F(\Phi^j_n(x;\bY_n))\right]}{\mathbb{E}\left[\prod_{j\in \Leaves^n}F(\Phi^j_n(x;\bY_n))\right]}\\ 
&\lesssim 2^n\kappa_n+2^n\kappa_n^2\vert x\vert+2^{n/2}\kappa_n R_n(x),
\end{aligned}
\end{align}
for every $x\in \mathbb{R}^d$, where the remainder $R_n$ reads
$$R_n(x):=\frac{\mathbb{E}\left[\Vert \bY_n\Vert\prod_{j\in \Leaves^n} F(\Phi^j_n(x;\bY_n))\right]}{\mathbb{E}\left[\prod_{j\in \Leaves^n} F(\Phi^j_n(x;\bY_n))\right]},\quad x\in \mathbb{R}^d.$$
This estimation provides a first insight about the ergodicity property. Indeed the correlations, as measured by $\kappa_n$, decreases faster than $2^{-n}$ ({\em cf.} Lemma \ref{L-sexual-reproduction-coefficients-asymptotics}), meaning that the contribution $2^n \kappa_n$ decays fast to zero. It is also the case of the other contributions, locally in $x$, but the last one, namely $2^{n/2}\kappa_n R_n(x)$, requires more work.
Consider any $\sqrt{2}<q<\frac{1}{\sqrt{2}\bk_\alpha}$, which exists because $2\bk_\alpha<1$ for $\alpha\in \mathbb{R}_+^*$, and split $R_n(x)=R_n^1(x)+R_n^2(x)$ in terms of the functions
\begin{align}
R_n^1(x)&:=\frac{1}{Z_n(x)}\int_{\Vert \by_n\Vert\leq q^n} \Vert \by_n\Vert\prod_{j\in \Leaves^n}F(\Phi^j_n(x;\by_n)) e^{-Q_n(\by_n)}\,d\mathbf{y}_n,\label{E-sexual-reproduction-b1}\\
R_n^2(x)&:=\frac{1}{Z_n(x)}\int_{\Vert \by_n\Vert> q^n} \Vert \by_n\Vert\prod_{j\in \Leaves^n}F(\Phi^j_n(x;\by_n)) e^{-Q_n(\by_n)}\,d\mathbf{y}_n,\label{E-sexual-reproduction-b2}
\end{align}
and the normalization factor
\begin{equation}\label{E-sexual-reproduction-ZA}
Z_n(x):=\int_{\mathbb{R}^{2(2^n-1)d}}\prod_{j\in \Leaves^n}F(\Phi^j_n(x;\by_n))e^{-Q_n(\by_n)}\,d\mathbf{y}_n.
\end{equation}

\medskip

$\bullet$ {\sc Step 3}: Lower bound for $Z_n$.\\
Using the above lower bound by a Gaussian in \eqref{E-sexual-reproduction-F-hypothesis-lower} along with the upper bound for $Q_n$ in Lemma \ref{L-sexual-reproduction-control-Qn} implies the following estimate of $Z_n$ in \eqref{E-sexual-reproduction-ZA}:
\begin{align*}
Z_n(x)&\geq c^{2^n}\int_{\mathbb{R}^{2(2^n-1)d}}\exp\left(-\beta\sum_{j\in \Leaves^n} \vert \Phi^j_n(x;\by_n)\vert^2\right)\exp(-\beta_{\max} \Vert \by_n\Vert^2)\,d\mathbf{y}_n\\
&\geq c^{2^n} e^{-2^{2n+1}\kappa_n^2\beta \Vert x\Vert^2}\int_{\mathbb{R}^{2(2^n-1)d}} \exp\left(-(2^{n+2}\beta +\beta_{\max})\Vert \by_n\Vert^2\right)d\mathbf{y}_n\\
&=c^{2^n} e^{-2^{2n+1}\kappa_n^2\beta \vert x\vert^2} \left\vert \mathbb{S}^{2(2^n-1)d-1}\right\vert\int_0^{+\infty}r^{2(2^n-1)d-1} \exp\left(-(2^{n+2}\beta +\beta_{\max}) r^2\right)\,dr\\
&=\frac{c^{2^n} e^{-2^{2n+1}\kappa_n^2\beta \vert x\vert^2} }{2(2^{n+2}\beta+\beta_{\max})^{(2^n-1)d}} \left\vert \mathbb{S}^{2(2^n-1)d-1}\right\vert\int_0^{+\infty} s^{(2^n-1)d-1}e^{-s}\,ds\\
&=\frac{c^{2^n} e^{-2^{2n+1}\kappa_n^2\beta \vert x\vert^2} }{2(2^{n+2}\beta+\beta_{\max})^{(2^n-1)d}} \left\vert \mathbb{S}^{2(2^n-1)d-1}\right\vert\Gamma((2^n-1)d),
\end{align*}
where in the last step we have used the Gamma function, that is defined by
$$\Gamma(a):=\int_0^{+\infty}s^{a-1}e^{-s}\,ds,\quad a>0.$$
Recall that area of the hypersphere is $\vert \mathbb{S}^{d-1}\vert=\frac{2\pi^{d/2}}{\Gamma\left(\frac{d}{2}\right)}$. Then, we obtain
\begin{equation}\label{E-sexual-reproduction-ZA-lower-bound}
Z_n(x)\geq \frac{c^{2^n}\pi^{(2^n-1)d}}{(2^{n+2}\beta+\beta_{\max})^{(2^n-1)d}} e^{-2^{2n+1}\kappa_n^2\beta \vert x\vert^2},
\end{equation}
for every $x\in \mathbb{R}^d$.

\medskip

$\bullet$ {\sc Step 4}: Upper bound for $R_n^1$ and $R_n^2$.\\
On the one hand, notice that 
\begin{equation}\label{E-sexual-reproduction-b1-estimate}
R_n^1(x)\leq q^n,
\end{equation}
for all $x\in \mathbb{R}^d$ by definition of $R_n^1$ in \eqref{E-sexual-reproduction-b1}. On the other hand, using the lower bound of $Q_n$ in Lemma \ref{L-sexual-reproduction-control-Qn} we can estimate $R_n^2$ in \eqref{E-sexual-reproduction-b2} by
\begin{align}
R_n^2(x)&\leq \frac{C^{2^n}}{q^n Z_n(x)}\int_{\Vert \by_n\Vert\geq q^n}\Vert\mathbf{y}_n\Vert^2e^{-\beta_{\min}\Vert \by_n\Vert^2}\,d\mathbf{y}_n\nonumber\\
&=\frac{C^{2^n}}{q^n Z_n(x)}\left\vert \mathbb{S}^{2(2^n-1)d-1}\right\vert\int_{q^n}^{+\infty} r^{2(2^n-1)d+1}e^{-\beta_{\min} r^2}\,dr\nonumber\\
&=\frac{C^{2^n}}{q^n Z_n(x)}\left\vert \mathbb{S}^{2(2^n-1)d-1}\right\vert\frac{1}{2\beta_{\min}^{(2^n-1)d+1}}\int_{\beta_{\min}q^{2n}}^{+\infty} s^{(2^n-1)d}e^{-s}\,ds\nonumber\\
&=\frac{C^{2^n}\pi^{(2^n-1)d}}{q^n \beta_{\min}^{(2^n-1)d+1} Z_n(x)}\frac{\Gamma\left((2^n-1)d+1,\beta_{\min}q^{2n}\right)}{\Gamma((2^n-1)d)},\label{E-sexual-reproduction-b2-pre-estimate}
\end{align}
for all $x\in \mathbb{R}^d$, where we have used again the  incomplete Gamma function \eqref{E-incomplete-gamma-function} in Lemma \ref{L-estimate-incomplete-gamma-function}. Putting the lower estimate \eqref{E-sexual-reproduction-ZA-lower-bound} for $Z_n$ into \eqref{E-sexual-reproduction-b2-pre-estimate} yields
\begin{equation}\label{E-sexual-reproduction-b2-pre-estimate-2}
R_n^2(x)\lesssim e^{2^{2n+1}\kappa_n^2\beta \vert x\vert^2}\frac{C^{2^n} (2^{n+2}\beta+\beta_{\max})^{(2^n-1)d}}{q^n c^{2^n}\beta_{\min}^{(2^n-1)d}}\frac{\Gamma\left((2^n-1)d+1,\beta_{\min}q^{2n}\right)}{\Gamma((2^n-1)d)},
\end{equation}
for all $x\in \mathbb{R}^d$. Now, take $a=(2^n-1)d$ and $x=\beta_{\min}q^{2n}$ and notice that the constraints \eqref{E-estimate-incomplete-gamma-function-hypothesis} in Lemma \ref{L-estimate-incomplete-gamma-function} hold true for large enough $n\in \mathbb{N}$ as long as we choose $q>\sqrt{2}$. Hence, we obtain the estimate
$$\Gamma\left((2^n-1)d+1,\beta_{\min}q^{2n}\right)\leq 2e^{-\frac{\beta_{\min}}{2}q^{2n}},$$
for all $n\geq n_*=n_*(\beta_{\min},d,q)$. Using Stirling's formula for the remaining Gamma function in \eqref{E-sexual-reproduction-b2-pre-estimate-2}, we conclude the estimate
\begin{equation}\label{E-sexual-reproduction-b2-estimate}
R_n^2(x)\lesssim e^{2^{2n+1}\kappa_n^2\beta \vert x\vert^2}\frac{C^{2^n} (2^{n+2}\beta+\beta_{\max})^{(2^n-1)d}}{q^n c^{2^n}\beta_{\min}^{(2^n-1)d}}\frac{e^{-\frac{\beta_{\min}}{2}q^{2n}}}{(2^n-1)^{1/2} \left(\frac{(2^n-1)d}{e}\right)^{(2^n-1)d}},
\end{equation}
for every $x\in \mathbb{R}^d$ and each $n\geq n_*$.

\medskip

$\bullet$ {\sc Step 5}: Final conclusion.

Let us put the preceding estimates \eqref{E-sexual-reproduction-b1-estimate} and \eqref{E-sexual-reproduction-b2-estimate} into \eqref{E-sexual-reproduction-log-derivative-an} to achieve
\begin{equation}\label{E-sexual-reproduction-weak-dependency-pre}
\vert \nabla \log E_n(x)\vert\lesssim 2^n\kappa_n+2^{n/2}q^n\kappa_n +2^n\kappa_n^2\vert x\vert+\frac{M^{2^n}}{e^{\frac{\beta_{\min}}{2} q^{2n}}}\,e^{2^{2n+1}\kappa_n^2\beta \vert x\vert^2},
\end{equation}
for every $x\in \mathbb{R}^d$ and each $n\geq n_*$, where $M=M(C,c,\beta,\beta_{\min},\beta_{\max},d,q)$ is a universal constant that we have introduced to absorb all the double exponential factors of the form $a^{2^n}$ with $a>0$ in the fourth term along with any further lower order factor. On the one hand, choosing $q>\sqrt{2}$ sufficiently close to $\sqrt{2}$ and using \eqref{E-sexual-reproduction-coefficients-asymptotics-2} in Lemma \ref{L-sexual-reproduction-coefficients-asymptotics} we obtain the following asymptotics for the factors in the first three terms of \eqref{E-sexual-reproduction-weak-dependency-pre}:
\begin{align*}
2^n\kappa_n&\leq C_\varepsilon(2\bk_\alpha+\varepsilon)^n,\\
2^{n/2} q^n\kappa_n&\leq C_\varepsilon (2\bk_\alpha+\varepsilon)^n,\\
2^n\kappa_n^2&\leq C_\varepsilon(2(\bk_\alpha+\varepsilon)^2)^n\leq C_\varepsilon(\br_\alpha+\varepsilon)^n,
\end{align*}
for arbitrarily small $\varepsilon\in \mathbb{R}_+^*$ and sufficiently large $C_\varepsilon\in \mathbb{R}_+^*$. Note that in the last identity we have used the relation $\br_\alpha=2\bk_\alpha^2$. On the other hand, since $q$ has been taken larger  than $\sqrt{2}$ and arbitrarily close to $\sqrt{2}$, then the double exponential in the denominator of the last term in \eqref{E-sexual-reproduction-weak-dependency-pre} kills the double exponential on the numerator independently on the explicit value of $M$. Therefore, we obtain 
$$\frac{M^{2^n}}{e^{\frac{\beta_{\min}}{2}q^{2n}}}\leq C_\varepsilon\,e^{-\frac{(2+\varepsilon)^n}{\varepsilon}},$$
by appropriately increasing $C_\varepsilon$ if necessary. Similarly, we infer the following control for the remaining $x$-dependent exponential factor in the fourth term of \eqref{E-sexual-reproduction-weak-dependency-pre}
$$e^{2^{2n+1}\kappa_n^2\beta \vert x\vert^2}\leq e^{C_\varepsilon (4(\bk_\alpha+\varepsilon)^2)^n\vert x\vert^2}\leq e^{C_\varepsilon(2\br_\alpha+\varepsilon)^n\vert x\vert^2},$$
which yields the announced estimate \eqref{E-sexual-reproduction-weak-dependency}.
\end{proof}

\begin{remark}\label{R-sexual-reproduction-local-convergence-expectations}
Under the assumptions in Lemma \ref{L-sexual-reproduction-weak-dependency}, notice that the above result implies in particular that for every $\varepsilon\in \mathbb{R}_+^*$ and $R\in \mathbb{R}_+^*$, there exists a large enough constant $C_{\varepsilon,R}\in \mathbb{R}_+^*$ such that
$$\sup_{\vert x\vert\leq R}\vert \nabla\log\mathcal{E}_n[F](x)\vert\leq C_{\varepsilon,R}\,(2\bk_\alpha)^n,$$
for any $n\in \mathbb{N}$. In particular, $\nabla\log \mathcal{E}_n[F]$ converge to zero as $n\rightarrow \infty$ uniformly on compact sets.
\end{remark}

The above observation is the cornerstone to prove the following local version of our main Theorem \ref{T-main}.

\begin{corollary}[Local convergence result]\label{C-sexual-reproduction-local-convergence}
Assume that $\alpha\in \mathbb{R}_+^*$, set any initial datum $F_0\in \mathcal{M}_+(\mathbb{R}^d)$ and consider the solution $\{F_n\}_{n\in \mathbb{N}}$ to the time-discrete problem \eqref{E-sexual-reproduction-time-discrete}. Then, for every $\varepsilon\in \mathbb{R}_+^*$ and $R\in \mathbb{R}_+^*$, there exists a large enough constant $C_{\varepsilon,R}\in \mathbb{R}_+^*$ such that
\begin{equation}\label{E-sexual-reproduction-local-convergence}
\sup_{\vert x\vert\leq R}\left\vert\frac{F_n(x)}{\Vert F_n\Vert_{L^1(\mathbb{R}^d)}}-\bF_\alpha(x)\right\vert\leq C_{\varepsilon,R}\,(2\bk_\alpha)^n,
\end{equation}
for any $n\in \mathbb{N}$, where $\bF_\alpha$ is the Gaussian profile \eqref{E-sexual-reproduction-Gaussian-solution}. In particular, $F_n/\Vert F_n\Vert_{L^1(\mathbb{R}^d)}$ converge to $\bF_\alpha$ as $n\rightarrow \infty$ uniformly on compact sets.
\end{corollary}

In the proof of Corollary \ref{C-sexual-reproduction-local-convergence} we will require a careful control of the asymptotic behavior of the ratios $F_n(0)/\Vert F_n\Vert_{L^1(\mathbb{R}^d)}$ as $n\rightarrow \infty$, which we provide below. First, note that if $\{F_n\}_{n\in \mathbb{N}}$ is the solution \eqref{E-sexual-reproduction-Gaussian-solution-time-discrete} of the time-discrete problem \eqref{E-sexual-reproduction-time-discrete} issued at a Gaussian initial datum $F_0$ like in Proposition \ref{P-sexual-reproduction-Gaussian-solution-time-discrete}, we have
$$\frac{F_n(0)}{\Vert F_n\Vert_{L^1(\mathbb{R}^d)}}=G_{\mu_n,\sigma_n^2}(0)\rightarrow \bF_\alpha(0)=\frac{1}{(2\pi\bsigma_\alpha^2)^{d/2}},$$
as $n\rightarrow \infty$ explicitly, where $\mu_n$ and $\sigma_n^2$ are determined by the recursive relations \eqref{E-sexual-reproduction-Gaussian-solution-time-discrete-parameters} and we have used that $\mu_n\rightarrow 0$ and $\sigma_n^2\rightarrow\bsigma_\alpha^2$ according to Lemma \ref{L-sexual-reproduction-Gaussian-solution-time-discrete-variance-relaxation}. In the following lemma, we provide a similar quantitative result with convergence rates for generic initial data $F_0\in \mathcal{M}_+(\mathbb{R}^d)$.

\begin{lemma}\label{L-convergence-ratios-max-mass}
Assume that $\alpha\in \mathbb{R}_+^*$, set any initial datum $F_0\in \mathcal{M}_+(\mathbb{R}^d)$ and consider the solution $\{F_n\}_{n\in \mathbb{N}}$ to the time-discrete problem \eqref{E-sexual-reproduction-time-discrete}. For any $\varepsilon\in \mathbb{R}_+^*$ there exists $C_\varepsilon\in \mathbb{R}_+^*$ sufficiently large with
\begin{equation}\label{E-convergence-ratios-max-mass}
\left\vert \frac{F_n(0)}{\Vert F_n\Vert_{L^1(\mathbb{R}^d)}}-\bF_\alpha(0)\right\vert\leq C_\varepsilon(2\bk_\alpha+\varepsilon)^n,
\end{equation}
for any $n\in \mathbb{N}$.
\end{lemma}

\begin{proof}
Let us fix any $\varepsilon\in \mathbb{R}_+^*$ arbitrarily small so that $2\bk_\alpha+\varepsilon<1$. Again, this is possible because $2\bk_\alpha<1$ for $\alpha\in \mathbb{R}_+^*$. For simplicity of notation, we shall define the following sequence of coefficients:
$$
c_n:=\frac{F_n(0)}{\Vert F_n\Vert_{L^1(\mathbb{R}^d)}},\quad n\in \mathbb{N}.
$$
Our goal then reduces to studying the asymptotic behavior of $\{c_n\}_{n\in \mathbb{N}}$ and obtaining quantitative convergence rates. Note that by formula \eqref{E-sexual-Gaussian-quadratic-emF} in Proposition \ref{P-iterative-method-sexual-reproduction} for the recursion, we obtain
$$\frac{1}{c_n}=\frac{\int_{\mathbb{R}^d}e^{-\frac{1+\alpha-k_n}{2}\vert x\vert^2}\mathcal{E}_n[\bar F_0](x)\,dx}{\mathcal{E}_n[\bar F_0](0)}=\displaystyle\int_{\vert x\vert\leq R_n}e^{-\frac{1+\alpha-k_n}{2}\vert x\vert^2}\frac{\mathcal{E}_n[\bar F_0](x)}{\mathcal{E}_n[\bar F_0](0)}\,dx+\frac{1}{c_n}\int_{\vert x\vert>R_n}\frac{F_n(x)}{\Vert F_n\Vert_{L^1(\mathbb{R}^d)}}\,dx,$$
where we set the sequence of radii $\{R_n\}_{n\in \mathbb{N}}$ as follows
\begin{equation}\label{E-convergence-ratios-max-mass-radius}
R_n:=\frac{1}{(2\bk_\alpha)^{\frac{n}{2m+1}}},\quad n\in\mathbb{N},
\end{equation}
for a fixed value $m\in \mathbb{N}$, which we take large enough so that
\begin{equation}\label{E-convergence-ratios-max-mass-radius-choice-m}
(2\bk_\alpha)^\frac{2m}{2m+1}\leq 2\bk_\alpha+\frac{\varepsilon}{2}.
\end{equation}
Whilst our choice of $R_n$ is not justified at first glance, we claim that it has been taken as to minimize the decay rate on the error terms $\widetilde{E}_n$ and $E_n$ below. By solving the implicit equation on $c_n$ we infer
\begin{equation}\label{E-convergence-ratios-max-mass-decomposition}
c_n=\frac{1-\tilde{E}_n}{(2\pi\bsigma_\alpha^2)^{d/2}+E_n},
\end{equation}
for every $n\in \mathbb{N}$, where the error terms $\tilde{E}_n$ and $E_n$ take the form
\begin{align}\label{E-convergence-ratios-max-mass-errors}
\begin{aligned}
\tilde{E}_n&:=\int_{\vert x\vert>R_n}\frac{F_n(x)}{\Vert F_n\Vert_{L^1(\mathbb{R}^d)}}\,dx,\\
E_n&:=\int_{|x|\leq R_n}e^{-\frac{1+\alpha-k_n}{2}|x|^2}\frac{\mathcal{E}_n[\bar F_0](x)}{\mathcal{E}_n[\bar F_0](0)}\,dx-(2\pi\bsigma_\alpha^2)^{d/2},
\end{aligned}
\end{align}
Additionally, we can split the second error term as $E_n=E_{n,1}-E_{n,2}+E_{n,3}$ with
\begin{align}\label{E-convergence-ratios-max-mass-errors-decomposition}
\begin{aligned}
E_{n,1}&:=\left(\frac{2\pi}{1+\alpha-k_n}\right)^{d/2}-\left(\frac{2\pi}{1+\alpha-\bk_\alpha}\right)^{d/2},\\
E_{n,2}&:=\int_{\vert x\vert>R_n}e^{-\frac{1+\alpha-k_n}{2}\vert x\vert^2}\,dx,\\
E_{n,3}&:=\int_{\vert x\vert\leq R_n}e^{-\frac{1+\alpha-k_n}{2}\vert x\vert^2}\left(\frac{\mathcal{E}_n[\bar F_0](x)}{\mathcal{E}_n[\bar F_0](0)}-1\right)\,dx,
\end{aligned}
\end{align}
where we have used the relationship $\bsigma_\alpha^2(1+\alpha-\bk_\alpha) = 1$, which results from \eqref{E-sexual-reproduction-Gaussian-solution-variance-equation} and \eqref{E-sexual-reproduction-coefficients-limit}, and therefore the decomposition of $E_n$ becomes clear since $E_{n,1}$ can be reformulated as
$$E_{n,1}=\int_{\mathbb{R}^d}e^{-\frac{1+\alpha-k_n}{2}|x|^2}\,dx-(2\pi\bsigma_\alpha^2)^{d/2}.$$

On the one hand, given any fixed value $\theta\in \mathbb{R}^+$ with $\theta<\frac{\alpha}{2}$ (for instance $\theta=\frac{\alpha}{4}$) the propagation of exponential moments in Corollary \ref{C-exponential-moments} implies that
$$E_\theta:=\sup_{n\in \mathbb{N}}\int_{\mathbb{R}^d}e^{\theta |x|^2}\frac{F_n(x)}{\Vert F_n\Vert_{L^1(\mathbb{R}^d)}}\,dx<\infty.$$
By expanding the exponential in power series, one obtains in particular uniformly bounded moments of any order, and in particular, of order $2m$, namely
$$M_m:=\sup_{n\in \mathbb{N}}\int_{\mathbb{R}^d}\vert x\vert^{2m}\frac{F_n(x)}{\Vert F_n\Vert_{L^1(\mathbb{R}^d)}}\,dx\leq \frac{E_\theta \,m!}{\theta^m}<\infty.$$
Hence, we infer the following control on the first error term in \eqref{E-convergence-ratios-max-mass-errors}
\begin{equation}\label{E-convergence-ratios-max-mass-errors-1}
\tilde{E}_n\leq \frac{1}{R_n^{2m}}\int_{\mathbb{R}^d}\vert x\vert^{2m}\frac{F_n(x)}{\Vert F_n\Vert_{L^1(\mathbb{R}^d)}}\,dx\lesssim \frac{1}{R_n^{2m}},
\end{equation}
for any $n\in \mathbb{N}$. On the other hand, by Lemma \ref{L-sexual-reproduction-coefficients-asymptotics} we have that $E_{n,1}$ in \eqref{E-convergence-ratios-max-mass-errors-decomposition} can be bounded by
\begin{equation}\label{E-convergence-ratios-max-mass-errors-2-1}
\vert E_{n,1}\vert\lesssim k_n-\bk_\alpha\lesssim \br_\alpha^n.
\end{equation}
By direct calculation we also obtain
\begin{equation}\label{E-convergence-ratios-max-mass-errors-2-2}
E_{n,2}=\vert \mathbb{S}^{d-1}\vert\int_{R_n}^{+\infty}r^{d-1}e^{-\frac{1+\alpha-k_n}{2}r^2}\,dr\lesssim\Gamma\left(\frac{d}{2},\frac{1+\alpha-k_n}{2}R_n^2\right)\lesssim \exp\left(-\frac{1+\alpha-k_n}{4(2\bk_\alpha)^\frac{2n}{2m+1}}\right),
\end{equation}
where $\Gamma(a,x)$ is the incomplete Gamma function \eqref{E-incomplete-gamma-function} and we have used estimate \eqref{E-estimate-incomplete-gamma-function} in Lemma \ref{L-estimate-incomplete-gamma-function}. Note that the constraint \eqref{E-estimate-incomplete-gamma-function-hypothesis} is trivially satisfied since $R_n\rightarrow \infty$ by our choice \eqref{E-convergence-ratios-max-mass-radius}. Finally, we control the error term $E_{n,3}$ in \eqref{E-convergence-ratios-max-mass-errors-decomposition} under the addition assumption that $\bar F_0$ satisfies the hypothesis \eqref{E-sexual-reproduction-F-hypothesis} in Lemma \ref{L-sexual-reproduction-weak-dependency}. Whilst this is not always true, we show at the end of the proof that we can always assume so without loss of generality by replacing the argument on $\bar F_0$ by an advance enough time step $\bar F_m$ so that selection has properly shaped the Gaussian tails. Under this condition, Lemma \ref{L-sexual-reproduction-weak-dependency} implies that given any $\varepsilon'\in \mathbb{R}_+^*$ there is $C_{\varepsilon'}\in \mathbb{R}_+^*$ so that, for $|x|\leq R_n$, we have
\begin{align*}
&\left\vert\int_0^1(\nabla\log\mathcal{E}_n)[\bar F_0](\theta x)\cdot x\,d\theta\right\vert\\
&\quad \lesssim (2\bk_\alpha+\varepsilon')^n R_n+(\br_\alpha+\varepsilon')^n R_n^2+e^{-\frac{(2+\varepsilon')^n}{\varepsilon'}}e^{C_{\varepsilon'}(2\br_\alpha+\varepsilon')^nR_n^2}R_n\\
&\quad \lesssim \left(\frac{2\bk_\alpha}{(2\bk_\alpha)^{\frac{1}{2m+1}}}+\frac{\varepsilon}{2}\right)^n+\left(\frac{\br_\alpha}{(2\bk_\alpha)^{\frac{2}{2m+1}}}+\frac{\varepsilon}{2}\right)^n+e^{-\frac{(2+\varepsilon')^n}{\varepsilon'}}\exp\left[C_{\varepsilon'}\left(\frac{2\br_\alpha}{(2\bk_\alpha)^{\frac{2}{2m+1}}}+\frac{\varepsilon}{2}\right)^n\right]\frac{1}{(2\bk_\alpha)^{\frac{n}{2m+1}}}\\
&\quad \lesssim \left((2\bk_\alpha)^{\frac{2m}{2m+1}}+\frac{\varepsilon}{2}\right)^n\lesssim \left(2\bk_\alpha+\varepsilon\right)^n.
\end{align*}
In the second inequality above we have taken $\varepsilon'$ small enough compared with $\varepsilon$. In the third inequality we have used the relation $\br_\alpha=2\bk_\alpha^2$, which guarantees that all the contributions in the third line are dominated by the first term $((2\bk_\alpha)^{\frac{2m}{2m+1}}+\frac{\varepsilon}{2})^n$. Finally, in the last inequality we have used our choice of $m\in \mathbb{N}$ large enough so that \eqref{E-convergence-ratios-max-mass-radius-choice-m} holds. Therefore, the mean value theorem implies
\begin{align*}
\left\vert\frac{\mathcal{E}_n[\bar F_0](x)}{\mathcal{E}_n[\bar F_0](0)}-1\right\vert&=\left\vert\exp\left(\int_0^1(\nabla\log\mathcal{E}_n[\bar F_0])(\theta x)\cdot x\,d\theta\right)-1\right\vert\lesssim (2\bk_\alpha+\varepsilon)^n.
\end{align*}
Hence, $E_{n,3}$ in \eqref{E-convergence-ratios-max-mass-errors-decomposition} can be controlled by
\begin{equation}\label{E-convergence-ratios-max-mass-errors-2-3}
\vert E_{n,3}\vert\leq \int_{\vert x\vert\leq R_n}e^{-\frac{1+\alpha-k_n}{2}\vert x\vert^2}\left\vert\frac{\mathcal{E}_n[\bar F_0](x)}{\mathcal{E}_n[\bar F_0](0)}-1\right\vert\,dx\lesssim (2\bk_\alpha+\varepsilon)^n.
\end{equation}
Putting \eqref{E-convergence-ratios-max-mass-errors-2-1}, \eqref{E-convergence-ratios-max-mass-errors-2-2} and \eqref{E-convergence-ratios-max-mass-errors-2-3} together yields
\begin{equation}\label{E-convergence-ratios-max-mass-errors-2}
\vert E_n\vert \lesssim (2\bk_\alpha+\varepsilon)^n.
\end{equation}
Thereby, using the mean value theorem on \eqref{E-convergence-ratios-max-mass-decomposition} along with the bounds \eqref{E-convergence-ratios-max-mass-errors-1} and \eqref{E-convergence-ratios-max-mass-errors-2} concludes that
$$\left\vert c_n-\frac{1}{(2\pi\bsigma_\alpha^2)^{d/2}}\right\vert \lesssim \tilde{E}_n+\vert E_n\vert\lesssim (2\bk_\alpha+\varepsilon)^n.$$

To end the proof, we show that we can always assume that the hypothesis \eqref{E-sexual-reproduction-F-hypothesis} in Lemma \ref{L-sexual-reproduction-weak-dependency} is satisfied without loss of generality. Indeed, note that after two iterations $\bar F_2$ already satisfies $\eqref{E-sexual-reproduction-F-hypothesis}_2$ thanks to Corollary \ref{C-sexual-reproduction-Gaussian-tails}, but only two iterations are not enough to guarantee $\eqref{E-sexual-reproduction-F-hypothesis}_1$, in general. Nevertheless, applying Corollary \ref{C-sexual-reproduction-Gaussian-tails} leads to the following upper bound on the normalized profiles:
$$\bar F_{m}(x)\leq C_{m}\,e^{\frac{1}{2}\left(\alpha+\frac{1}{2}-\frac{1}{\overline{\sigma}_{m}^2}\right)\vert x\vert^2},$$
for appropriate $C_m\in \mathbb{R}_+^*$, with variances $\{\overline{\sigma}_m^2\}_{m\in \mathbb{N}}$ defined by the recurrence \eqref{E-variances-barriers}, {\it i.e.},
$$\frac{1}{\overline{\sigma}_{m+1}^2}=\alpha+\frac{1}{1+\frac{\overline{\sigma}_m^2}{2}},\quad m\in \mathbb{N},$$
and with initial datum $\overline{\sigma}_1^2=1/\alpha$. The boundedness condition $\eqref{E-sexual-reproduction-F-hypothesis}_2$ is then satisfied by $\bar F_m$ if we can show that the prefactor $\alpha+\frac{1}{2}-\frac{1}{\overline{\sigma}_m^2}$ in the above exponential bound  becomes non-positive for large enough $m$. This is where the precise choice of normalization in Definition \ref{D-rescaled-F} plays a role. Specifically, recall that so defined $\overline{\sigma}_m^2\rightarrow \bsigma_\alpha^2$ and we have precise relaxation rates by Lemma \ref{L-sexual-reproduction-Gaussian-solution-time-discrete-variance-relaxation}. Hence,
$$\alpha+\frac{1}{2}-\frac{1}{\overline{\sigma}_m^2}=\left(\alpha+\frac{1}{2}-\frac{1}{\bsigma_\alpha^2}\right)+\left(\frac{1}{\bsigma_\alpha^2}-\frac{1}{\overline{\sigma}_m^2}\right)\leq \left(\alpha+\frac{1}{2}-\frac{1}{\bsigma_\alpha^2}\right)+C_v\br_\alpha^m,$$
for any $m\in \mathbb{N}$. Since we chose a normalization $\bar F_n$ of the profiles by a Gaussian $G_{0,\sigma^2}$ with variance $\sigma^2$ strictly larger that the variance $\bsigma_\alpha^2$ of the equilibrium $\bF_\alpha$ (more specifically $\frac{1}{\sigma^2}=\frac{1}{\sigma^2}=\alpha+\frac{1}{2}$), we can guarantee that the right hand side above is non-positive if $m\geq n_\alpha$ for a sufficiently large $n_\alpha\geq 2$ (depending only on $\alpha$). This justifies that $\bar F_{n_\alpha}$ satisfies both conditions in \eqref{E-sexual-reproduction-F-hypothesis}.

\end{proof}

We are now in position to prove the above local convergence result in Corollary \ref{C-sexual-reproduction-local-convergence}

\begin{proof}[Proof of Corollary \ref{C-sexual-reproduction-local-convergence}]
By formula \eqref{E-sexual-Gaussian-quadratic-emF} in Proposition \ref{P-iterative-method-sexual-reproduction} and Definition \ref{D-expectation-operator} we obtain
$$\nabla\log F_n(x)=-(1+\alpha-k_n)x+\nabla\log \mathcal{E}_n[\bar F_0](x),$$
for any $x\in \mathbb{R}^d$ and every $n\in \mathbb{N}$. Using the mean value theorem we then achieve
\begin{align}\label{E-sexual-reproduction-local-convergence-pre}
\begin{aligned}
\log \frac{F_n(x)}{F_n(0)}-\log\frac{\bF_\alpha(x)}{\bF_\alpha(0)}&=\int_0^1(\nabla\log F_n)(\theta x)\cdot x\,d\theta+\frac{1}{2\bsigma_\alpha^2}\vert x\vert^2\\
&=-\frac{1+\alpha-k_n}{2}\vert x\vert+\int_0^1\nabla\log\mathcal{E}_n[\bar F_0](\theta x)\cdot x\,d\theta+\frac{1}{2\bsigma_\alpha^2}\vert x\vert^2\\
&=\frac{k_n-\bk_\alpha}{2}\vert x\vert^2+\int_0^1\nabla\log\mathcal{E}_n[\bar F_0](\theta x)\cdot x\,d\theta,
\end{aligned}
\end{align}
where we have used the relation $1/\bsigma_\alpha^2=1+\alpha-\bk_\alpha$. On the one hand, the first term in the right hand side of \eqref{E-sexual-reproduction-local-convergence-pre} converges to zero thanks to Lemma \ref{L-sexual-reproduction-coefficients-asymptotics}. On the other hand, for the second term we shall apply Lemma \ref{L-sexual-reproduction-weak-dependency}. To do so we need to guarantee again that $\bar F_0$ verifies the hypothesis \eqref{E-sexual-reproduction-F-hypothesis} of such a lemma. Recall that we can always assume the condition \eqref{E-sexual-reproduction-F-hypothesis} without loss of generality ({\it cf.} last step in the proof of Lemma \ref{L-convergence-ratios-max-mass}). Therefore, we can apply Remark \ref{R-sexual-reproduction-local-convergence-expectations} with $F=\bar F_0$ and obtain that $\nabla\log\mathcal{E}_n[\bar F_0]$ converges to zero uniformly over compact sets with explicit convergence rates. Putting it into \eqref{E-sexual-reproduction-local-convergence-pre} we obtain more explicitly
$$\sup_{\vert x\vert \leq R} \left\vert \log \frac{F_n(x)}{F_n(0)}-\log\frac{\bF_\alpha(x)}{\bF_\alpha(0)}\right\vert \leq C_{\varepsilon,R}\,(2\bk_\alpha)^n,$$
for any $R\in \mathbb{R}_+^*$, any $n\in \mathbb{N}$ and a large enough $C_{\varepsilon,R}\in \mathbb{R}_+^*$. This, together with the above control on the asymptotics of $\Vert F_n\Vert_{L^1(\mathbb{R}^d)}/F_n(0)$ in Lemma \ref{L-convergence-ratios-max-mass}, allow proving \eqref{E-sexual-reproduction-local-convergence}.
\end{proof}

In particular, note that the above Corollary \ref{C-sexual-reproduction-local-convergence} is enough to prove the uniqueness of solutions to the eigenproblem \eqref{E-eigenproblem-time-discrete}, as stated in Corollary \ref{C-main}.

\begin{proof}[Proof of Corollary \ref{C-main}]
Let $(\lambda,F)$ be any solution of the eigenproblem \eqref{E-eigenproblem-time-discrete}. Then, the ansatz \eqref{E-ansatz-discrete} defines a solution $F_n=\lambda^n F$ of the time-discrete problem \eqref{E-sexual-reproduction-time-discrete}. By Corollary \ref{C-sexual-reproduction-local-convergence} we obtain that
$$\lim_{n\rightarrow \infty}\sup_{\vert x\vert\leq R}\left\vert \frac{\lambda^n F(x)}{\lambda^n}-\bF_\alpha(x)\right\vert=0,$$
for any $R\in \mathbb{R}_+^*$. Hence, $F\equiv\bF_\alpha$, and thus $\lambda=\blambda_\alpha$.
\end{proof}

Whilst the above local convergence result is enough to identify asymptotically the profile $\bF_\alpha$, a global result with quantitative convergence rates is still missing. In particular, note that the constants $C_{\varepsilon,R}$ above blow up when $R\rightarrow\infty$ as we see explicitly in Lemma \ref{L-sexual-reproduction-weak-dependency}. A second drawback of this we are unable to characterize the long-time behavior of the mass $\Vert F_n\Vert_{L^1(\mathbb{R}^n)}$ in terms of the eigenvalue $\blambda_\alpha$ via this method. In the following section, we give an answer to both questions by better exploiting the previous fundamental Lemma \ref{L-sexual-reproduction-weak-dependency} and using the propagation of quadratic and exponential moments in Section \ref{SS-propagation-moments}.

\subsection{Global convergence result}\label{SS-proof-main-theorem}
We are now ready to prove our main result. Let us emphasize that our final convergence result in Theorem \ref{T-main} is presented using the Kullback-Leibler divergence, which is a very different metric from the log-Lipschitz type metrics used in the previous Section \ref{SS-local-convergence-result} for the local convergence results. As anticipated in Remark \ref{R-choice-metric}, this decision does not obey aesthetic reasons only, but we actually need to move from uniform norms (like the log-Lipschitz norm) to averaged norms (like the Kullback-Leibler divergence) in order to address the deficiency encountered in Lemma \ref{L-sexual-reproduction-weak-dependency}. Specifically, recall that for generic initial data $F_0\in \mathcal{M}_+(\mathbb{R}^d)$, the log-Lipschitz norms of the high-dimensional integral $\mathcal{E}_n[\bar F_0]$ cannot be controlled uniformly due to additional exponentially growing terms.

\begin{proof}[Proof of Theorem \ref{T-main}]

~\medskip

$\bullet$ {\sc Step 1}: Convergence of the profiles $F_n/\Vert F_n\Vert_{L^1(\mathbb{R}^d)}$.\\
Since $\bF_\alpha$ is Gaussian (thus strongly log-concave), then the logarithmic-Sobolev inequality holds true and therefore we obtain the following control of the relative entropy by the relative Fisher information
$$\mathcal{D}_{\rm KL}\left(\left.\frac{F_n}{\Vert F_n\Vert}\right\Vert \bF_\alpha\right)\leq \frac{\bsigma_\alpha^2}{2}\int_{\mathbb{R}^d}\left\vert\nabla\log\left(\frac{F_n(x)}{\bF_\alpha(x)}\right)\right\vert^2\,\frac{F_n(x)}{\Vert F_n\Vert_{L^1(\mathbb{R}^d)}}\,dx,$$
see Corollary 5.7.2 and Section 9.3.1 in \cite{BGL-14} for details. By the reformulation of $F_n$ in formula \eqref{E-sexual-Gaussian-quadratic-emF} of Proposition \ref{P-iterative-method-sexual-reproduction} and using the notation in Definition \ref{D-expectation-operator} we have that
$$\nabla\log \left(\frac{F_n(x)}{\bF_\alpha(x)}\right)=(k_n-\bk_\alpha)\,x+\nabla\log \mathcal{E}_n[\bar F_0](x).$$
Therefore, we obtain the following upper bound
\begin{align}\label{E-relative-entropy-bound-decomposition}
\mathcal{D}_{\rm KL}\left(\left.\frac{F_n}{\Vert F_n\Vert}\right\Vert \bF_\alpha\right)\lesssim D_{n,1}+D_{n,2},
\end{align}
where each factor takes the form
\begin{align*}
D_{n,1}&:=|k_n-\bk_\alpha|^2\int_{\mathbb{R}^d}|x|^2\frac{F_n(x)}{\Vert F_n\Vert_{L^1(\mathbb{R}^d)}}\,dx,\\
D_{n,2}&:=\int_{\mathbb{R}^d}\left\vert\nabla \log \mathcal{E}_n[\bar F_0](x)\right\vert^2\frac{F_n(x)}{\Vert F_n\Vert_{L^1(\mathbb{R}^d)}}\,dx.
\end{align*}

On the one hand, for $D_{n,1}$ we can use the propagation of quadratic moments in Corollary \ref{C-quadratic-moments} together with the explicit relaxation rates of $\{k_n\}_{n\in \mathbb{N}}$ to $\bk_\alpha$ in Lemma \ref{L-sexual-reproduction-coefficients-asymptotics} to show that 
\begin{equation}\label{E-Dn1-control}
D_{n,1}\lesssim \br_\alpha^{2n}.
\end{equation}
On the other hand, for $D_{n,2}$ we shall refine the argument in the proof of Corollary \ref{C-sexual-reproduction-local-convergence}. Specifically, recall again that we can assume that $\bar F_0$ satisfies the hypothesis \eqref{E-sexual-reproduction-F-hypothesis} without loss of generality ({\em cf.} proof of Corollary \ref{C-sexual-reproduction-local-convergence}). Then, applying Lemma \ref{L-sexual-reproduction-weak-dependency} in order to control $\nabla\log\mathcal{E}_n[\bar F_0]$ with fixed $\varepsilon\in \mathbb{R}_+^*$ yields
$$D_{n,2}\lesssim \left((2\bk_\alpha)^2+\varepsilon\right)^n\int_{\mathbb{R}^d}\left(1+|x|^2+e^{2\,C_\varepsilon(2\br_\alpha+\varepsilon)^n|x|^2}\right)\frac{F_n(x)}{\Vert F_n\Vert_{L^1(\mathbb{R}^d)}}\,dx,$$
for some sufficiently large $C_\varepsilon\in \mathbb{R}_+^*$. Above we have used that $(2\bk_\alpha)^n$ is the decay rate that controls all the others in formula \eqref{E-sexual-reproduction-weak-dependency}. Take $n\geq n_\varepsilon$ sufficiently large such that
$$\theta_\varepsilon:=\sup_{n\geq n_\varepsilon}2\,C_\varepsilon(2\br_\alpha+\varepsilon)^n<\frac{\alpha}{2}.$$
Then, Corollaries \ref{C-quadratic-moments} and \ref{C-exponential-moments} imply that the above quadratic and exponential moments are uniformly bounded with respect to $n$ for $n\geq n_\varepsilon$. Therefore, we have 
\begin{equation}\label{E-Dn2-control}
D_{n,2}\lesssim ((2\bk_\alpha)^2+\varepsilon)^n.
\end{equation}
Finally, putting the estimates \eqref{E-Dn1-control} and \eqref{E-Dn2-control} into the split \eqref{E-relative-entropy-bound-decomposition}, and noticing that $\br_\alpha^2<(2\bk_\alpha)^2$ thanks to the relation $\br_\alpha=2\bk_\alpha^2$, ends this part of the proof.

\medskip

$\bullet$ {\sc Step 2}: Convergence of the growth rates $\Vert F_n\Vert_{L^1(\mathbb{R}^d)}/\Vert F_{n-1}\Vert_{L^1(\mathbb{R}^d)}$.\\
For simplicity of notation we shall define the following sequence of coefficients: 
$$\lambda_n:=\frac{\Vert F_n\Vert_{L^1(\mathbb{R}^d)}}{\Vert F_{n-1}\Vert_{L^1(\mathbb{R}^d)}},\quad n\in \mathbb{N}.$$
Our goal then reduces to studying the asymptotic behavior of $\{\lambda_n\}_{n\in \mathbb{N}}$ and obtaining quantitative convergence rates. By definition of operator $\mathcal{T}$ in \eqref{E-sexual-reproduction-T-operator} we obtain that
$$\lambda_n=\int_{\mathbb{R}^d}e^{-m(x)}\int_{\mathbb{R}^{2d}}G\left(x-\frac{x_1+x_2}{2}\right)\frac{F_{n-1}(x_1)}{\Vert F_{n-1}\Vert_{L^1(\mathbb{R}^d)}}\frac{F_{n-1}(x_2)}{\Vert F_{n-1}\Vert_{L^1(\mathbb{R}^d)}}\,dx_1\,dx_2\,dx.$$
By direct computation of the integral with respect to $x$ as it was done in \eqref{E-computation-explicit-integral-1} we find that
\begin{equation}\label{E-rate-growth-restatement}
\lambda_n=\int_{\mathbb{R}^{2d}}H(x_1,x_2)\frac{F_{n-1}(x_1)}{\Vert F_{n-1}\Vert_{L^1(\mathbb{R}^d)}}\frac{F_{n-1}(x_2)}{\Vert F_{n-1}\Vert_{L^1(\mathbb{R}^d)}}\,dx_1\,dx_2,
\end{equation}
where the function $H(x_1,x_2)$ takes the form
\begin{equation}\label{E-rate-growth-function-H}
H(x_1,x_2):=\frac{1}{(1+\alpha)^{d/2}}\exp\left(-\frac{\alpha}{(1+\alpha)}\left\vert \frac{x_1+x_2}{2}\right\vert^2\right),\quad (x_1,x_2)\in \mathbb{R}^d.
\end{equation}
In addition, recalling the explicit form of the eigen-pair $(\blambda_\alpha,\bF_\alpha)$ in \eqref{E-sexual-reproduction-Gaussian-solution} implies 
\begin{equation}\label{E-rate-growth-restatement-lambda}
\blambda_\alpha=\int_{\mathbb{R}^{2d}}H(x_1,x_2)\bF_\alpha(x_1)\bF_\alpha(x_2)\,dx_1\,dx_2.
\end{equation}
Therefore, taking the difference of \eqref{E-rate-growth-restatement} and \eqref{E-rate-growth-restatement-lambda} and noticing that $H$ in \eqref{E-rate-growth-function-H} is bounded yields
\begin{equation}\label{E-convergence-ratio-mass-Pinsker-1}
\vert \lambda_n-\blambda_\alpha\vert\lesssim \left\Vert \frac{F_n}{\Vert F_n\Vert_{L^1(\mathbb{R}^d)}}\otimes \frac{F_n}{\Vert F_n\Vert_{L^1(\mathbb{R}^d)}}-\bF_\alpha\otimes \bF_\alpha\right\Vert_{L^1(\mathbb{R}^{2d})},
\end{equation}
for any $n\in \mathbb{N}$, where $P\otimes Q$ denotes the Kronecker product of two measures $P,Q\in \mathcal{M}(\mathbb{R}^d)$, namely
$$\int_{\mathbb{R}^{2d}}\varphi(x,y)\,(P\otimes Q)(dx,dy)=\int_{\mathbb{R}^d}\left(\int_{\mathbb{R}^d}\varphi(x,y)\,P(dx)\right)\,Q(dy),$$
for all $\varphi\in C_c(\mathbb{R}^d)$  ({\it cf.} Table \ref{tab:list-notations}). We do not have a direct convergence result of $F_n/\Vert F_n\Vert_{L^1(\mathbb{R}^d)}$ in $L^1$ norms, but we do have convergence of the Kullback-Leibler divergence thanks to {\sc Step 1}. By Pinsker's inequality, the latter metric controls the former, namely
\begin{align}\label{E-convergence-ratio-mass-Pinsker-2}
\begin{aligned}
&\left\Vert \frac{F_n}{\Vert F_n\Vert_{L^1(\mathbb{R}^d)}}\otimes \frac{F_n}{\Vert F_n\Vert_{L^1(\mathbb{R}^d)}}-\bF_\alpha\otimes \bF_\alpha\right\Vert_{L^1(\mathbb{R}^{2d})}\\
&\qquad \leq  \frac{1}{\sqrt{2}}\sqrt{\mathcal{D}_{\rm KL}\left(\left.\frac{F_n}{\Vert F_n\Vert_{L^1(\mathbb{R}^d)}}\otimes \frac{F_n}{\Vert F_n\Vert_{L^1(\mathbb{R}^d)}}\right\Vert \bF_\alpha\otimes \bF_\alpha\right)}\\
&\qquad =\frac{1}{\sqrt{2}}\sqrt{2\,\mathcal{D}_{\rm KL}\left(\left.\frac{F_n}{\Vert F_n\Vert_{L^1(\mathbb{R}^d)}}\otimes \right\Vert \bF_\alpha\right)},
\end{aligned}
\end{align}
where in last step we have used the tensorization property of the Kullback-Leilber divergence. The result then follows from \eqref{E-convergence-ratio-mass-Pinsker-1}-\eqref{E-convergence-ratio-mass-Pinsker-2} and the above explicit convergence rates of the normalized profiles $F_n/\Vert F_n\Vert_{L^1(\mathbb{R}^d)}$ in {\sc Step 1}.
\end{proof}

\section{Numerical experiments}\label{S-numerical-experiments}

For our numerical simulation, we have restricted to one-dimensional traits ({\em i.e.}, $d=1$) and we have considered a step function of the following form as initial datum:
\begin{equation}\label{E-simulation-initial-datum}
F_0=\frac{1}{Z}(30\,\mathds{1}_{[-7,-3]}+20\,\mathds{1}_{[7.5,12.5]}+50\,\mathds{1}_{[30,40]}+30\,\mathds{1}_{[52,5,57,5]}),
\end{equation}
where $Z\in\mathbb{R}_+^*$ is a normalizing factor so that $F_0\in \mathcal{P}(\mathbb{R})$. Let $\{F_n\}_{n\in \mathbb{N}}$ be the solution to the time-discrete problem \eqref{E-sexual-reproduction-time-discrete} starting at $F_0$. Our numerical simulation is performed with {\sc Python} on the finite computation domain $[-15,60]$ for the trait variable $x$, which contains the support of the previous initial datum $F_0$ and all the mass of each $F_n$ (except for a negligible Gaussian tail). In our simulation, we set $\Delta x=0.001$ as the step for the discretization of our computational domain. In particular, with such a step we compute numerical integrals with respect to $x$ according to the left rectangle rule.

In the following, we explore numerically two different scenarios: {\em weak selection} and {\em strong selection}. Specifically, we obtain numerical approximations for the profiles $F_n$ in both regimes and we illustrate numerically the convergence of the growth rates $\Vert F_n\Vert_{L^1(\mathbb{R}^d)}/\Vert F_{n-1}\Vert_{L^1(\mathbb{R}^d)}$ towards the eigenvalue $\blambda_\alpha$, and the relaxation of the normalized profiles $F_n/\Vert F_n\Vert_{L^1(\mathbb{R}^d)}$ towards the eigenfunction $\bF_\alpha$. As a consequence, we derive numerical approximations of the convergence rates to be compared with the theoretical results in this paper. More specifically, we note that the theoretical convergence rates in Theorem \ref{T-main} are sharp, except a mismatch for the rates of growth of mass, which was discussed previously in Remark \ref{R-rough-convergence-rates} (see also Figure \ref{fig:non-optimal-rates}). Indeed, we actually attain numerically the same convergence rates as in Corollary \ref{C-sexual-reproduction-Gaussian-solution-time-discrete-relaxation}, which were sharp for Gaussian initial data.

\subsection{Weak selection}

In this part, we consider a small value $\alpha=0.015$. This leads to the following numerical values of the features of the equilibrium:
$$\blambda_\alpha\approx 0.9857,\qquad \bsigma_\alpha^2\approx 1.8897.$$
characterizing the eigenpair $(\blambda_\alpha,\bF_\alpha)$, according to \eqref{E-sexual-reproduction-Gaussian-solution} and \eqref{E-sexual-reproduction-Gaussian-solution-variance-equation}. Note that since the selection parameter has been taken very small, then the eigenvalue and the variance of the eigenfunction are close to those at linkage equilibrium, namely, $\blambda_{\alpha=0}=1$ and $\bsigma_{\alpha=0}^2=2$ (see Remark \ref{R-sexual-reproduction-Gaussian-solution-variance} and Figure \ref{fig:behavior-eigenvalue-variance-Gaussian-solution}).

\begin{figure}
\centering
\begin{subfigure}[t]{0.49\textwidth}
\centering
\includegraphics[width=\textwidth]{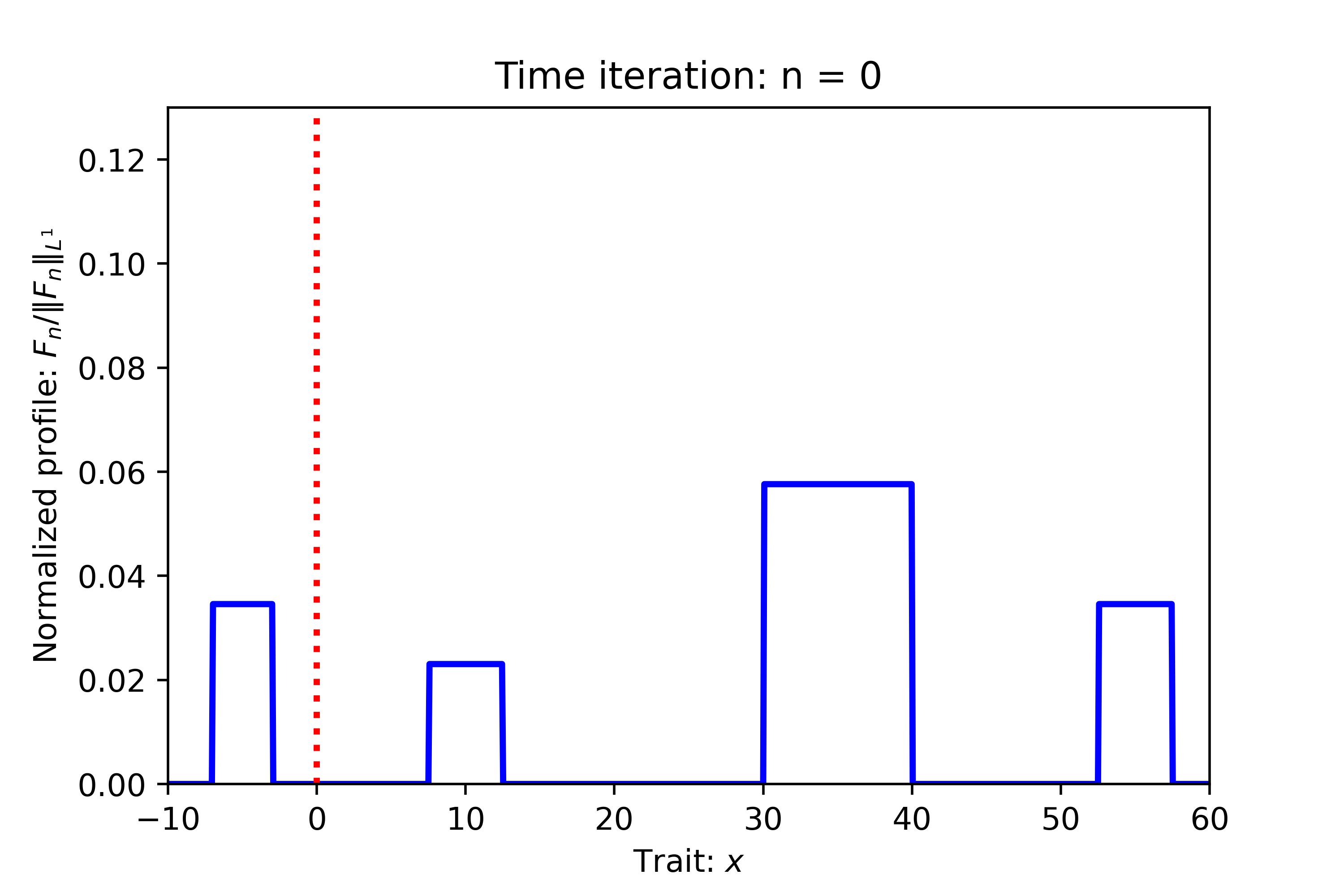}
\end{subfigure}
\begin{subfigure}[t]{0.49\textwidth}
\centering
\includegraphics[width=\textwidth]{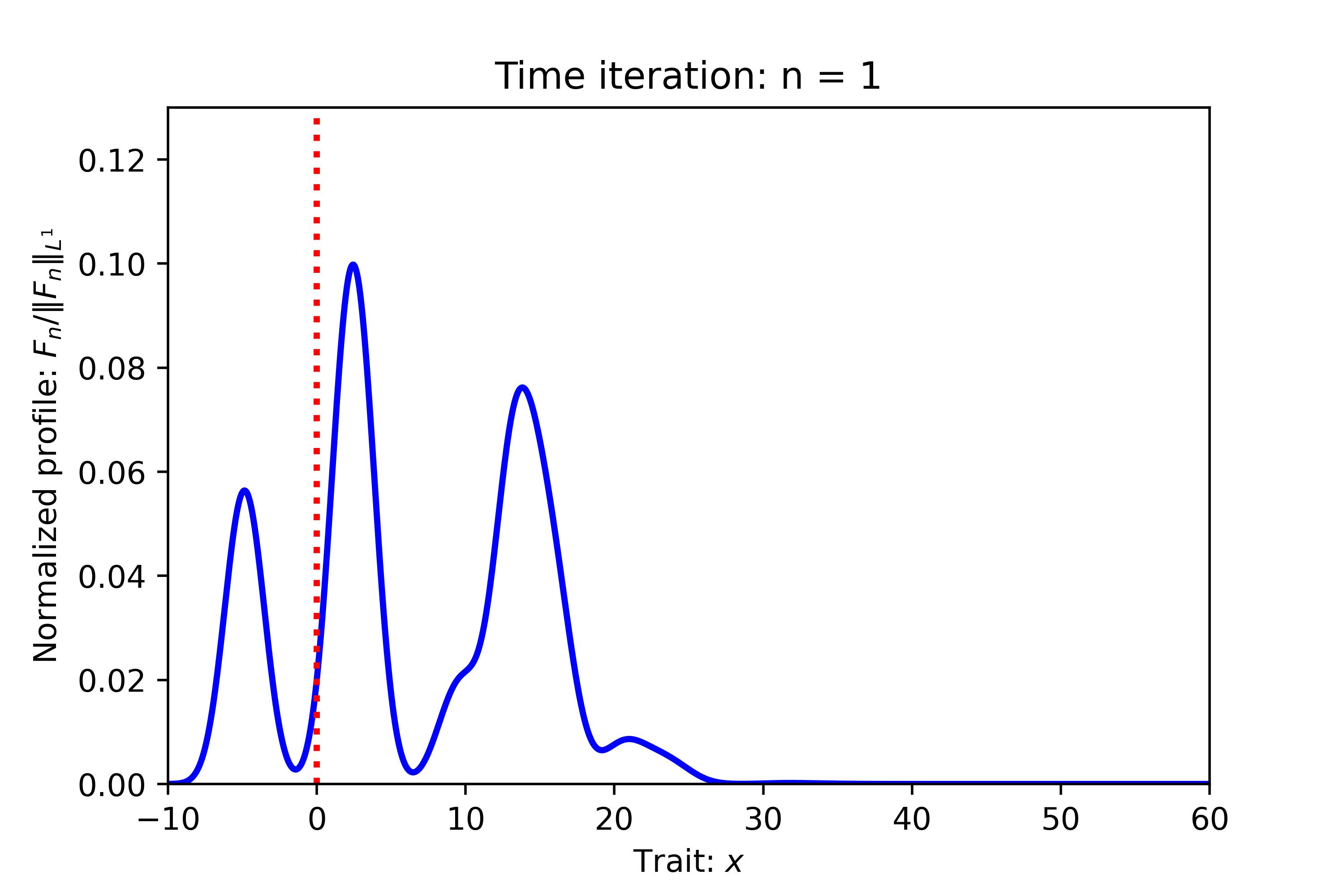}
\end{subfigure}
\begin{subfigure}[t]{0.49\textwidth}
\centering
\includegraphics[width=\textwidth]{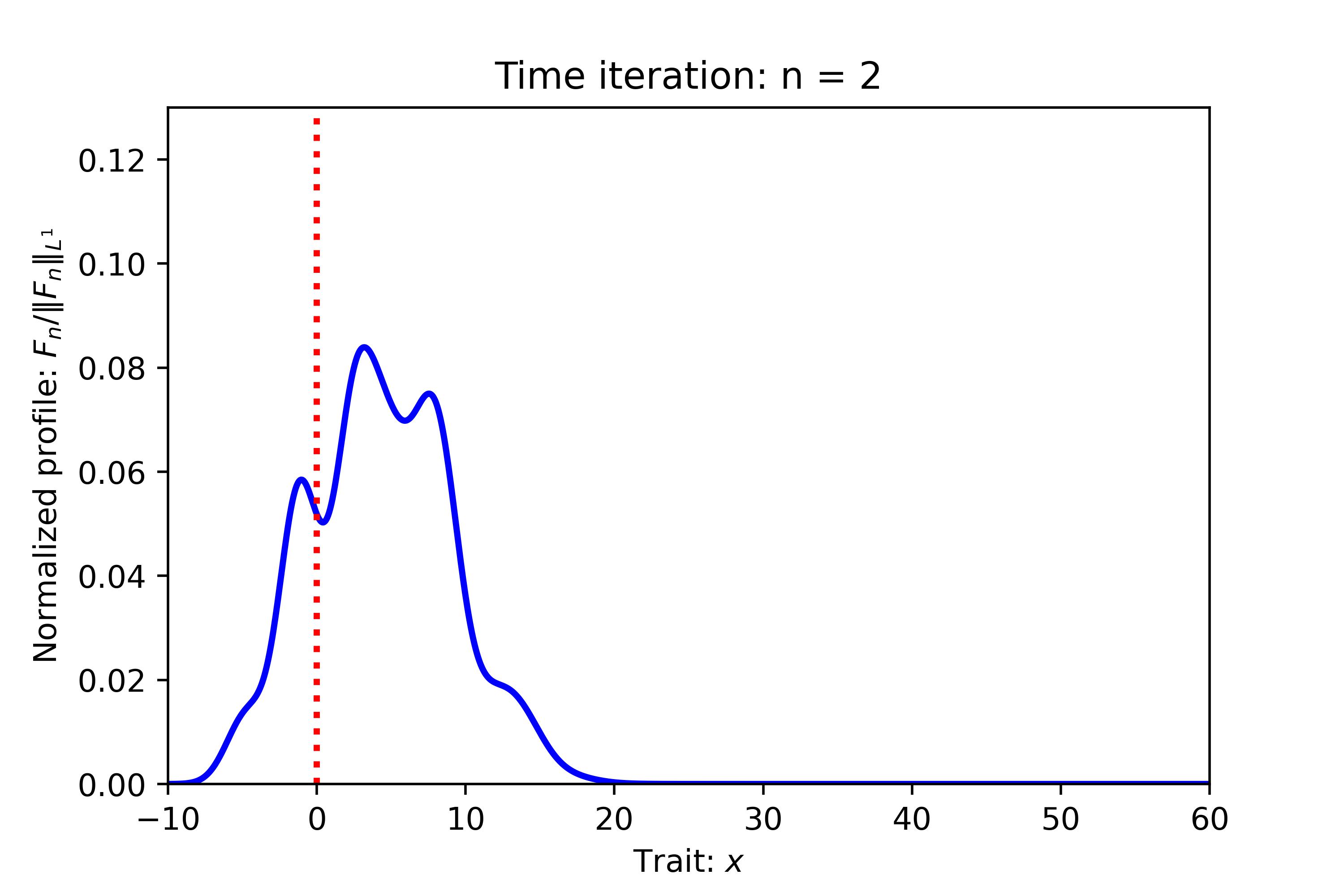}
\end{subfigure}
\begin{subfigure}[t]{0.49\textwidth}
\centering
\includegraphics[width=\textwidth]{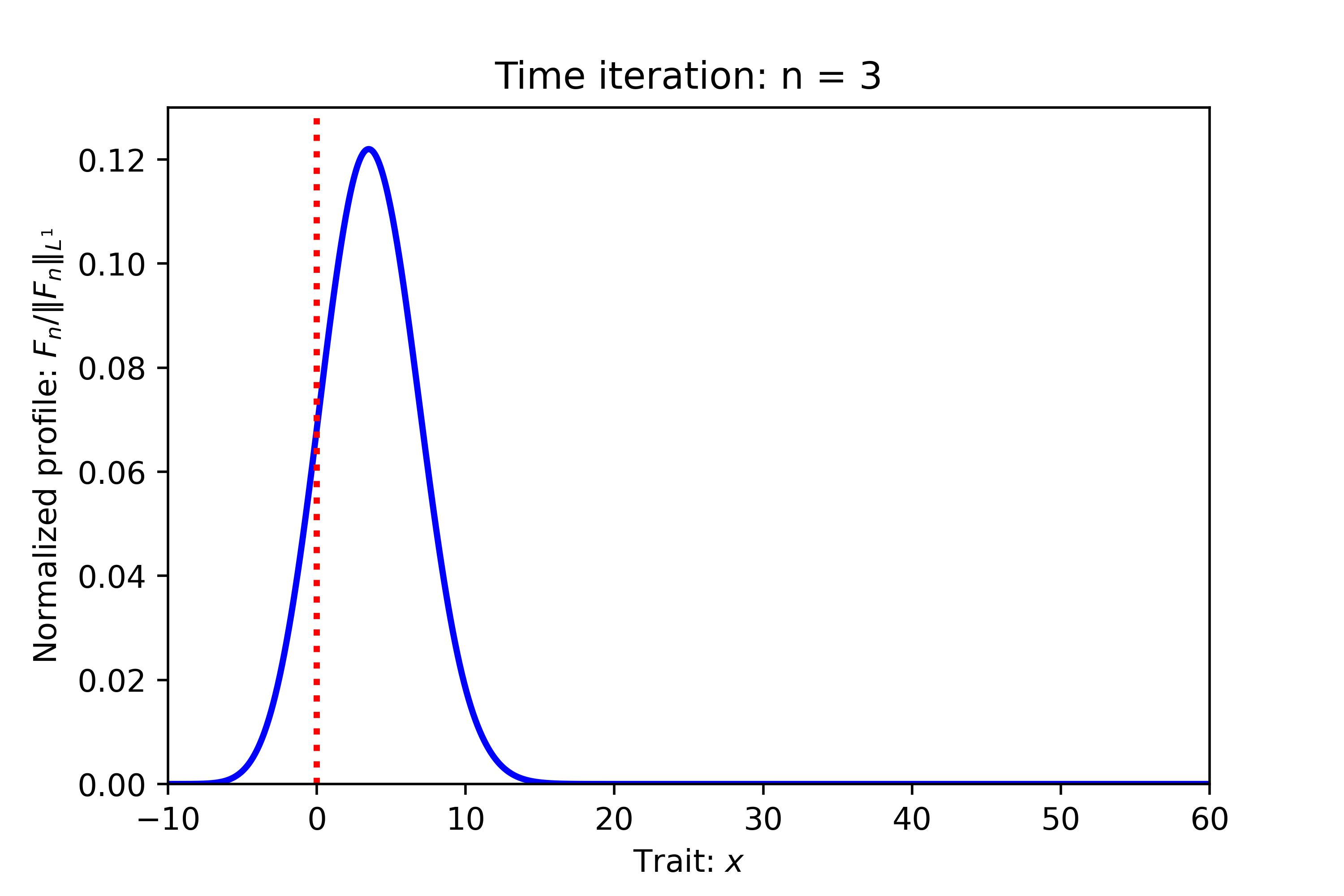}
\end{subfigure}
\begin{subfigure}[t]{0.55\textwidth}
\centering
\includegraphics[width=\textwidth]{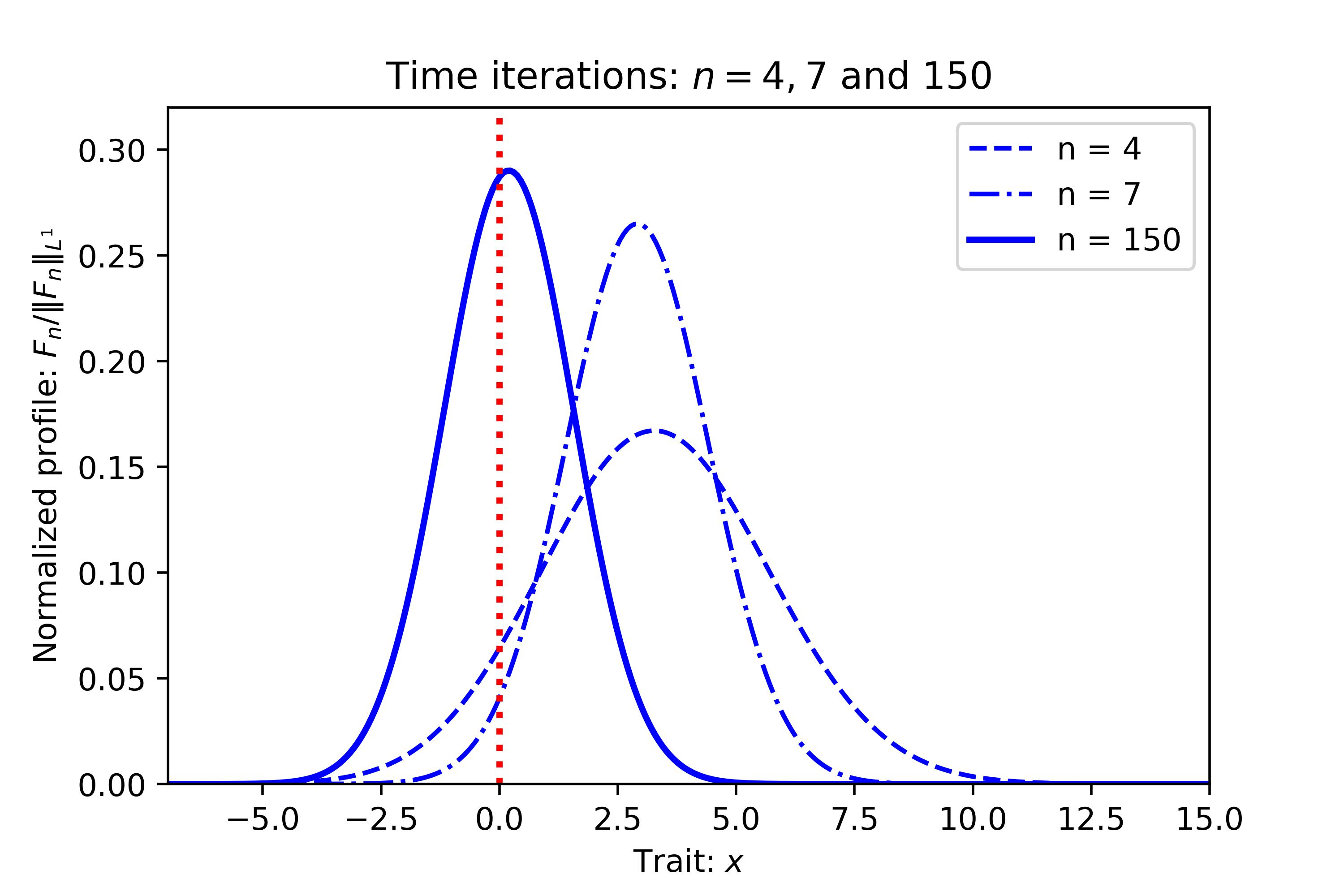}
\end{subfigure}
\caption{Relaxation of the normalized profiles $F_n/\Vert F_n\Vert_{L^1(\mathbb{R}^d)}$ towards the eigenfunction $\bF_\alpha$ along the time iterations $n=0,1,2,3,4,7$ and $150$ for a step function \eqref{E-simulation-initial-datum} as initial datum and weak selection parameter $\alpha=0.015$. The vertical dotted line represents the location of the mean of the equilibrium profile $\bF_\alpha$.}
\label{fig:simulation-1-profiles}
\end{figure}

In Figure \ref{fig:simulation-1-profiles} we observe the relaxation of the normalized profiles $F_n/\Vert F_n\Vert_{L^1(\mathbb{R}^d)}$ towards the eigenfunction $\bF_\alpha$ along the time iterations $n=0,1,2,3,4,7,150$. We remark on the strong contraction of the variance during the first few iterations leading to a well-identified Gaussian-shaped profile at time $n=3$. At time $n=7$ the variance of the profile is already close to that of the equilibrium $\bF_\alpha$, but its mean is still substantially shifted to the right. Thereafter, we observe the gradual motion of the profiles towards the left at a much lower speed. After $150$ iterations, the approximate mean and variance of $F_n/\Vert F_n\Vert_{L^1(\mathbb{R}^d)}$ are given by $\approx 0.0474$ and $\approx 1.8897$ respectively, where the latter agrees with the above exact value $\bsigma_\alpha^2$ up to $5$ digits.

\begin{figure}
\centering
\includegraphics[width=0.55\textwidth]{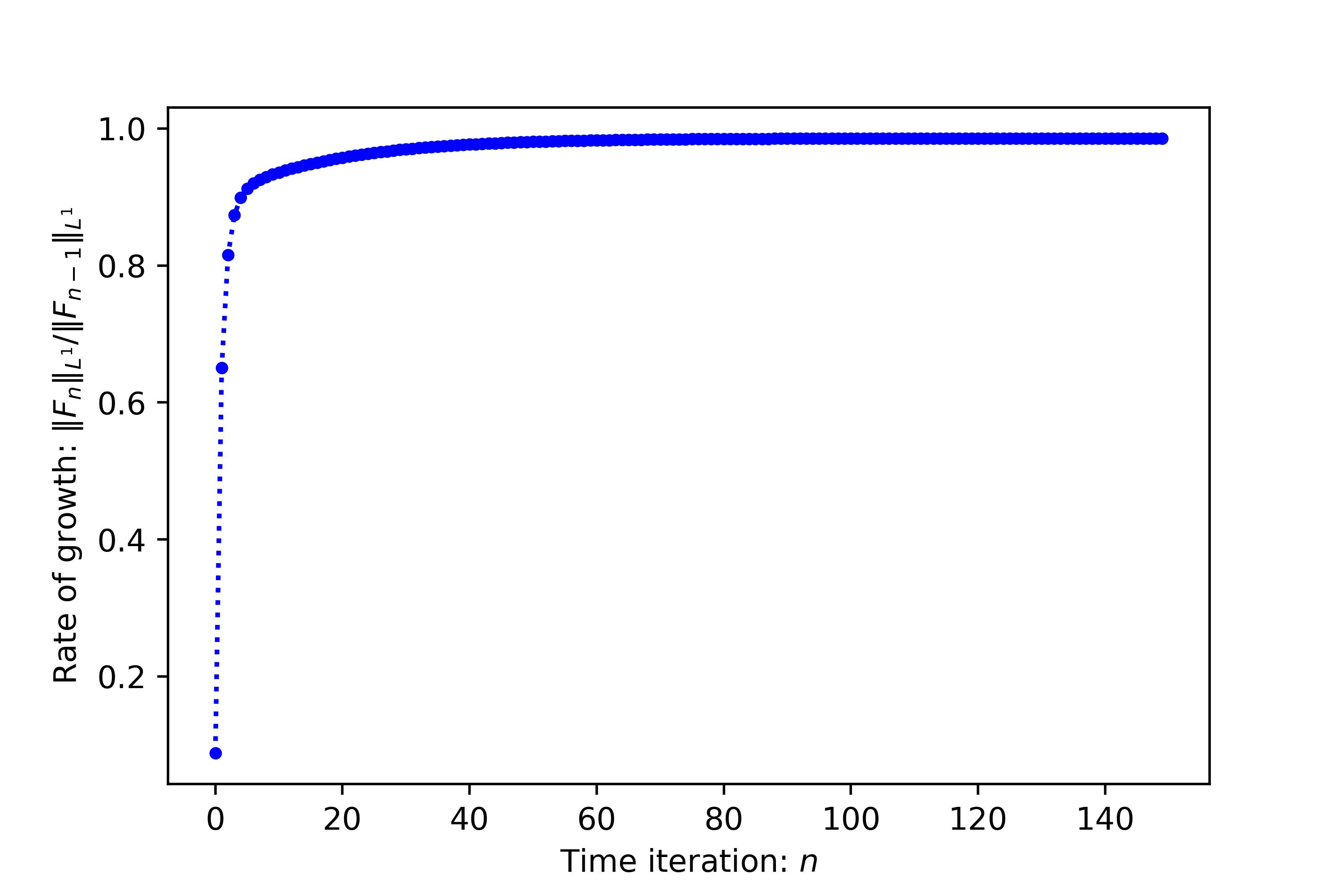}
\caption{Relaxation of the growth rates $\Vert F_n\Vert_{L^1(\mathbb{R}^d)}/\Vert F_{n-1}\Vert_{L^1(\mathbb{R}^d)}$ towards the eigenvalue $\blambda_\alpha$ along time iterations $0\leq n\leq 150$ for a step function \eqref{E-simulation-initial-datum} as initial datum and weak selection parameter $\alpha=0.015$.}
\label{fig:simulation-1-rates-growth}
\end{figure}

In Figure \ref{fig:simulation-1-rates-growth} we represent the convergence of the growth rates $\Vert F_n\Vert_{L^1(\mathbb{R}^d)}/\Vert F_{n-1}\Vert_{L^1(\mathbb{R}^d)}$ towards the eigenvalue $\blambda_\alpha$. After $150$ iterations we obtain that the growth rate takes the approximate value $\approx 0.9857$, which again agrees with the exact value $\blambda_\alpha$ above up to $5$ digits.

\begin{figure}
\centering
\begin{subfigure}[t]{0.49\textwidth}
\centering
\includegraphics[width=\textwidth]{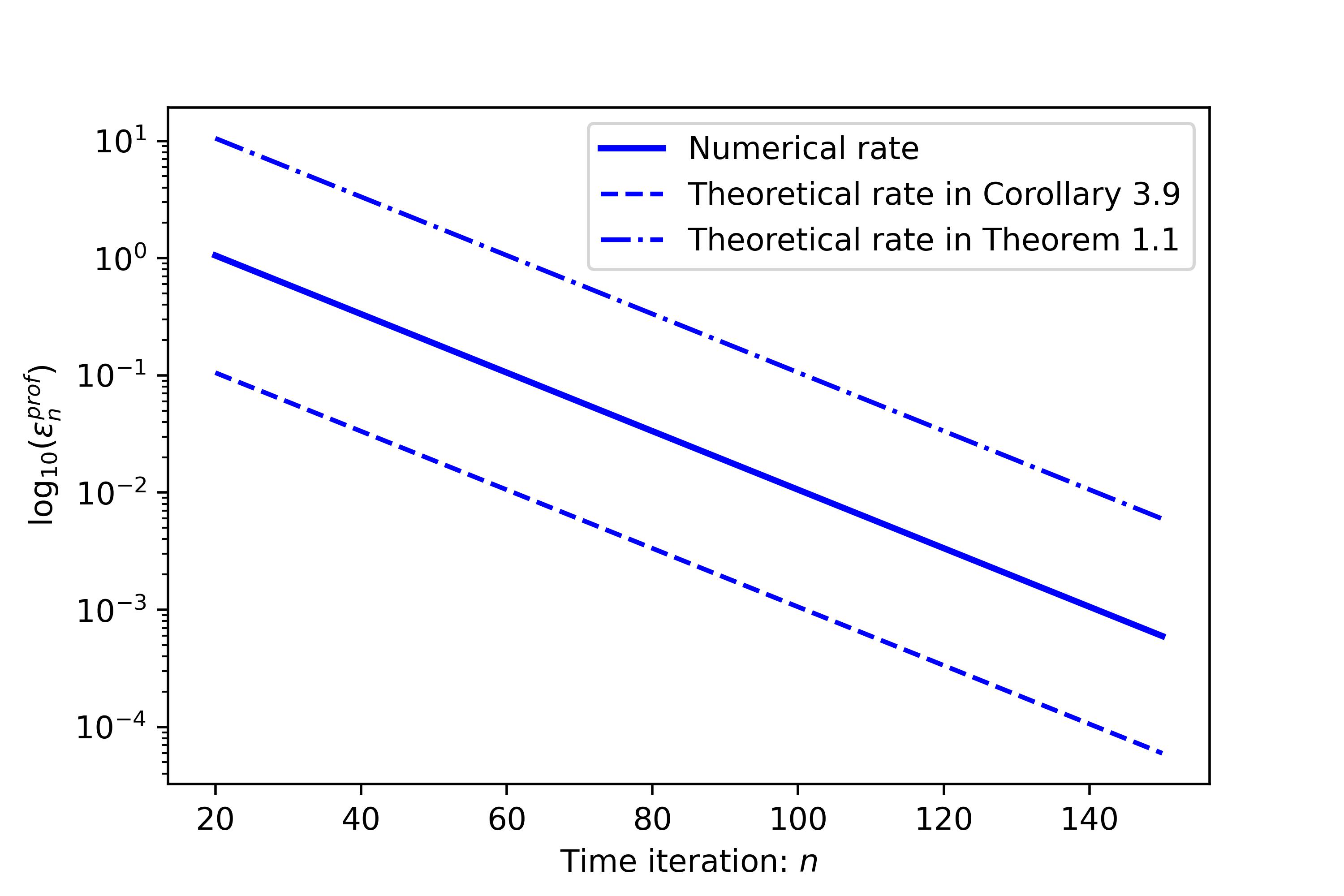}
\caption{ }
\label{fig:simulation-1-rates-profiles}
\end{subfigure}
\begin{subfigure}[t]{0.49\textwidth}
\centering
\includegraphics[width=\textwidth]{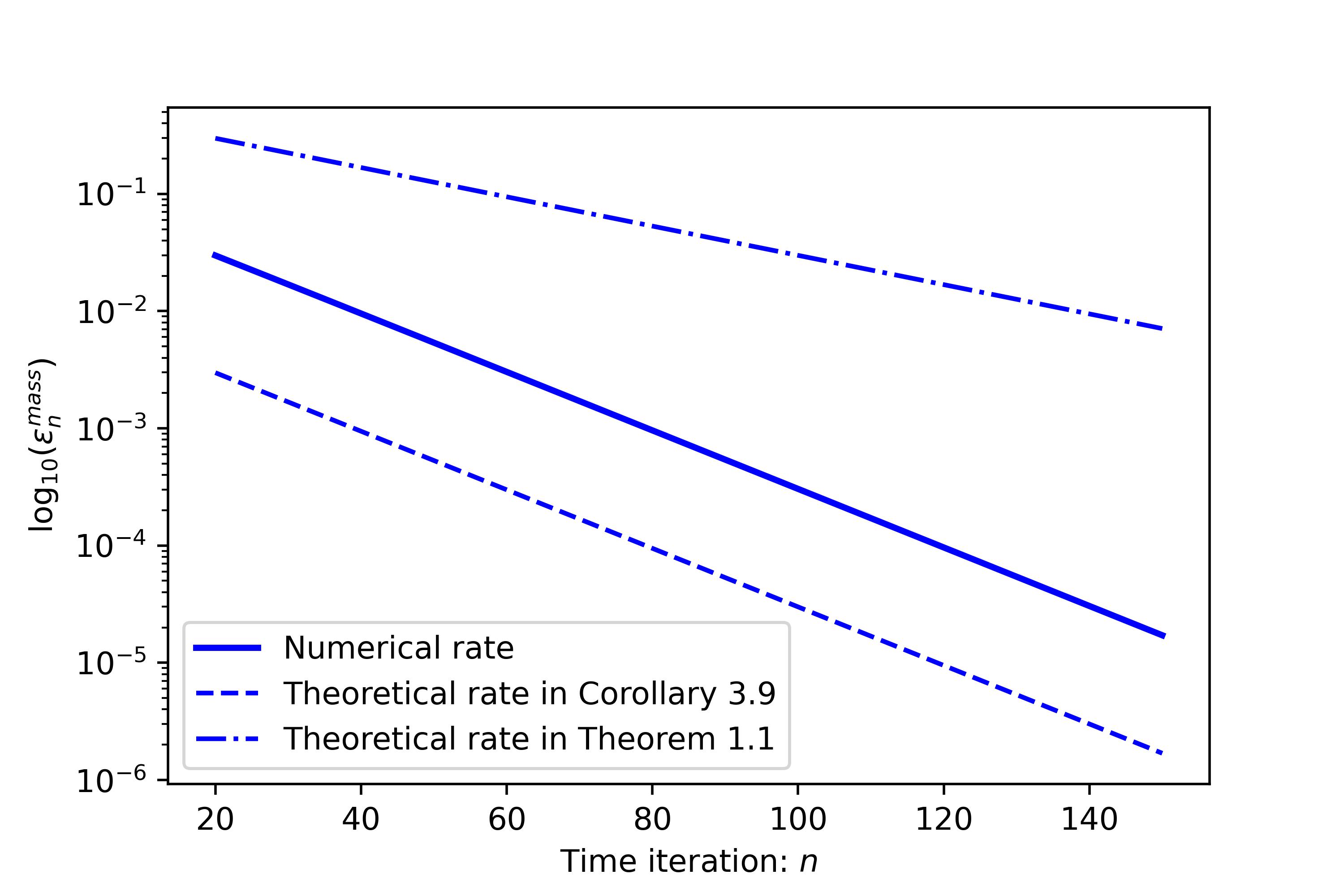}
\caption{ }
\label{fig:simulation-1-rates-mass}
\end{subfigure}
\caption{Numerical computation of the rates of convergence of  the normalized profiles and the growth rates: (a)  semi-log plot of the errors $\varepsilon_n^{\text{prof}}:=\mathcal{D}_{\rm KL}\left(\left.F_n/\Vert F_n\Vert_{L^1}\,\right\Vert \bF_\alpha\right)$, and (b) semi-log plot of the errors $\varepsilon_n^{\text{mass}}:=\left\vert \Vert F_n\Vert_{L^1}/\Vert F_{n-1}\Vert_{L^1}-\blambda_\alpha\right\vert$.}
\label{fig:simulation-1-rates}
\end{figure}

In Figure \ref{fig:simulation-1-rates} we have represented the errors
$$\varepsilon_n^{\text{prof}}:=\mathcal{D}_{\rm KL}\left(\left.\frac{F_n}{\Vert F_n\Vert_{L^1(\mathbb{R}^d)}}\,\right\Vert \bF_\alpha\right),\quad \varepsilon_n^{\text{mass}}:=\left\vert\frac{\Vert F_n\Vert_{L^1(\mathbb{R}^d)}}{\Vert F_{n-1}\Vert_{L^1(\mathbb{R}^d)}}-\blambda_\alpha\right\vert,$$
in semi-log plots so that the horizontal axis appears in the natural scale and represents each time iteration, and the vertical axis contains the logarithm of the errors. As observed, both plots reduce to essentially straight lines, suggesting exponential relaxation. Interestingly, the numerical rate of convergence in the Kullback-Leibler divergence measured in Figure \ref{fig:simulation-1-rates-profiles} is approximately $0.9441$. It coincides numerically with the rate of relaxation among the subclass of Gaussian solutions (see Corollary \ref{C-sexual-reproduction-Gaussian-solution-time-discrete-relaxation}), namely, $\blambda_\alpha^4\approx 0.9441$, which in turns is identical to the theoretical rate $(2\bk_\alpha)^2$ obtained in Theorem \ref{T-main}.
These results are in perfect agreement with the other convergence results in Figure \ref{fig:simulation-1-rates-mass}: the numerical rate of convergence in the rate of growth of mass is approximately $0.9442$, which is close to the one among the class of Gaussian solutions $\blambda_\alpha^4\approx 0.9441$, in contrast with the theoretical upper bound obtained in Theorem \ref{T-main}, namely $2\bk_\alpha\approx 0.9717$. Moreover, the numerical rate of convergence of the variance of the normalized profiles $F_n/\Vert F_n\Vert_{L^1(\mathbb{R}^d)}$ is much faster: it is approximately $\approx 0.4721<0.5$, again close to the expected value for Gaussian solutions, namely,
$\br_\alpha\approx 0.472$. These numerical results illustrate the two-step process arising in the relaxation dynamics: convergence towards a Gaussian profile occurs much faster than relaxation of the mean towards the origin due to (weak) selection. Alternatively speaking, the equilibrium variance builds up much faster than the center of the distribution gets to the origin.

\subsection{Strong selection} In this part, we consider a larger value of $\alpha$. We opted for $\alpha=0.4$. Indeed, if $\alpha$ is taken too large then there is no visual difference between the solution and the equilibrium configuration after one single iteration. 
This time, the numerical values of the features of the equilibrium are:
$$\blambda_\alpha\approx 0.7944,\quad \bsigma_\alpha^2\approx 0.9221.$$
The latter substantially differs from the value at linkage equilibrium ({\em i.e.}, $\bsigma_{\alpha=0}^2 = 2$). In Figure \ref{fig:simulation-2-profiles} we note that the solution resembles a Gaussian distribution after a couple of iterations. After $15$ iterations we obtain that the normalized profile $F_n/\Vert F_n\Vert_{L^1(\mathbb{R}^d)}$ has approximate mean and variance respectively given by $\approx 0.0017$ and $\approx 0.9221$. In Figure \ref{fig:simulation-2-mass} we observe the convergence of the growth rates $\Vert F_n\Vert_{L^1(\mathbb{R}^d)}/\Vert F_{n-1}\Vert_{L^1(\mathbb{R}^d)}$ towards the eigenvalue $\blambda_\alpha$. After $15$ iterations the growth rate becomes $\approx 0.7944$, which again agrees with the exact value $\blambda_\alpha$ above up to $5$ digits.

\begin{figure}
\centering
\begin{subfigure}[t]{0.49\textwidth}
\centering
\includegraphics[width=\textwidth]{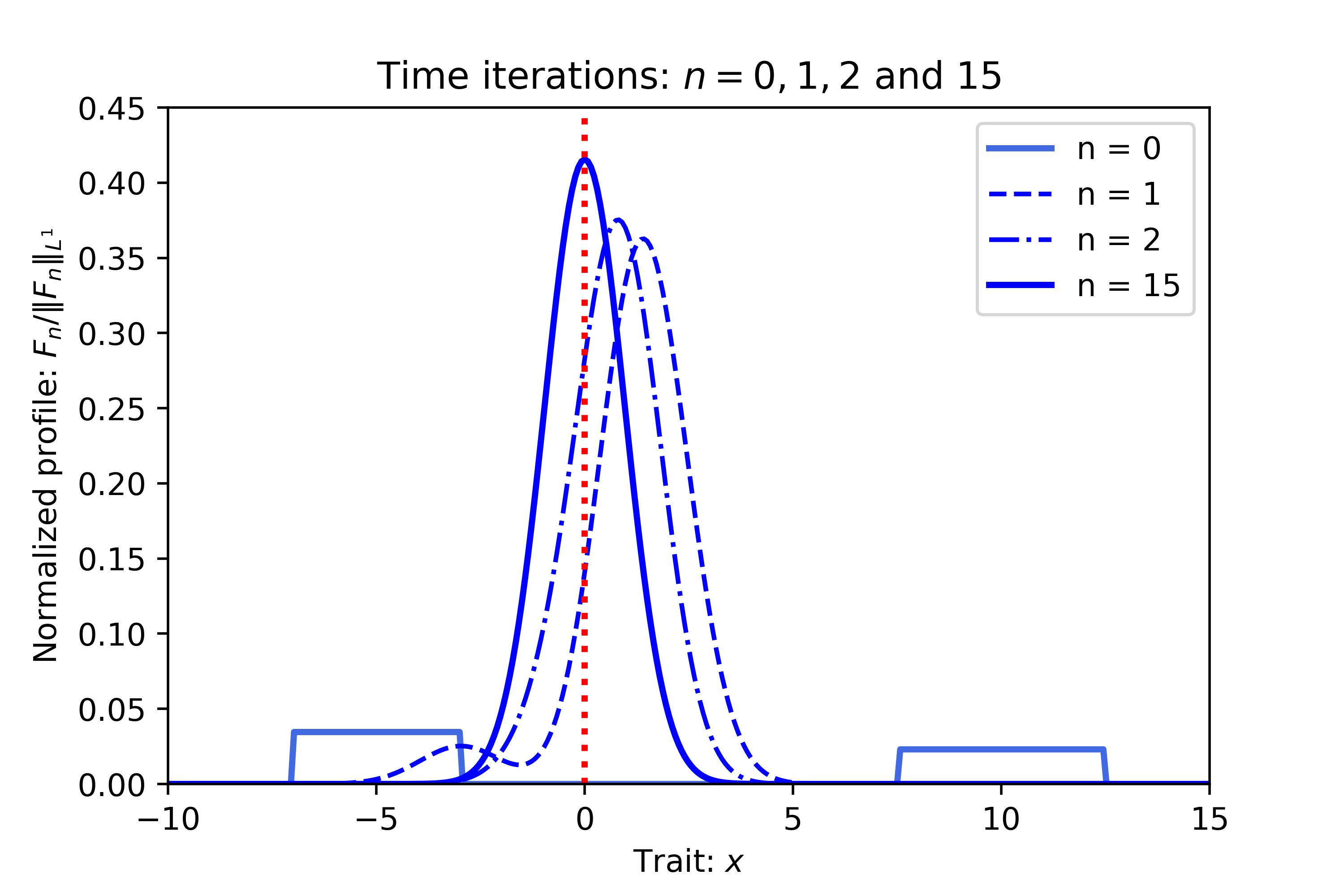}
\caption{ }
\label{fig:simulation-2-profiles}
\end{subfigure}
\begin{subfigure}[t]{0.49\textwidth}
\centering
\includegraphics[width=\textwidth]{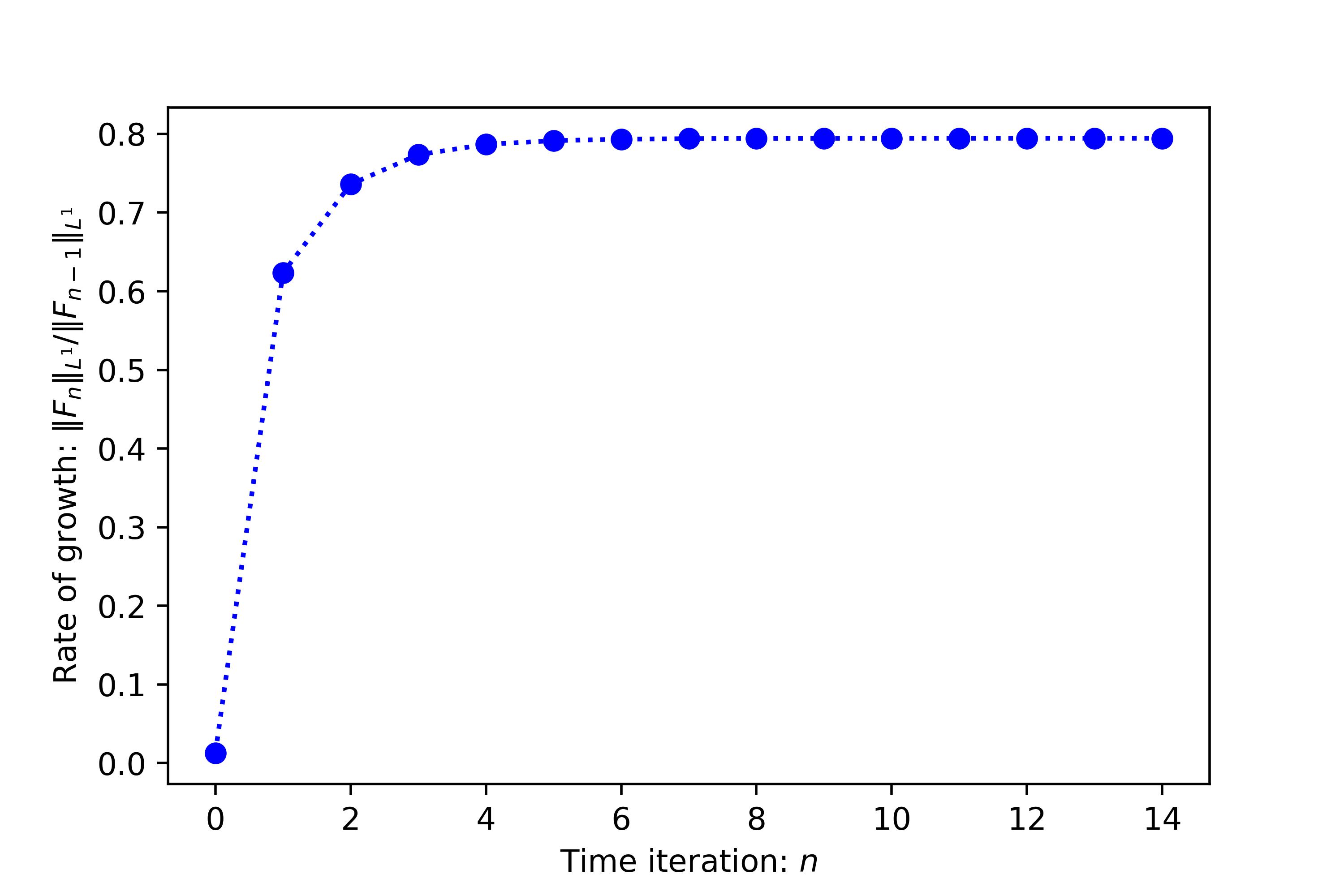}
\caption{ }
\label{fig:simulation-2-mass}
\end{subfigure}
\caption{(a) Zoom near the origin of the relaxation of the normalized profiles $F_n/\Vert F_n\Vert_{L^1(\mathbb{R}^d)}$ towards the eigenfunction $\bF_\alpha$ along the time iterations $n=0,1,2$ and $15$ for the step function \eqref{E-simulation-initial-datum} as initial datum and strong selection parameter $\alpha=0.4$. The vertical dotted line represents the location of the mean of the equilibrium profile $\bF_\alpha$. (b) Relaxation of the growth rates $\Vert F_n\Vert_{L^1(\mathbb{R}^d)}/\Vert F_{n-1}\Vert_{L^1(\mathbb{R}^d)}$ towards the eigenvalue $\blambda_\alpha$.}
\label{fig:simulation-2-profiles-mass}
\end{figure}

Similar computations as in the previous Figure \ref{fig:simulation-1-rates} allow finding numerical approximations for the rate of convergence of the normalized profiles and the growth rates. Specifically, we obtain an approximation $0.3966$ for the rate of convergence of the Kullback-Leibler divergence. Once again, such a value is close to the rate of relaxation among the subclass of Gaussian solutions, namely, $\blambda_\alpha^4\approx 0.3983$, which in turns agrees with the theoretical rate $(2\bk_\alpha)^2$ obtained in Theorem \ref{T-main}. Similarly, the numerical rate of convergence of the growth rates is approximately $0.3901$, which is close to the one among the class of Gaussian solutions $\blambda_\alpha^4\approx 0.3983$, in contrast with the upper bound $2\bk_\alpha\approx 0.6311$ in Theorem \ref{T-main}.

\section{Conclusions and perspectives}\label{S-conclusions}

In this paper, we have proven asynchronous exponential growth in a quantitative genetics model for the evolution of the distribution of traits in a population governed by sexual reproduction and multiplicative effect of selection. Our model assumes time-discrete non-overlapping generations, which rule out an eventual mixing with previous generations of ancestors. In addition, our non-linear sexual reproduction operator is set in agreement with Fisher's infinitesimal model, and we have chosen selection to act on the survival probability of individuals. Our main result provides quantitative convergence rates of the renormalized distributions towards a unique stationary profile. Indeed, rates are exponential, which can be interpreted loosely as a spectral gap in this non-linear context.

It is noticeable that the sexual reproduction operator is contractive under the Wasserstein distance (see Lemma \ref{lem:contraction W2}).
However, the generic incompatibility of multiplicative selection with transport distances  becomes an apparent obstruction to the use of a direct perturbative approach in the regime of weak selection (see Example \ref{ex:incompatibility}). This obstacle was circumvented by {\sc G. Raoul} \cite{R-21-arxiv}, assuming that selection is restricted to a compact support. We follow a different route in a purely non-perturbative setting. To this end, we restrict to the specific choice of a quadratic selection  function of the trait, for the sake of simplicity. Our alternative approach relies on an appropriate study of the propagation of information along large binary trees of ancestors, by a suitable reformulation of high-dimensional integrals, inspired by the changes of variables performed in \cite{CGP-19} in the regime of small variance. This reveals an ergodicity property where the exact shape of the initial distribution is quickly forgotten across generations. 

We remark that the above heuristic arguments indicate that the quadratic Wasserstein metric is not fully appropriate for dealing with the problem at hand. However, we could not identify yet a proper metric that extends the contraction property of the neutral case to the quadratic selection case. 

Several perspectives are envisaged. First, there is an apparent price to pay with our method, in which we fully exploit the Gaussian structure induced by quadratic selection in order to perform tractable computations within the high-dimensional integrals. We believe though that the restriction to quadratic selection might be overcome in future works to allow for more general selection functions.  Second, as studied in \cite{CGP-19}, the case of multiple minima on the selection function leads to non-uniqueness of stable equilibria. Then, uncovering the hidden metastability and quantifying the relaxation towards a specific equilibrium is of great interest for its applications in quantitative genetics. Finally, a major problem is to transcend non-overlapping generations and tackle the full time-continuous model as presented in previous literature.

%%%%%%%%%%
\appendix

\section{Nondimensionalization and derivation of the time-discrete version}\label{Appendix-nondimensionalization}

In this section, we shall nondimensionalize the time-continuous model \eqref{E-general-reproduction}-\eqref{E-sexual-reproduction-operator} with Gaussian mixing kernel $G$ and quadratic selection function $m$. By time discretization on the Duhamel formulation, we will derive the time-discrete version \eqref{E-sexual-reproduction-time-discrete} which has been central in this paper. Indeed, we shall show that we can reduce parameters into only one, namely, parameter $\alpha\in \mathbb{R}_+$ in the quadratic selection function \eqref{E-selection-quadratic}. Bearing in mind all the biological parameters and dimensions, our evolution problem reads:
\begin{equation}\label{E-sexual-reproduction-time-continuous-all-parameters}
\left\{\begin{array}{ll}
\partial_t f =-r_m\,m(x)f+r_b\,\mathcal{B}_{\sigma^2}[f], & t\geq 0,\,x\in \mathbb{R}^d\\
f(0,x)=f_0(x), & x\in \mathbb{R}^d.
\end{array}\right.
\end{equation}
Here, $r_m,r_b\in \mathbb{R}_+$ represent the mortality and birth rates and have frequency units, whilst the dimensionless corrections of $m$ and $\mathcal{B}$ take the form
$$
m(x)=\frac{\vert x\vert^2}{\sigma_m^2},\qquad
\mathcal{B}_{\sigma^2}[f](t,x)=\int_{\mathbb{R}^{2d}}G_{0,\sigma^2}\left(x-\frac{x_1+x_2}{2}\right)\frac{f(t,x_1)f(t,x_2)}{\int_{\mathbb{R}^d} f(t,x')\,dx'}\,dx_1\,dx_2,$$
for each $x\in \mathbb{R}^d$. Here, $G_{0,\sigma^2}$ represents the Gaussian centered at $0$ with variance $\sigma^2$, and  $\sigma_m$ and $\sigma$ have the same units as the quantitative trait. On the one hand, $\sigma_m$ can be regarded as a characteristic unit quantifying the effective range of selection. Namely, if $\vert x\vert \geq \sigma_m$ then $m(x)\geq 1$. On the other hand, $\sigma^2$ is the genetic variance and $2\sigma^2$ represents the variance at linkage equilibrium often denoted by $\sigma_{LE}^2:=2\sigma^2$ ({\em cf.} \cite{B-80}). We nondimensionalize our system by appropriately scaling our variables as follows 
$$\widehat{t}=\frac{t}{T},\quad \widehat{x}=\frac{x}{L},\quad \widehat{f}(\widehat{t},\widehat{x})=L^d f(t,x).$$
Here, $T$ and $L$ are characteristic units for time and trait. For simplicity, we set them as follows in terms of the parameters of the system:
$$T:=\frac{1}{r_m},\quad L:=\sigma.$$
In other words, we scale time according to the mortality rate $r_m$ and the trait variable according to the genetic variance $\sigma^2$. Dropping hats for simplicity yields the following dimensionless form of the equation
\begin{equation}\label{E-sexual-reproduction-time-continuous-reduced-parameters}
\left\{\begin{array}{ll}
\partial_t f=-\frac{\alpha}{2}\vert x\vert^2 f+\beta \mathcal{B}[f], & t\geq 0,\,x\in \mathbb{R}^d,\\
f(0,x)=f_0(x), & x\in \mathbb{R}^d,
\end{array}\right.
\end{equation}
where $\mathcal{B}=\mathcal{B}_1$ involves again unit genetic variance thanks to our scaling assumption. This reduces the amount of free parameters to only two of them $\alpha,\beta\in \mathbb{R}_+$ defined as follows:
$$\alpha:=\frac{\sigma^2_{LE}}{\sigma_m^2},\quad \beta:=\frac{r_b}{r_m}.$$

In the sequel, we justify our time-discrete version \eqref{E-eigenproblem-time-discrete}. Specifically, let us integrate equation $\eqref{E-sexual-reproduction-time-continuous-reduced-parameters}_1$ for $t\in[t_n,t_{n+1}]$ where $\{t_n\}_{n\in \mathbb{N}}\subseteq \mathbb{R}_+$ is any increasing sequence. Then, we recover Duhamel's formula
$$f(t_n,x)=e^{-\frac{\alpha}{2}\vert x\vert^2(t_n-t_{n-1})}f(t_{n-1},x)+\beta\int_{t_{n-1}}^{t_n} e^{-\frac{\alpha}{2}\vert x\vert^2(t_n-s)}\mathcal{B}[f(s,\cdot)](x)\,ds,$$
for any $n\in \mathbb{N}$ and each $x\in \mathbb{R}^d$. Using the rectangle rule for the integral in the right hand side leads to the approximations $F_n(x)\simeq f(t_n,x)$ of the the above time-continuous problem for $f=f(t,x)$:
\begin{equation}\label{E-sexual-reproduction-time-discrete-overlapping-generations}
F_n(x)=e^{-\frac{\alpha}{2}\vert x\vert^2}F_{n-1}(x)+\beta e^{-\frac{\alpha}{2}\vert x\vert^2}\mathcal{B}[F_{n-1}](x),
\end{equation}
for any $n\in \mathbb{N}$ and each $x\in \mathbb{R}^d$. Here, we have set unitary time steps $t_n-t_{n-1}=1$ for simplicity and $t_0=0$. We emphasize that the right hand side of \eqref{E-sexual-reproduction-time-discrete-overlapping-generations} consists of two different terms. On the one hand, the first term describes the amount of old individuals from the previously generation $n-1$ having resisted the effect of selection in generation $n$. On the other hand, the second term computes the amount of offspring conceived by individuals in the previous generation $n-1$, where an eventual decline due to selection has also been taken into account. Let us emphasize that in \eqref{E-sexual-reproduction-time-discrete-overlapping-generations} individuals from generation $n-1$ can merge (after breeding) with their offspring to form next generation $n$. This is typical in most biological populations, which have \textit{overlapping generations} ({\em e.g.}, some multivoltine flies like Drosophila melanogaster). However, some other systems have {\it non-overlapping generations} so that the full generation $n-1$ gets extinct after breeding and only their offspring survive at generation $n$ ({\em e.g.}, some univoltine insects like the Dawson's burrowing bee). We refer to \cite{M-61} for a discussion on the various types of voltinism and the role of diapause. In that case, by canceling the first term \eqref{E-sexual-reproduction-time-discrete-overlapping-generations} reduces to
\begin{equation}\label{E-sexual-reproduction-time-discrete-two-parameters}
F_n(x)=\beta e^{-\frac{\alpha}{2}\vert x\vert^2}\mathcal{B}[F_{n-1}](x),
\end{equation}
for any $n\in \mathbb{N}$ and each $x\in \mathbb{R}$. We remark that two different parameters  $\alpha$ and $\beta$ are still present in \eqref{E-sexual-reproduction-time-discrete-two-parameters}. However, given that $F\mapsto e^{-\frac{\alpha}{2}\vert x\vert^2}\mathcal{B}[F]$ is a 1-homogeneous operator we can kill the parameter $\beta$ by rescaling the trait distributions $F_n$. Specifically, define $\widehat{F}_n=\beta^{-n}F_n$ and note that \eqref{E-sexual-reproduction-time-discrete-two-parameters} reduces to our time-discrete problem \eqref{E-sexual-reproduction-time-discrete} with only one parameter $\alpha$.

\bibliographystyle{amsplain} %Options: amsplain, elsarticle-harv
\bibliography{biblio}

\end{document}